\documentclass[a4paper]{amsart}

\usepackage{amsfonts,amssymb,mathrsfs,bbold,circledsteps}
\usepackage{stmaryrd,adjustbox,circuitikz}
\usepackage[a4paper,hmargin=26mm,vmargin=28mm]{geometry}

\usepackage{enumitem}
\setlist[enumerate]{label=\upshape(\alph*)., ref=\alph*}

\usepackage[svgnames]{xcolor}
\usepackage{booktabs,upgreek,tcolorbox,caption}
\usepackage{xparse,mathtools,etoolbox,bbm,bm}
\usepackage{tikz}
\usepackage{tikz-cd}
\usetikzlibrary{arrows,decorations.markings,calc,backgrounds}

\usepackage{todonotes}
\usepackage{dynkin-diagrams}
\usepackage{atableau}

\usepackage{cite}

\usepackage[pagebackref=false]{hyperref}
\makeatletter
\hypersetup{%
  pdftitle={\@title},%
  pdfauthor={Tao Qin},%
  colorlinks=true,%
  linkcolor=blue,%
  anchorcolor=red,%
  citecolor=blue,%
  urlcolor=blue%
}
\makeatother

\newcommand\enumref[2]{\hyperref[#2]{\autoref*{#1}\textup{(\ref*{#2})}}}
\title{Subdivision and Runner Removal Theorems}
\author{Tao Qin}

\subjclass[2020]{Primary 20C08; Secondary 05E10, 16G20, 17B37}
\keywords{Cyclotomic KLR algebras, KLRW algebras, subdivision map, runner removal, Specht modules, canonical bases, categorification}
\let\<=\langle
\let\>=\rangle

\newcommand{\Sym}{\mathfrak S}
\DeclareMathOperator{\Gar}{Gar}
\newcommand{\Z}{\mathbb{Z}}

\newcommand{\Par}[1][\Lambda]{\mathscr{P}^{#1}}

\newcommand\bT{{\mathscr{T}}}
\newcommand\blam{{\boldsymbol\lambda}}
\newcommand\bmu{{\boldsymbol\mu}}
\newcommand\bnu{{\boldsymbol\nu}}

\newcommand{\bms}{\bm{s}}
\newcommand{\bmr}{\bm{r}}
\NewDocumentCommand{\subdatum}{ O{\Lambda} O{\bkappa} O{k} }{
    (e,I,#1,\alpha,#2,#3)
}
\NewDocumentCommand{\absubdatum}{ O{a} O{c} O{d} O{a'} }{%
    (#1,#2,#3,#4)%
}

\newcommand\bk{\mathbb{k}}
\newcommand\bbd{\boldsymbol{d}}
\newcommand\height{\text{ht}}
\newcommand\oalpha{\overline{\alpha}}
\newcommand\okappa{\overline{\kappa}}
\newcommand\oLambda{\overline{\Lambda}}
\newcommand\olambda{\overline{\lambda}}
\newcommand\oGamma{\overline{\Gamma}}

\newcommand{\unorder}[1][\oalpha]{\widebar{I}^{#1}_{\text{un}}}
\newcommand{\wellorder}[1][\oalpha]{\widebar{I}^{#1}_{\text{ord}}}
\newcommand{\almostorder}[1][\oalpha]{\widebar{I}^{#1}_{\text{al}}}

\newcommand\badideal{\mathfrak{J}}

\NewDocumentCommand{\Aba}{O{a} O{e}}{\mathrm{Ab}^{#1}_{#2}}

\DeclarePairedDelimiterX{\set}[1]{\{}{\}}{\setargs{#1}}
\NewDocumentCommand{\setargs}{>{\SplitArgument{1}{|}}m}{\setargsaux#1}
\NewDocumentCommand{\setargsaux}{mm}
{\IfNoValueTF{#2}{#1}{#1\,\delimsize|\,\mathopen{}#2}}

\renewcommand{\pmod}[1]{\text{ }(\textrm{mod } #1)\,}

\usepackage{aliascnt}

\newtheoremstyle{noparens}
  {}          
  {}          
  {\itshape}  
  {}          
  {\bfseries} 
  {.}         
  {.5em}      
  {\thmname{#1}~\thmnumber{#2}\thmnote{ \textnormal{#3}}} 

\def\NewTheorem#1{%
  \newaliascnt{#1}{equation}%
  \newtheorem{#1}[#1]{#1}%
  \aliascntresetthe{#1}%
  \expandafter\def\csname #1autorefname\endcsname{#1}%
}

\def\equationautorefname~#1\null{(#1)\null}
\def\itemautorefname~#1\null{(#1)\null}

\numberwithin{equation}{section}

\theoremstyle{noparens}
\NewTheorem{Proposition}
\NewTheorem{Theorem}
\NewTheorem{Conjecture}
\NewTheorem{Corollary}
\NewTheorem{Lemma}
\NewTheorem{Question}

\NewTheorem{Definition}
\theoremstyle{definition}
\NewTheorem{Notation}
\NewTheorem{Example}
\AtEndEnvironment{Example}{\null\hfill$\Diamond$}%
\NewTheorem{Examples}
\AtEndEnvironment{Examples}{\null\hfill$\Diamond$}%
\theoremstyle{remark}
\NewTheorem{Remark}
\AtEndEnvironment{Remark}{\null\hfill$\Diamond$}%
\NewTheorem{Claim}

\makeatletter
\newcommand*\rel@kern[1]{\kern#1\dimexpr\macc@kerna}
\newcommand*\widebar[1]{%
  \begingroup
  \def\mathaccent##1##2{%
    \rel@kern{0.8}%
    \overline{\rel@kern{-0.8}\macc@nucleus\rel@kern{0.2}}%
    \rel@kern{-0.2}%
  }%
  \macc@depth\@ne
  \let\math@bgroup\@empty \let\math@egroup\macc@set@skewchar
  \mathsurround\z@ \frozen@everymath{\mathgroup\macc@group\relax}%
  \macc@set@skewchar\relax
  \let\mathaccentV\macc@nested@a
  \macc@nested@a\relax111{#1}%
  \endgroup
}

\newcommand\bi{\mathbf{i}}
\newcommand\bj{\mathbf{j}}
\newcommand{\bkappa}{\boldsymbol{\kappa}}\newcommand\be{\mathbbm{e}}
\newcommand{\defe}{\operatorname{def}}
\newcommand{\DeclareMyOperator}[1]{%
  \expandafter\DeclareMathOperator\csname #1\endcsname{#1}
}
\forcsvlist{\DeclareMyOperator}{%
  Mat,Sh,Shaded,cont,Add,
  ch,diag,End,END,head,Hom,HOM,im,inv,Irr,Deg,Ind,Rad,Rem,rad,Res,res,
  soc,supp,Shape,Std,ST,CST,Top
}

\newcommand{\RST}[1]{\operatorname{RStd}(#1)}

\newcommand\ShT[1][l]{\mathop{ShT}}

\DeclareMathOperator\arrow{\text{---}}
\let\kill\relax
\DeclarePairedDelimiterX\kill[2]{(}{)}{#1|#2}

\newcommand\Ainfty{A_{\infty}}
\newcommand\Aone[1][e+1]{A^{(1)}_{#1}}
\newcommand\Cone[1][e]{C^{(1)}_{#1}}
\newcommand\Atwo[1][2e]{A^{(2)}_{#1}}
\newcommand\Dtwo[1][e+1]{D^{(2)}_{#1}}

\newcommand\h{\mathfrak{h}}

\renewcommand{\phi}{\varphi}

\tikzset{
  centered/.style = {
     baseline = {([yshift=#1]current bounding box.center)}
  },
  centered/.default={-0.5ex},
  ->-/.style={
    decoration={
      markings,
      mark=at position 0.6 with {\arrow[thin,scale=2]{>}}
    },
    postaction={decorate}
  },
  -<-/.style={
    decoration={
      markings,
      mark=at position 0.6 with {\arrow[thin,scale=2]{<}}
    },
    postaction={decorate}
  },
  circled/.style = {fill=Tan},
  domstyle/.style = {
    thick,
    scale=#1,
  },
  vertex/.style = {
    circle,
    ball color=MidnightBlue,
    font=\small,
    inner sep=0pt,
    minimum size=2mm
  },
  pics/domm/.style = {
    code = {
      \draw[domstyle=#1](0.8,0)--++(0.6,0.4)--++(0.6,-0.4)--++(-0.6,-0.4)--++(-0.6,0.4);
    }
  },
  pics/dom/.style = {
    code = {
      \draw[domstyle=#1](0.2,0)--(0.8,0)
        --++(0.6,0.4)--++(0.6,-0.4)--++(-0.6,-0.4)--++(-0.6,0.4);
    }
  },
  pics/domeq/.style = {
    code = {
      \pic at (0,0){dom=#1};
      \draw[scale=#1](0.8,-0.5)--++(1.2,0);
    }
  },
  pics/domneq/.style = {
    code = {
      \pic at (0,0){domeq=#1};
      \draw[scale=#1](0.9,-0.75)--++(1.0,0.5);
    }
  },
  pics/mod/.style = {
    code = {
      \draw[domstyle=#1](0,0)--++(-0.6,0)--++(-0.6,-0.4)--++(-0.6,0.4)--++(0.6,0.4)--++(0.6,-0.4);
    }
  },
  pics/Sdom/.style = {
    code = {
      \draw[domstyle=#1](0.2,0)--(0.8,0);
      \draw[scale=#1,fill=black](0.8,0)--++(0.6,0.4)--++(0.6,-0.4)--++(-0.6,-0.4)--++(-0.6,0.4);
    }
  },
  pics/Sdomeq/.style = {
    code = {
      \pic at (0,0){Sdom=#1};
      \draw[scale=#1](0.8,-0.5)--++(1.2,0);
    }
  },
  pics/Sdomneq/.style = {
    code = {
      \pic at (0,0){Sdomeq=#1};
      \draw[scale=#1](0.9,-0.75)--++(1.0,0.5);
    }
  },
}

\let\realItem\item
\NewDocumentCommand\centeredItem{so}{%
   \IfNoValueTF{#2}{\realItem}{\realItem[#2]}%
   \IfBooleanF{#1}{\hfil}%
}
\newlist{Enumerate}{enumerate}{1}
\setlist[Enumerate]{%
  label=\textup{\alph*)},
  ref=\theequation\alph*,
}
\newlist{relations}{enumerate}{1}
\setlist[relations]{%
  label=$(\text{KLR}_{\arabic*})$,
  before=\let\item\centeredItem
}

\begin{document}
\begin{abstract}
    We develop a combinatorial framework for the subdivision map—introduced by Maksimau, Mathas and Tubbenhauer—between the KLR(W) algebras of types $\Aone[e-1]$ and $\Aone[e]$, which provides a partial categorification of the runner removal theorems.
\end{abstract}
\maketitle
\setcounter{tocdepth}{1}
\tableofcontents

\section{Introduction}
Khovanov, Lauda, and Rouquier introduced KLR algebras to categorify quantum groups; see \cite{khovanovlauda-klr-1,khovanovlauda-klr-3,rouquier-2kacmoody}. Webster \cite{webster-rouquier-conjecture-diag-algebra,webster-weighted-klr} subsequently introduced KLRW algebras (weighted KLR algebras) as a generalization of KLR algebras.

There are many well-known results concerning KLR and KLRW algebras. Notably, the Brundan–Kleshchev isomorphism \cite{bk-blocks-iso} relates Ariki–Koike algebras and cyclotomic KLR algebras of type $\Aone[e-1]$. Kang and Kashiwara \cite{kangkashiwara-klr-categorification} proved that cyclotomic KLR algebras of any symmetrizable type categorify the highest weight modules of the corresponding quantum group. Additionally, Hu and Mathas \cite{humathas-graded-cellular} constructed a graded cellular basis for cyclotomic KLR algebras of type $\Aone[e-1]$, and showed that the cell modules are isomorphic to the graded Specht modules constructed by Brundan, Kleshchev, and Wang \cite{bkw-graded-specht}. Later, Evseev and Mathas \cite{evseevmathas-klr-deformation} introduced a deformation method to construct cellular bases for KLR algebras of types $\Aone[e-1]$ and $\Cone[e]$. It is also worth mentioning that Bowman \cite{bowman-many-cellular} used idempotent truncation of KLRW algebras to construct a family of diagrammatic cellular bases for KLR algebras in type $\Aone[e-1]$. Mathas and Tubbenhauer further investigated the (sandwiched) cellularity of KLRW algebras in types $\Cone[e]$, $\Atwo[2e]$, and $\Dtwo[e+1]$ \cite{mathastubbenhauer-klrw-ac,mathastubbenhauer-klrw-bad}, as well as in finite types \cite{mathastubbenhauer-klrw-finite}.

The subdivision map for KLR algebras of type $\Aone[e-1]$ was introduced by Maksimau in \cite{maksimau-subdivision} to relate categorical representations of $\widehat{\mathfrak{sl}_e}$ and $\widehat{\mathfrak{sl}_{e+1}}$. Mathas and Tubbenhauer \cite{mathastubbenhauer-klrw-ac} subsequently extended this construction to KLRW algebras, thereby generalizing Maksimau's result. In \cite{qin-subdivision-klrw}, we attempted to generalize this map to cyclotomic KLRW algebras. In this paper, we further extend the combinatorial framework of \cite{qin-subdivision-klrw} to full generality.

More precisely, KLR algebras have three types of generators: idempotents $e(\bi)$, polynomial generators (or dots) $y_i$, and permutation generators $\psi_i$; see \autoref{subsec:klr-algebras}. The subdivision map for KLR algebras is a map between a KLR algebra of type $\Aone[e-1]$ and a quotient of an idempotent truncation of a KLR algebra of type $\Aone[e]$; see \autoref{subsec:subdivision-KLR}.

The first goal in understanding this map is to describe the images of its generators. In \cite{qin-subdivision-klrw}, a subdivision map on partitions was introduced (in a restricted specialization) in order to describe the image of idempotents under the subdivision map $\Phi_k$ for KLRW algebras. This map admits two equivalent combinatorial descriptions: via Young diagrams and via James' abaci.

In the present work, we streamline both descriptions by distinguishing two kinds of strips in the Young-diagram model (see \autoref{subsec:subdivision-young-diagram}) and by using the standard finite abaci in the abacus model (see \autoref{subsec:subdivision-abacus}). We prove that these simplified definitions are equivalent (see \autoref{thm:equivalence-of-two-definitions-subdivision}). Moreover, we extend the construction from partitions to row-standard tableaux (see \autoref{subsec:subdivision-standard-tableaux}) and prove degree preservation for standard tableaux \autoref{thm:degree-invariance-standard-tableau}. This tableau-level refinement is used later to describe the images of the basis vectors $\psi^T$ of permutation modules, as in the proof of \autoref{thm:image-of-permutation-module}.

Any cyclotomic KLR algebra has a distinguished family of modules called Specht modules. There is a subfamily of Specht modules whose heads give all irreducible modules of the corresponding cyclotomic KLR algebra; see \cite{humathas-graded-cellular} for more details. One motivation for this project is that the subdivision map for KLR algebras may relate the Specht modules of two KLR algebras and hence provide information about the corresponding graded decomposition numbers. In this paper, we show that, under a certain condition (see \autoref{def:k-horizontal}), the Specht modules are related in the most natural way; see \autoref{thm:image-of-specht-module}. The proof relies on the highest-weight presentation of Specht modules from \cite{kmr-universal-specht-type-A}, and, as an intermediate step, we also prove that the permutation modules are related naturally; see \autoref{thm:image-of-permutation-module} and \autoref{cor:image-of-permutation-module}.

Mathas suggested that subdivision for KLR(W) algebras should be connected to the runner removal theorems of James and Mathas \cite[Theorem 4.5]{jamesmathas-empty-runner-removal} and of Fayers \cite[Theorem 4.1]{fayers-full-runner-removal}, \cite[Theorem 3.4]{fayers-general-runner-removal} for Hecke and $q$-Schur algebras. We confirm this connection in this paper by showing that the abacus combinatorics used to construct the subdivision map on partitions can be interpreted as runner addition (see \autoref{subsec:subdivision-abacus}), in agreement with the constructions in \cite{fayers-full-runner-removal,alice-full-runner-removal,alice-empty-runner-removal}; see \autoref{sec:categorification}.

Recently, Dell’Arciprete and Putignano extended the empty runner removal theorem \cite{jamesmathas-empty-runner-removal} and the full runner removal theorem \cite{fayers-full-runner-removal} to higher levels, i.e.\ to Ariki--Koike algebras; see \cite{alice-full-runner-removal,alice-empty-runner-removal}. However, the general runner removal theorem of \cite{fayers-general-runner-removal} has not yet been generalized. At the end of this paper, we formulate some natural conjectures extending this result and present supporting evidence; see \autoref{conj:general-runner-removal-high-level}, \autoref{conj:general-runner-removal-level-two}, and \autoref{eg:general-runner-removal-level-2}.

In view of the categorical properties of the subdivision maps, whose combinatorial construction agrees with that of the runner removal theorems, subdivision can be viewed as a categorification of the latter. Moreover, the proofs of the runner removal theorems above are based on calculations in the Fock space and therefore work only in characteristic zero, and they do not readily extend to positive characteristic. The advantage of the subdivision approach is that it works in positive characteristic, as in \cite{chuangmiyachi-runner-removal}. However, in this paper we do not address conditions under which subdivision is exact as a functor on module categories, which is a key ingredient for extending \cite{chuangmiyachi-runner-removal}. We expect that this is related to \cite{mt-face-functors-klr} and leave it for future work.

The paper is organized as follows. In \autoref{sec:preliminary} we recall the Lie-theoretic setup and the combinatorics used throughout, and in \autoref{sec:klr-specht} we review the KLR algebras of type $\Aone[e-1]$ together with the universal graded Specht modules of \cite{kmr-universal-specht-type-A}; readers familiar with these foundations may safely skip both sections. In \autoref{sec:algebraic-subdivision} we give a self-contained definition of the subdivision map $\Phi_k$ on the affine quiver and the associated data (positive roots, words, dominant weights, KLR algebras, and their cyclotomic quotients), providing a base and extended version of \cite{maksimau-subdivision}. In \autoref{sec:combinatorial-subdivision} we develop the corresponding subdivision on partitions in two parallel languages—Young diagrams and $e$-abaci—and prove that the two definitions agree. Furthermore, we extend the map to the set of row-standard tableaux. In \autoref{sec:categorical-subdivisions} we translate these combinatorial results back to the algebraic setting by analyzing the images of idempotents under $\Phi_k$. Moreover, we relate permutation modules and (universal) Specht modules in the expected way (\autoref{subsec:subdivision-permutation-modules} and \autoref{subsec:subdivision-specht-modules}), giving our first categorical consequences, while leaving exactness issues for future work. Finally, in \autoref{sec:categorification} we connect subdivision with the runner removal theorems, treating both the level-one and higher-level cases, and we end with two conjectures on equalities of decomposition numbers in higher level; an appendix records explicit canonical-basis expansions that provide computational evidence.
\section{Lie Data and Combinatorics}\label{sec:preliminary}
\subsection{Cartan data}\label{subsec:Cartan-data}
Let $e > 2$ be an integer. In this paper, we consider the quiver $\Aone[e-1]$:

\begin{center}
\begin{tikzpicture}[
    scale=1,
    every node/.style={transform shape},
    dynnode/.style={circle, draw, thick, minimum size=3mm, inner sep=0pt, fill=white},
    label_style/.style={yshift=-2mm, font=\small}
]

    \node[dynnode, label={[yshift=1mm]above:$0$}] (n0) at (2, 1) {};


    \node[dynnode, label={[label_style]below:$1$}] (n1) at (0,0) {};
    \node[dynnode, label={[label_style]below:$2$}] (n2) at (1,0) {};

    \node (dots) at (2,0) {$\cdots$};

    \node[dynnode, label={[label_style]below:$e-2$}] (ne_2) at (3,0) {};
    \node[dynnode, label={[label_style]below:$e-1$}] (ne_1) at (4,0) {};

    
    \draw[thick] (n1) -- (n2) -- (dots) -- (ne_2) -- (ne_1);

    \draw[thick] (n0) -- (n1);
    \draw[thick] (n0) -- (ne_1);

\end{tikzpicture}
\end{center}

Let $\Gamma$ be the quiver of type $\Aone[e-1]$. It has vertex set $I=\{0,1,\dots,e-1\}$, which we identify with $\mathbb{Z}/e\mathbb{Z}$. In particular, we identify $e$ with $0$. Throughout this paper, we fix the cyclic orientation $i\rightarrow i+1$ for each $i \in I$.

The affine Cartan matrix $(a_{ij})_{i,j\in I}$ of type $\Aone[e-1]$ has $2$ on the diagonal.
Each $i\in I$ is adjacent to exactly two neighbours, namely $i-1$ and $i+1$ (with indices taken modulo $e$). For these neighbors we have $a_{i,i-1}=a_{i,i+1}=-1$, and all other off-diagonal entries are $0$.


The simple roots are $\{\alpha_i\mid i\in I\}$ and $Q^{+} := \bigoplus_{i \in I} \Z_{\geq 0} \alpha_i$ is the positive part of the root lattice. For $\alpha \in Q^{+}$, let $\height(\alpha)$ denote the \emph{height of~$\alpha$}; that is, $\height(\alpha)$ is the sum of the coefficients when $\alpha$ is expanded in terms of the simple roots $\alpha_i$.

Let~$\Sym_n$ be the symmetric group on~$n$ letters and let $\sigma_r = (r, r+1)$, for $1\leq r < n$, be the simple transpositions of~$\Sym_n$. Then~$\Sym_n$ acts on the set~$I^n$ on the left by place permutations.

If $\bi = (\bi_1, \dots , \bi_n) \in I^n$, then we define the associated positive root $\alpha(\bi)$ by:
\begin{equation}\label{eq:associated-positive-root}
    \alpha(\bi) := \alpha_{i_1} + \dots + \alpha_{i_n} \in Q^+    
\end{equation}
The $\Sym_n$-orbits on $I^n$ are the sets
\begin{equation*}
  I^\alpha := \{\bi \in I^n\mid \alpha=\alpha(\bi)\},
\end{equation*}
which are parametrized by $\alpha \in Q^+$ of height $n$.

Let $\Lambda_0,\dots,\Lambda_{e-1}$ be the fundamental weights. The dominant weight lattice is defined as
\begin{equation*}
    P^+:=\bigoplus_{i \in I} \Z_{\geq 0} \Lambda_i.
\end{equation*}
Any element $\Lambda=\sum_{i \in I} a_i\Lambda_i\in P^+$ is called a dominant weight, and the sum of the coefficients $\sum_{i \in I} a_i$ is the \emph{level} of $\Lambda$.

We denote the set of simple coroots by $\{\alpha_i^\vee \mid i \in I\}$. The coroot lattice is defined as $Q^\vee := \bigoplus_{i \in I} \Z \alpha_i^\vee$.
Let $\mathfrak{h}$ be the affine Cartan subalgebra satisfying $Q^\vee \subset \mathfrak{h}$. The dual Cartan subalgebra $\mathfrak{h}^*$ contains the root lattice $Q$ and the dominant weight lattice $P^+$.
There is a canonical bilinear pairing $\langle \cdot, \cdot \rangle : \mathfrak{h}^* \times \mathfrak{h} \rightarrow \mathbb{C}$ satisfying
\begin{equation*}
  \langle \alpha_j, \alpha_i^\vee \rangle = a_{ji} \quad \text{and} \quad \langle \Lambda_j, \alpha_i^\vee \rangle = \delta_{ji}
\end{equation*}
for all $i,j \in I$. Note that since we are in type $\Aone[e-1]$, the Cartan matrix $(a_{ij})$ is symmetric. Consequently, we can identify $\alpha_i$ with $\alpha_i^\vee$ to define a symmetric bilinear form $(\cdot | \cdot)$ on $\mathfrak{h}^*$. This form is determined by the values:
\begin{equation*}
  (\alpha_i | \alpha_j) = a_{ij} \quad \text{and} \quad (\Lambda_i | \alpha_j) = \delta_{ij}
\end{equation*}
for all $i,j \in I$.
\subsection{Young diagrams and Abaci}\label{subsec:young-diagram-abaci}
A \emph{partition} $\lambda$ is a finite weakly decreasing sequence of positive integers
$(\lambda_1,\ldots,\lambda_r)$. It is often convenient to regard $\lambda$ as a longer (or even infinite)
sequence $(\lambda_1,\cdots,\lambda_{r+1},\cdots)$ by setting $\lambda_i=0$ for all $i>r$. Each $\lambda_i$ is
called a \emph{part} of $\lambda$.

The \emph{length} $\ell(\lambda)$ is the number of nonzero parts, and the \emph{size} is
$|\lambda|=\sum_i \lambda_i$. Let $\Par[]_n$ be the set of partitions of $n$, and set
$\Par[]:=\bigsqcup_{n\ge 0}\Par_n$.

Similarly, a \emph{composition} $\mu$ is a finite sequence of positive integers $(\mu_1,\ldots,\mu_r)$. Its length and size are defined analogously. Every partition is a composition, but not
conversely.

A \emph{Young diagram} is a finite collection of square boxes arranged in left-justified rows.
Given a composition $\mu=(\mu_1,\ldots,\mu_r)$, its Young diagram $[\mu]$ is obtained by drawing $\mu_i$ boxes in the $i$th row. Conversely, counting the number of boxes in each row from top to bottom recovers the composition. Thus, we may identify a composition $\mu$ with its Young diagram $[\mu]$. Let $[\mu]_j$ be the $j$th row of $[\mu]$.
 
A \emph{node} is a square box of the Young diagram $[\lambda]$, specified by its coordinates $(r,c)$, where $r$ and $c$ denote the row and column in which the node lies. Its \emph{content} is defined to be
$c-r$. If there are two nodes $A=(r,c)$ and $B=(r',c')$ in one Young diagram, we say $A$ is \emph{above} $B$, or $B$ is
\emph{below} $A$ if $r<r'$.

\smallskip 
Fix a positive integer $e$. An \emph{$e$-abacus} has $e$ \emph{runners}, arranged from left to right and labeled
$0,1,\ldots,e-1$. The positions are indexed by $\mathbb{Z}_{\ge 0}$ and arranged on the runners as follows:
if $x\in\mathbb{Z}_{\ge 0}$ is written uniquely as $x=ae+b$ with $a\in\mathbb{Z}_{\ge 0}$ and
$b\in\{0,1,\ldots,e-1\}$, then position $x$ lies on the $b$-runner in row $a$. Rows increase downward.

An \emph{$e$-abacus with $M$ beads} is a subset $B\subset\mathbb{Z}_{\ge 0}$ of
cardinality $|B|=M$. Elements of $B$ are called \emph{bead positions}, and all other positions are \emph{gaps}.
We depict the $M$-abacus by placing a bead at each position in $B$ and leaving all other positions empty. A runner is called \emph{flush} if all the beads on that runner are as high as possible. Equivalently, if $c\in B$ is a bead on this runner, then for any $k\in\Z_{>0}$ such that $c-ke\geq 0$, we have $c-ke\in B$.

Let $\lambda=(\lambda_1,\lambda_2,\ldots)$ be a partition. Fix an integer $M\ge \ell(\lambda)$. The \emph{$M$-beta numbers} of $\lambda$ are
\[
  \beta_i^M(\lambda)=\lambda_i-i+M \qquad (1\le i\le M),
\]
and the corresponding \emph{$M$-beta set} is
\begin{equation}\label{eq:beta-set}
      B(\lambda;M)=\{\beta_i^M(\lambda)\mid 1\le i\le M\}\subset\mathbb{Z}_{\ge 0}
\end{equation}
The \emph{$e$-abacus of $\lambda$ with $M$ beads}, denoted $\Aba[M](\lambda)$, is the $M$-abacus having beads
precisely at the positions in $B(\lambda;M)$.

If $\beta_1>\beta_2>\cdots>\beta_M$ are the elements of $B(\lambda;M)$ in decreasing order, then
\[
  \lambda_i=\beta_i-M+i \qquad (1\le i\le M),
\]
so $B(\lambda;M)$ determines $\lambda$ uniquely (discard trailing zeros).

\begin{Example}\label{eg:basic-partition-abacus}
    Take $\lambda=(4,3,2,2)$. This partition has length $\ell(\lambda)=4$ and size $|\lambda|=11$.
    The Young diagram $[\lambda]$ is
    \begin{center}
        \Diagram[entries=contents,no border]{4,3,2,2}
    \end{center}
    where each node is filled with its content. 
 Choose $M=\ell(\lambda)=4$ and form the $M$–beta numbers $\beta^M_{i}(\lambda)$:
  \begin{center}
        $7,5,3,2$
  \end{center}
  and the $e$–abacus of $\lambda$ with $M$ beads is
   \begin{center}
        \Abacus[runner labels={0,1,2}]{3}{4,3,2,2}\
   \end{center}
\end{Example}
\subsection{Multipartitions, charges and residues}\label{subsec:multipartition-charge-residue}
The Young diagram has a natural filling of the contents defined in the last section, in this section, we introduce a more general notion of content, residue, which depends on the type $\Aone[e-1]$. 

Fix type $\Aone[e-1]$ and a dominant weight $\Lambda=\Lambda_i$. Let $\lambda$ be a partition and let $[\lambda]$ be its Young diagram. We fill each node of $[\lambda]$ with its \emph{residue}, defined by
\begin{equation}\label{eq:residue-of-node}
  \res_{\Lambda}(r,c)\equiv c-r+i \pmod e.
\end{equation}
A node with residue $i$ is called an $i$-node. The Young diagram filled with residues is denoted $[\lambda]_{\Lambda}$. When $\Lambda$ is clear from the context, we simply write $[\lambda]$ and $\res$ in place of $[\lambda]_{\Lambda}$ and $\res_{\Lambda}$, respectively.

Following the definition from \cite{bkw-graded-specht}, for each partition $\lambda$, we define the \emph{residue content} of $\lambda$ to be the positive root $\alpha_{\lambda}\in Q^+$ given by
\begin{equation}\label{eq:residue-content}
  \alpha_\lambda:=\sum_{A\in [\lambda]}\alpha_{\res A} \in Q^+.
\end{equation}

Let $\Par[\Lambda]$ be the set of partitions equipped with this residue labelling. For $\alpha\in Q^+$, we also set
\begin{equation}\label{eq:partition-with-fixed-residue-content}
  \Par[\Lambda]_\alpha:=\{\lambda\in\Par[\Lambda]\mid \alpha_{\lambda}=\alpha\}.
\end{equation}
\begin{Example}\label{eg:residue-sequence-pattern}
    Take $\Lambda=\Lambda_0$ and $e=5$, and consider the positive root $\alpha=\alpha_0+\alpha_1+\alpha_2+\alpha_3+\alpha_4$. Then $\Par[\Lambda]_\alpha$ consists of the partitions with the following Young diagrams:
    \begin{center}
        \Diagram[no border,e=5,entries=residues]{5}\quad \Diagram[no border,e=5,entries=residues]{4,1}\quad \Diagram[no border,e=5,entries=residues]{3,1,1}\quad \Diagram[no border,e=5,entries=residues]{2,1^3}\quad \Diagram[no border,e=5,entries=residues]{1^5}
    \end{center}
    In contrast, if we take $\beta=\alpha_1$, then $\Par[\Lambda]_\beta=\emptyset$.
\end{Example}

A multipartition with $\ell$ components or simply an \emph{$\ell$-partition} of $n$ is an $\ell$-tuple
$\blam=(\blam^{(1)},\ldots,\blam^{(\ell)})$ of partitions such that
\[
\sum_{m=1}^{\ell}\bigl|\blam^{(m)}\bigr|=n.
\]
For each $m\in\{1,\ldots,\ell\}$, let
$[\blam^{(m)}]$ be the Young diagram of $\blam^{(m)}$, and let $\blam^{(m)}_r$ be the $r$th part of $\blam^{(m)}$.
The Young diagram of $\blam$ is the disjoint union
\[
    [\blam]\;:=\;\bigsqcup_{m=1}^{\ell}\{m\}\times[\blam^{(m)}]
    \;=\;\{(m,r,c)\in\mathbb{Z}_{>0}^3 \mid 1\le m\le \ell,\ 1\le c\le \blam^{(m)}_r\}.
\]
As in the partition case, a \emph{node} of $[\blam]$ is an element $A=(m,r,c)\in[\blam]$, where $A$ is the square box in row $r$ and column $c$ of the $m$th component $[\blam^{(m)}]$. 

Similarly, $\ell$-compositions and the corresponding notions are defined in the same way.

Fix a dominant weight $\Lambda$ of level $\ell$. To define residues of nodes, we use the notion of a charge. A \emph{charge of $\Lambda$} is an $\ell$-tuple $\bkappa=(\kappa_1,\ldots,\kappa_\ell)\in I^\ell$ such that $\Lambda=\sum\limits_{m=1}^{\ell}\Lambda_{\kappa_m}$. Given such a charge and a node $A=(m,r,c)\in[\blam]$, we define its residue by
\[
  \res_{\bkappa}(A)\equiv \kappa_m+c-r \pmod e.
\]
Equivalently, the residue function on the $m$th component $[\blam^{(m)}]$ is $\res_{\Lambda_{\kappa_m}}$. The residue content of an $\ell$-partition is the sum of the residue contents of its components, namely
$\alpha_{\blam}=\sum_{1\le m\le \ell}\alpha_{\blam^{(m)}}$.

Let $\Par[\bkappa]_n$ be the set of $\ell$-partitions of $n$, with residues defined with respect to the charge $\bkappa$, and set $\Par[\bkappa]=\bigsqcup\limits_{n\ge 0}\Par[\bkappa]_n$. For $\alpha\in Q^+$, let $\Par[\bkappa]_{\alpha}$ denote the set of $\ell$-partitions of residue content $\alpha$. We remark that if $\ell=1$ and $\kappa\in I$, then $\Par[\Lambda_{\kappa}]_{\alpha}$ coincides with $\Par[\kappa]_{\alpha}$ for any $\alpha\in Q^+$.
\begin{Example}\label{eg:multi-partition-charge}
    Fix type $\Aone[2]$, $\Lambda=\Lambda_2+\Lambda_1$. Consider the $2$-partition $\blam=((3,2,1),(5,1,1))$. Take a charge of $\Lambda$ to be $\bkappa=(1,2)$, then $[\blam]$ is the following:
    \begin{center}
        \Multidiagram[no border,e=3,entries=residues,charge={1,2}]{3,2,1|5,1,1}
    \end{center}
    The residue content of $\blam$ is $\beta:=5\alpha_0+4\alpha_1+4\alpha_2$ and hence $\blam\in\Par[\bkappa]_{\beta}$.
    Take another charge of $\Lambda$ to be $\bkappa'=(2,1)$, then $[\blam]$ is the following:
    \begin{center}
        \Multidiagram[no border,e=3,entries=residues,charge={2,1}]{3,2,1|5,1,1}
    \end{center}
    The residue content of $\blam$ is $\alpha_{\blam}:=4\alpha_0+4\alpha_1+5\alpha_2$ and hence $\blam\in\Par[\bkappa']_{\gamma}$.
\end{Example}
\subsection{Tableaux}\label{subsec:tableaux}
Fix the quiver of type $\Aone[e-1]$ and let $I$ be the vertex set. Fix $\ell\in\mathbb{Z}_{>0}$ and a multicharge $\bkappa=(\kappa_1,\dots,\kappa_\ell)\in I^\ell$.
Let $\blam=(\blam^{(1)},\dots,\blam^{(\ell)})\in\Par[\bkappa]_{n}$ be an $\ell$-multipartition of $n$.

A \emph{$\blam$-tableau}, or \emph{(multi)tableau of shape $\blam$}, is a bijection
\[
    T:[\blam]\xrightarrow{\ \sim\ }\{1,2,\dots,n\}.
\]

The \emph{shape} of a tableau $T$, written $\Shape(T)$, is the $\ell$-partition $\blam$. A $\blam$-tableau $T$ can be viewed as a filling of the nodes of the Young diagram $[\blam]$ with the numbers in $\{1,2,\dots,n\}$ without repetitions.

Let $r,s\in\{1,2,\cdots,n \}$. If $T^{-1}(r)$ and $T^{-1}(s)$ are horizontally adjacent or vertically adjacent in $[\blam]$, we write $r\rightarrow_{T} s$ or $r\downarrow_{T} s$, respectively.

For $1\le m\le n$, write $T\downarrow m$ for the tableau obtained from $T$ by deleting all nodes occupied by entries $>m$. In other words, $T\downarrow m$ is the restriction of $T$ to the $\ell$-composition $T^{-1}(\{1,2,\dots,m\})$.

For $r\in\{1,\dots,n\}$, let $A_r:=T^{-1}(r)\in[\blam]$ and define the residue of $r$ in $T$ by
\[
    \res_r(T)\;:=\;\res_{\bkappa}(A_r)\in I.
\]
The \emph{residue sequence} of $T$ is
\begin{equation}\label{eq:residue-sequence}
    \bi^{T}\;=\;\res(T)\;:=\;\bigl(\res_1(T),\dots,\res_n(T)\bigr)\in I^n.
\end{equation}

A tableau $T$ of shape $\blam$ is \emph{row-standard} if, for each component $\blam^{(m)}$, its entries increase strictly from left to right
along each row of $\blam^{(m)}$. It is \emph{column-standard} if, for each component $\blam^{(m)}$, its entries increase strictly
from top to bottom along each column of $\blam^{(m)}$. It is \emph{standard} if it is both row-standard and column-standard.
Equivalently, $T$ is standard if and only if, for each $d\in\{1,\dots,n\}$, the $\Shape(T\downarrow d)$ is an $\ell$-partition.
Let $\RST{\blam}$ and $\Std(\blam)$ be the sets of row-standard $\blam$-tableaux and standard $\blam$-tableaux, respectively.

Equip $[\blam]$ with the \emph{row-reading order} $\prec$ defined by
\[
    (m,a,b)\prec(m',a',b')
    \quad\Longleftrightarrow\quad
    \bigl(m<m'\bigr)\ \text{or}\ \bigl(m=m'\ \text{and}\ (a<a'\ \text{or}\ (a=a'\ \text{and}\ b<b'))\bigr).
\]
Informally, we read the nodes of $[\blam]$ from the first component to the last component, and within each component we read from left to right along each row, taking rows from top to bottom.

Define the \emph{initial tableau} $T^{\blam}$ to be the unique $\blam$-tableau whose entries
increase strictly along $\prec$. Set
\begin{equation}\label{eq:initial-residue-sequence}
    \bi^{\blam}\;:=\;\bi\bigl(T^{\blam}\bigr)\in I^n.
\end{equation}

The symmetric group $\Sym_n$ acts on the set of $\blam$-tableaux from the left by permuting entries:
\[
    (\sigma\cdot T)(A)\;:=\;\sigma\bigl(T(A)\bigr)\qquad(\sigma\in\Sym_n,\ A\in[\blam]).
\]
For each $\blam$-tableau $T$, define $w^T\in\Sym_n$ by the condition
\begin{equation}\label{eq:wt-def}
    w^T\cdot T^{\blam}\;=\;T.
\end{equation}
\begin{Example}\label{eg:multi-tableau-residue-sequence}
    Continue with \autoref{eg:multi-partition-charge}, the initial $\blam$-tableau $T^{\blam}$ is the following:
    \begin{center}
        \Multitableau[]{123,45,6 | 789{10}{11},{12},{13} }
    \end{center}
    If the charge is $\bkappa=(1,2)$, then the residue sequence $\bi^{\blam}$ of $T^{\blam}$ is $(1200122012010)$. Take $T$ to be the following tableau:
    \begin{center}
        \Multitableau[]{127,4{12},9 | 368{10}{11},{5},{13} }
    \end{center}
    the residue sequence $\bi^{T}$ is $(1220100122010)$ and
    \[ 
        w^T=\sigma_{5}\sigma_{6}\sigma_{7}\sigma_{8}\sigma_{9}\sigma_{10}\sigma_{11}\sigma_{10}\sigma_{9}\sigma_{7}\sigma_{8}\sigma_{5}\sigma_{6}\sigma_{7}\sigma_{3}\sigma_{4}\sigma_{5}\sigma_{6}\sigma_{5}\sigma_{3}
    \] 
    where $\sigma_i$ is the simple transposition $(i,i+1)\in\Sym_{13}$. 
\end{Example}
\subsection{Dominance order and Bruhat order on tableaux}\label{subsec:orders-on-tableaux}
For $\ell$-compositions $\blam,\bmu$ of $n$, write $\blam\unrhd \bmu$ if for every $1\le t\le \ell$ and every $k\ge 1$ one has
\[
\sum_{m=1}^{t-1}\bigl|\blam^{(m)}\bigr|+\sum_{i=1}^{k}\blam^{(t)}_i
\;\ge\;
\sum_{m=1}^{t-1}\bigl|\bmu^{(m)}\bigr|+\sum_{i=1}^{k}\bmu^{(t)}_i.
\]
Write $\blam\triangleright \bmu$ if $\blam\unrhd \bmu$ and $\blam\neq \bmu$.

Fix charge $\bkappa$ and positive root $\alpha$ with $\height(\alpha)=n$. Let $\bnu\in\Par[\bkappa]_{\alpha}$, and let $S,T$ be row-standard $\bnu$-tableaux. Write $S\unrhd T$ if
\[
\Shape(S\downarrow m)\unrhd \Shape(T\downarrow m)
\qquad\text{for all }m=1,2,\dots,n,
\]
and write $S\triangleright T$ if $S\unrhd T$ and $S\neq T$.

Let \(\le\) denote the (strong) Bruhat order on \(\Sym_n\), with respect to the Coxeter generators \(\sigma_1,\dots,\sigma_{n-1}\).
Thus, for \(u,w\in\Sym_n\), one has \(u\le w\) if and only if there exists a reduced expression
\(w=\sigma_{r_1}\cdots \sigma_{r_m}\) and indices \(1\le a_1<\cdots<a_b\le m\) such that
$u=\sigma_{r_{a_1}}\cdots \sigma_{r_{a_b}}$.
See, for example, \cite[Section~5.10]{humphreys-reflection-groups}. In particular, \(1\le w\) for all \(w\in\Sym_n\).

\begin{Lemma}\label{lem:dominance-vs-bruhat-rowstandard}
Let $S,T$ be row-standard $\bnu$-tableaux. Then
\[
S\unrhd T \quad\Longleftrightarrow\quad w^{S}\le w^{T}.
\]
\end{Lemma}

\begin{proof}
This is the Ehresmann--James theorem; see, for example, \cite[Theorem~3.8]{mathas-iwahori-hecke}.
\end{proof}
\subsection{Degree of tableaux}\label{subsec:degree-of-tableaux}

Following~\cite{bkw-graded-specht}, we recall the combinatorial degree on standard tableaux. 

Fix the quiver $\Aone[e-1]$ and a charge $\bkappa\in I^{\ell}$.
Let $\blam\in\Par[\bkappa]$ be an $\ell$-multipartition of $n$.
A node $A\in[\blam]$ is \emph{removable} if $[\blam]\setminus\{A\}$ is the Young diagram of an
$\ell$-multipartition, and a node $B\notin[\blam]$ is \emph{addable} if $[\blam]\cup\{B\}$ is the Young
diagram of an $\ell$-multipartition. If $\res(A)=i$ and $A$ is removable, then $A$ is called a removable $i$-node of $[\blam]$.
Similarly, if $\res(B)=i$ (in $[\blam]\cup \{B\}$) and $B$ is addable, then $B$ is called an addable $i$-node of $[\blam]$.

We order nodes by declaring $B=(m',r',c')$ to be \emph{below} $A=(m,r,c)$ if either
$m'>m$, or $m'=m$ and $r'>r$.
For a removable $i$-node $A$ of $\blam$, set
\begin{equation}\label{eq:degree-of-last-node}
    d_A(\blam)
  :=
  \#\{\text{addable $i$-nodes of $[\blam]$ below $A$}\}
  -
  \#\{\text{removable $i$-nodes of $[\blam]$ below $A$}\}.
\end{equation}

Now let $T\in\Std(\blam)$.  Define $\deg(T)$ inductively on $n$ by $\deg(\varnothing):=0$, and,
for $n>0$,
\[
  \deg(T) := d_A(\blam) + \deg\!\bigl(T\downarrow(n-1)\bigr),
\]
where $A$ is the node occupied by $n$ in $T$, and $T\downarrow(n-1)$ is the tableau obtained by
deleting the entry $n$. This definition is valid because, by definition of a standard tableau, the shape of $T\downarrow (n-1)$ is an $\ell$-partition, and $T\downarrow (n-1)$ is again standard.
\begin{Example}\label{eg:degree-of-tableaux}
    Continue with \autoref{eg:multi-tableau-residue-sequence}, the degrees are $\deg T^{\blam}=4$ and $\deg T=6$.
\end{Example}

\subsection{Garnir tableaux}\label{subsec:garnir-tableaux}
In this section, following \cite[Section 5]{kmr-universal-specht-type-A}, we introduce the Garnir combinatorics needed later. Let $\blam=(\blam^{(1)},\cdots,\blam^{(\ell)})$ be an $\ell$-partition, and let $[\blam]$ be its Young diagram.

\begin{Definition}\label{dfn:garnir}
    A node $A=(m,r,c)\in[\blam]$ is a \textbf{Garnir node} of $\blam$ if $(m,r+1,c)\in[\blam]$. The \textbf{Garnir belt} of $A$ is the set $\mathcal{B}^A$ of nodes of $[\blam]$ consisting of $A$ and all nodes directly to the right of $A$, together with the node directly below $A$ and all nodes directly to the left of that node in the same component. Explicitly,
    \[
    \mathcal{B}^A=\{(m,r,z)\in[\blam]\mid c\le z\le \blam^{(m)}_r\}\cup \{(m,r+1,z)\in[\blam]\mid 1\le z\le c\}.
    \]
\end{Definition}
\begin{Definition}\label{dfn:garnir-tableau}
    Let $A\in[\blam]$ be a Garnir node. The \textbf{Garnir tableau} $G^A$ is the unique row-standard tableau satisfying:
    \begin{itemize}
        \item it agrees with $T^{\blam}$ on all nodes outside the Garnir belt $\mathcal{B}^A$,
        \item its entries in $\mathcal{B}^A$ increase from the bottom-left to the top-right.
    \end{itemize}
\end{Definition}
\begin{Example}\label{eg:garnir-tableau}
   Consider the $2$-partition $\blam=(5,4,3,3,1\mid 14,10,2)$. We color the Garnir nodes by {\color{cyan}cyan}:
   \begin{center}
        \tikzset{
        C/.style={fill=cyan,text=white},
        O/.style={fill=OrangeRed, text=white}
            }

       \Multitableau[]{[C]1[C]2[C]3[C]45,[C]6[C]7[C]89,[C]{10}[C]{11}[C]{12},[C]{13}{14}{15},{16} | [C]{17}[C]{18}[C]{19}[C]{20}[C]{21}[C]{22}[C]{23}[C]{24}[C]{25}[C]{26}{27}{28}{29}{30},[C]{31}[C]{32}{33}{34}{35}{36}{37}{38}{39}{40},{41}{42} }
    \end{center}
    Let $A=(2,1,8)$, then $T(A)=24$ and the Garnir belt $\mathcal{B}^A$ consists of the following {\color{OrangeRed}orange} nodes:
    \begin{center}
        \tikzset{
        C/.style={fill=cyan,text=white},
        O/.style={fill=OrangeRed, text=white}
            }

        \Multitableau[]{12345,6789,{10}{11}{12},{13}{14}{15},{16} | {17}{18}{19}{20}{21}{22}{23}[O]{24}[O]{25}[O]{26}[O]{27}[O]{28}[O]{29}[O]{30},[O]{31}[O]{32}[O]{33}[O]{34}[O]{35}[O]{36}[O]{37}[O]{38}{39}{40},{41}{42} } 
   \end{center}
    and the Garnir tableau $G^A$ is the following:
    \begin{center}
        \tikzset{
        C/.style={fill=cyan,text=white},
        O/.style={fill=OrangeRed, text=white}
            }

        \Multitableau[]{12345,6789,{10}{11}{12},{13}{14}{15},{16} | {17}{18}{19}{20}{21}{22}{23}[O]{32}[O]{33}[O]{34}[O]{35}[O]{36}[O]{37}[O]{38},[O]{24}[O]{25}[O]{26}[O]{27}[O]{28}[O]{29}[O]{30}[O]{31}{39}{40},{41}{42} } 
   \end{center}
\end{Example}
Fix the quiver of type $\Aone[e-1]$ and let $I$ be its vertex set. Let $\Lambda\in P^+$ and let $\bkappa=(\kappa_1,\cdots,\kappa_\ell)\in I^{\ell}$ be a charge of $\Lambda$. Fix $\alpha\in Q^+$ of height $n$ and assume $\blam\in \Par[\bkappa]_{\alpha}$. In particular, we can assign a residue to each node, as discussed in \autoref{subsec:tableaux}.

Fix a Garnir node $A=(m,r,c)\in[\blam]$ and write $\mathcal{B}^A$ for its Garnir belt.

\begin{Definition}\label{dfn:A-brick}
    A \textbf{(row) $A$-brick} is a subset $B\subseteq \mathcal{B}^A$ of the form
    \[
        B=\{(m,x,z),(m,x,z+1),\dots,(m,x,z+e-1)\}
    \]
    for some $x\in\{r,r+1\}$ and some $z\in\mathbb{Z}_{>0}$, such that $\res(m,x,z)=\res(A)$.
\end{Definition}

Let $k=k^A$ be the number of row $A$-bricks contained in $\mathcal{B}^A$.
We always list these bricks as $B^A_1,B^A_2,\dots,B^A_k$ in the following order: first, the bricks contained in row $r+1$ of the $m$th component, ordered from left to right; then the bricks contained in row $r$ of the $m$th component, ordered from left to right.

\begin{Definition}\label{dfn:brick-group}
    Let $A\in[\blam]$ be a Garnir node and write $\mathcal{B}^A$ for its Garnir belt. Let
    $B^A_1,\dots,B^A_{k^A}$ be the row $A$-bricks in $\mathcal{B}^A$, listed as above.
    For each $1\le t\le k^A$, set
    $n^A_t:=\min\{\,G^A(x)\mid x\in B^A_t\,\}$.
    For $1\le t<k^A$, define the \emph{brick transposition} $w^A_t\in\Sym_n$ by
    \begin{equation}\label{eq:brick-transposition}
    w^A_t\;:=\;\prod_{a=0}^{e-1}\bigl(n^A_t+a,\;n^A_t+e+a\bigr).
    \end{equation}
    The \emph{brick permutation group} of $A$ is the subgroup
    $\mathfrak{S}^A:=\langle\,w^A_1,\dots,w^A_{k^A-1}\,\rangle\ \le\ \Sym_n$,
    with the convention $\mathfrak{S}^A=\{1\}$ if $k^A\le 1$.
\end{Definition}

The group \(\mathfrak{S}^A\) acts on the set of \(\blam\)-tableaux by permuting entries, hence it acts on the Garnir tableau \(G^A\).
Define the \emph{Garnir set} and the associated residue sequence by
\[
\Gar^{A}\;:=\;\{\,w\cdot G^A\mid w\in\mathfrak{S}^A\,\},
\qquad
\bi^{A}\;:=\;\bi^{G^A}\in I^n.
\]

\begin{Lemma}\label{lem:Gar-orbit-characterisation}
Suppose that \(\blam\in \Par[\bkappa]_{\alpha}\) and that \(A\in[\blam]\) is a Garnir node. Then
\[
\Gar^{A}\setminus\{G^A\}
=
\Bigl\{
T\in \Std(\blam)
\ \Bigm|\
T\unrhd G^{A}
\text{ and }
\bi^{T}=\bi^{A}
\Bigr\}.
\]
\end{Lemma}

\begin{proof}
This is \cite[Lemma~5.5]{kmr-universal-specht-type-A}.
\end{proof}

In particular, with respect to the partial order $\unrhd$ defined in \autoref{subsec:orders-on-tableaux}, $G^A$ is the unique minimal element in $\Gar{A}$. There is also a unique maximal element $T^A$ in $\Gar{A}$, obtained by rearranging the row $A$-bricks in $\mathcal{B}^A$ in row-reading order.

Let $f=f^A$ be the number of row $A$-bricks in row $r$ of the Garnir belt $\mathcal{B}^A$. Define $\mathscr{D}^A$ to be the set of minimal-length left coset representatives of $\mathfrak{S}_f\times\mathfrak{S}_{k-f}$ in $\mathfrak{S}^A\cong\mathfrak{S}_k$.
Note that $\mathfrak{S}^A\le \Sym_n$, hence $\mathscr{D}^A\subseteq \Sym_n$ and its elements act on $\blam$-tableaux. Moreover,
\begin{equation}\label{eq:GarA-param}
\Gar^{A}\;=\;\{\,w\cdot T^A \mid w\in\mathscr{D}^A\,\},
\end{equation}

\begin{Example}\label{eg:garnir-bricks-garnir-sets}
    Continue with \autoref{eg:garnir-tableau} and take $A=(2,1,8)$. Fix type $\Aone[2]$ and the charge $(0,1)$. The Young diagram $[\blam]$, filled with residues, is as follows:
    \begin{center}
        \Multidiagram[e=3,entries=residues,charge={0,1}]{5,4,3,3,1|14,10,2} 
   \end{center}
    
    The row $A$-bricks are colored as below. The bricks $B^A_1,B^A_2,B^A_3,B^A_4$ are colored {\color{pink}pink}, {\color{green}green}, {\color{cyan}cyan}, and {\color{OrangeRed}orange}, respectively.
    \begin{center}
        \tikzset{
            C/.style={fill=cyan,text=white},
            O/.style={fill=OrangeRed, text=white}, 
            G/.style={fill=pink,text=black},
            Y/.style={fill=green, text=black}
        }
        \Multitableau[]{12345,6789,{10}{11}{12},{13}{14}{15},{16} | {17}{18}{19}{20}{21}{22}{23}[C]{32}[C]{33}[C]{34}[O]{35}[O]{36}[O]{37}{38},{24}{25}[G]{26}[G]{27}[G]{28}[Y]{29}[Y]{30}[Y]{31}{39}{40},{41}{42} } 
   \end{center}
    So $k^A=4$ and $f^A=2$. For $1\le i\le 3$, let $w^A_i$ be the brick transposition. Then
    the brick permutation group is $\Sym^A=\langle w^A_1,w^A_2,w^A_3\rangle$, and $\mathscr{D}^A$ is
        \[
        \mathscr{D}^A=\{1,w^A_2,w^A_1w^A_2,w^A_3w^A_2,w^A_1w^A_3w^A_2,w^A_2w^A_1w^A_3w^A_2\}
    \]
    Apply $w^A_2$ to the Garnir tableau, we get the following:
    \begin{center}
        \tikzset{
            C/.style={fill=cyan,text=white},
            O/.style={fill=OrangeRed, text=white}, 
            G/.style={fill=pink,text=black},
            Y/.style={fill=green, text=black}
        }
        \Multitableau[]{12345,6789,{10}{11}{12},{13}{14}{15},{16} | {17}{18}{19}{20}{21}{22}{23}[Y]{29}[Y]{30}[Y]{31}[O]{35}[O]{36}[O]{37}{38},{24}{25}[G]{26}[G]{27}[G]{28}[C]{32}[C]{33}[C]{34}{39}{40},{41}{42} } 
   \end{center}
    The unique maximal element $T^A$ in $\Gar^A$ is the following:
    \begin{center}
        \tikzset{
            C/.style={fill=cyan,text=white},
            O/.style={fill=OrangeRed, text=white}, 
            G/.style={fill=pink,text=black},
            Y/.style={fill=green, text=black}
        }
        \Multitableau[]{12345,6789,{10}{11}{12},{13}{14}{15},{16} | {17}{18}{19}{20}{21}{22}{23}[G]{26}[G]{27}[G]{28}[Y]{29}[Y]{30}[Y]{31}{38},{24}{25}[C]{32}[C]{33}[C]{34}[O]{35}[O]{36}[O]{37}{39}{40},{41}{42} } 
   \end{center}
    The Garnir set is then $\Gar^{A}=\{w\cdot T^A\mid w\in\mathscr{D}^A\}$. In particular, $G^A=(w^A_2w^A_1w^A_3w^A_2)\cdot T^A$.
\end{Example}
\section{KLR Algebras and Universal Specht Module}\label{sec:klr-specht}
In this section, we introduce KLR algebras, universal Specht modules, and the results we need to state our main result. For more details, readers are welcome to refer to \cite{kmr-universal-specht-type-A}.
\subsection{KLR algebras}\label{subsec:klr-algebras}
Let $\bk$ be a field, and fix $\alpha\in Q^+$ such that $\height(\alpha)=n$.

\begin{Definition}[\cite{khovanovlauda-klr-1,khovanovlauda-klr-3,rouquier-2kacmoody}]\label{def:klr-algebras}
    The KLR algebra $R_{\alpha}$ of type $\Aone[e-1](e\geq 2)$ is the unital $\bk$-algebra generated by the elements:
    \begin{equation}\label{klr-generating-elements}
        \{e(\bi)|\bi\in I^\alpha\}\cup \{y_1,\cdots,y_n\}\cup\{\psi_1,\cdots,\psi_{n-1}\}
    \end{equation}
    subject only to the following relations:
    \begin{align}
        e(\bi)e(\bj) &= \delta_{\bi,\bj} e(\bi),  \sum\limits_{\bi \in I^\alpha} e(\bi) = 1 \label{eq:idempotents} \\ 
        y_r e(\bi) &= e(\bi) y_r, \psi_r e(\bi) = e(\sigma_r \bi) \psi_r \label{eq:commute_y_psi_e} \\
        y_r y_s &= y_s y_r \label{eq:y_commute} \\
        \psi_r y_s &= y_s \psi_r \quad \text{if } s \ne r, r+1 \label{eq:psi_y_commute} \\
        \psi_r \psi_s &= \psi_s \psi_r \quad \text{if } |r - s| > 1 \label{eq:psi_commute_far} \\
        \psi_r y_{r+1} e(\bi) &= (y_r \psi_r + \delta_{\bi_r, \bi_{r+1}}) e(\bi) \label{eq:psi_y_shift_left} \\
        y_{r+1} \psi_r e(\bi) &= (\psi_r y_r + \delta_{\bi_r, \bi_{r+1}}) e(\bi) \label{eq:y_psi_shift_right} \\
        \psi_r^2 e(\bi) &= Q_{\bi_r, \bi_{r+1}}(y_r, y_{r+1}) e(\bi) \label{eq:psi_square} \\
        \psi_r \psi_{r+1} \psi_r e(\bi) &= \psi_{r+1} \psi_r \psi_{r+1} e(\bi) + Q_{\bi_r, \bi_{r+1}, \bi_{r+2}}(y_r, y_{r+1}, y_{r+2}) e(\bi) \label{eq:braid_relation}
    \end{align}
    where
    \[Q_{\bi_r, \bi_{r+1}}(y_r, y_{r+1})=\begin{cases}0 &\text{ if }\bi_r=\bi_{r+1}\\1& \text{ if }\bi_{r+1}\neq \bi_r,\bi_{r}\pm 1\\ y_{r+1}-y_r & \text{ if }\bi_r\rightarrow \bi_{r+1}\\ y_r-y_{r+1}&\text{ if }\bi_r\leftarrow \bi_{r+1}\end{cases}\]
    and 
    \[Q_{\bi_r, \bi_{r+1},\bi_{r+2}}(y_r, y_{r+1},y_{r+2})=\begin{cases}1 &\text{ if }\bi_r=\bi_{r+2}\to \bi_{r+1}\\-1& \text{ if }\bi_r=\bi_{r+2}\leftarrow \bi_{r+1}\\ 0&\text{ else }\end{cases}\]
\end{Definition}
Given any dominant weight $\Lambda\in P^+$, the corresponding cyclotomic KLR algebra $R^\Lambda_\alpha$ is generated by the same elements \autoref{klr-generating-elements} subject only to the above relations with the additional cyclotomic relations
\begin{equation}\label{eq:cyclotomic}
    y_1^{(\Lambda\mid \alpha_{\bi_1})}e(\bi)=0
    \qquad\text{for all }\bi=(\bi_1,\dots,\bi_n)\in I^\alpha.
\end{equation}

Most importantly, $R_\alpha$ and $R^\Lambda_\alpha$ have $\mathbb{Z}$-gradings determined by setting $e(\bi)$ to be of degree $0$, $y_r$ of degree $2$, and $\psi_re(\bi)$ of degree $-a_{\bi_r,\bi_{r+1}}$ for all $r$ and $\bi\in I^\alpha$.

Each $w\in \Sym_n$ can be written as $w=\prod\limits_{1\leq j\leq m} \sigma_{i_j}$ for some $m$. If the expression is reduced (that is, of minimal length), then we define $\psi_w:=\psi_{i_1}\cdots\psi_{i_m}$. However, this definition depends on the choice of reduced expression of $w$, since \autoref{eq:braid_relation} breaks the usual braid relation. We fix, for each $w\in \Sym_n$, a choice of reduced expression, and define $\psi_w$ with respect to that reduced expression.

\begin{Definition}\label{def:psi-tableau}
    Fix quiver type $\Aone[e-1]$ with vertex set $I$, and let $\Lambda\in P^+$ and $\alpha\in Q^+$. Let $\bkappa=(\kappa_1,\cdots,\kappa_\ell)\in I^\ell$ be a charge of $\Lambda$. Take $\blam\in\Par[\bkappa]_\alpha$ and $T\in\RST{\blam}$ a row-standard $\blam$-tableau. Define $\psi^T:=\psi_{w^T}$ for the fixed choice of reduced expression of $w^T$, where $w^T$ is defined in \autoref{eq:wt-def}.
\end{Definition}
The degree of standard tableaux defined in \autoref{subsec:degree-of-tableaux} is compatible with the grading of KLR algebras in the following sense:
\begin{Lemma}[{\cite[Corollary 3.14]{bkw-graded-specht}}]\label{lm:degree-match-psi-tableau}
    With the same setting as in \autoref{def:psi-tableau}, let $T\in\Std(\blam)$ be a standard $\blam$-tableau, then
    \[
        \deg \big(\psi^T e(\bi^{\blam})\big)=\deg T-\deg T^{\blam}
    \]
\end{Lemma}
\begin{Remark}\label{rmk:degree-of-row-standard-tableau}
    The definition of the degree of standard tableaux extends to row-standard tableaux by \autoref{lm:degree-match-psi-tableau}.
\end{Remark}
We define the KLR algebra of rank $n$ to be $R_n:=\bigoplus\limits_{\alpha\in Q^+_n}R_{\alpha}$, where $Q^+_n=\{\alpha\in Q^+\mid \height(\alpha)=n\}$.
\begin{Theorem}[{\cite[Theorem 2.5]{khovanovlauda-klr-1}}]\label{thm:base-of-klr}
    The algebra $R_n$ is free as a $\bk$-module with basis $\{\psi_w y_1^{m_1} \dots y_n^{m_n} 1_{\mathbf{i}} \mid w \in \mathfrak{S}_n, m_1, \dots, m_n \in \mathbb{Z}_{\geq 0}, \mathbf{i} \in I^n\}$.
\end{Theorem}

Fix a charge $\bkappa=(\kappa_1,\cdots,\kappa_{\ell})\in I^\ell$ of $\Lambda$.
For each $\ell$-partition $\blam\in\Par[\bkappa]_\alpha$, recall from \autoref{eq:initial-residue-sequence} that $\bi^{\blam}$ is the residue sequence of the initial tableau. We define the idempotent $e(\blam)$ by
\begin{equation}\label{eq:idempotent-of-partition}
    e_{\blam}:=e(\bi^{\blam}),
\end{equation}

By definition, the elements $\{e(\bi)\mid \bi\in I^\alpha\}$ are pairwise orthogonal idempotents, and
the KLR algebra $R_\alpha$ is unital with identity element $1_{\alpha}\ :=\ \sum_{\bi\in I^\alpha} e(\bi)$.
In particular, for any $R_\alpha$-module $M$ one has the decomposition $M=\bigoplus_{\bi\in I^\alpha} e(\bi)M$.

\subsection{Module category}\label{subsec:module-category}
Throughout, we work with $\mathbb Z$-graded modules and (unless explicitly stated otherwise) homogeneous
homomorphisms of degree $0$.

Let $A=\bigoplus_{a\in\mathbb Z}A_a$ be a $\mathbb Z$-graded $\bk$-algebra. A left $A$-module $M$ is
\emph{$\mathbb Z$-graded} if $M=\bigoplus_{d\in\mathbb Z}M_d$ and, for homogeneous $x\in A_a$, one has
$xM_d\subseteq M_{a+d}$ for all $d\in\mathbb Z$. A homomorphism $f:M\to N$ of graded $A$-modules is
\emph{homogeneous of degree $k\in\mathbb Z$} if $f(M_d)\subseteq N_{d+k}$ for all $d\in\mathbb Z$. For
$k\in\mathbb Z$, the \emph{grading shift} $M\langle k\rangle$ is defined by $(M\langle k\rangle)_d:=M_{d-k}$.

Specializing to $A=R_\alpha$, write $R_\alpha\textup{-}\mathrm{mod}$ for the category of finitely generated
graded left $R_\alpha$-modules with degree $0$ homomorphisms.

Let $\alpha,\beta\in Q^+$. Set $R_{\alpha,\beta}:=R_\alpha\otimes_\bk R_\beta$ with tensor product grading. If
$M\in R_\alpha\textup{-}\mathrm{mod}$ and $N\in R_\beta\textup{-}\mathrm{mod}$, their \emph{outer tensor product}
is $M\boxtimes N:=M\otimes_\bk N$, regarded as a graded $R_{\alpha,\beta}$-module via
$(x\otimes y)(m\otimes n)=(xm)\otimes(yn)$.

There is a canonical injective homogeneous (non-unital) algebra homomorphism
$\iota_{\alpha,\beta}:R_{\alpha,\beta}\hookrightarrow R_{\alpha+\beta}$ sending
$e(\bi)\otimes e(\bj)\mapsto e(\bi\bj)$, where $\bi\bj$ denotes concatenation. The image of
$1_\alpha\otimes 1_\beta$ is the idempotent
\[
e_{\alpha,\beta}:=\iota_{\alpha,\beta}(1_\alpha\otimes 1_\beta)
=\sum_{\bi\in I^\alpha,\ \bj\in I^\beta} e(\bi\bj)\ \in R_{\alpha+\beta}.
\]

Define the induction functor
\[
\Ind^{\alpha+\beta}_{\alpha,\beta}
:=R_{\alpha+\beta}e_{\alpha,\beta}\otimes_{R_{\alpha,\beta}}-
\;:\;R_{\alpha,\beta}\textup{-}\mathrm{mod}\to R_{\alpha+\beta}\textup{-}\mathrm{mod}.
\]
Moreover, $\Ind^{\alpha+\beta}_{\alpha,\beta}$ is exact by \cite[Proposition~2.16]{khovanovlauda-klr-1}.

\medskip
These constructions generalize to $n\ge2$ factors. Given $\beta_1,\dots,\beta_n\in Q^+$, set
$R_{\beta_1,\dots,\beta_n}:=R_{\beta_1}\otimes_\bk\cdots\otimes_\bk R_{\beta_n}$ and
$\beta:=\beta_1+\cdots+\beta_n$. Let $R_{\beta_1,\dots,\beta_n}\hookrightarrow R_\beta$ be the canonical
injective homogeneous (non-unital) map. The image of
$1_{\beta_1}\otimes\cdots\otimes 1_{\beta_n}$ is the idempotent
\[
e_{\beta_1,\dots,\beta_n}
:=\sum_{\bi^{(1)}\in I^{\beta_1},\dots,\bi^{(n)}\in I^{\beta_n}}
e\!\bigl(\bi^{(1)}\cdots \bi^{(n)}\bigr)\ \in R_\beta.
\]
Define
\[
\Ind^{\beta}_{\beta_1,\dots,\beta_n}
:=R_\beta e_{\beta_1,\dots,\beta_n}\otimes_{R_{\beta_1,\dots,\beta_n}}-
\;:\;R_{\beta_1,\dots,\beta_n}\textup{-}\mathrm{mod}\to R_\beta\textup{-}\mathrm{mod}.
\]
Finally, if $M_a\in R_{\beta_a}\textup{-}\mathrm{mod}$ for $a=1,\dots,n$, define the induced product by
\[
M_1\circ\cdots\circ M_n
:=\Ind^{\beta}_{\beta_1,\dots,\beta_n}(M_1\boxtimes\cdots\boxtimes M_n).
\]
\subsection{Universal Specht module}\label{subsec:universal-specht}
In this section we recall the universal graded (row) Specht modules following
\cite{kmr-universal-specht-type-A}. The key point is that the corresponding permutation modules
are most naturally defined \emph{by induction} from one-dimensional segment modules
(\cite[\S3.6]{kmr-universal-specht-type-A}), and then identified with a cyclic
(highest-weight-type) presentation (\cite[\S5.3--\S5.4]{kmr-universal-specht-type-A}).
We then impose the homogeneous Garnir relations to obtain the universal Specht module.

Fix the quiver of type $A_\infty$ (i.e.\ $e=0$) or $A^{(1)}_{e-1}$ (i.e.\ $e>0$), with vertex set
$I=\mathbb Z/e\mathbb Z$. Fix $\Lambda\in P^+$ and
$\bkappa=(\kappa_1,\dots,\kappa_\ell)\in I^\ell$ a charge of $\Lambda$. Fix $\alpha\in Q^+$ and set $d=\height(\alpha)$.
Let $R_\alpha=R_\alpha^\Lambda$ be the cyclotomic KLR algebra of weight $\Lambda$.
For $\blam\in\Par[\bkappa]_\alpha$, let $T^{\blam}$ be the initial tableau and
$\bi^{\blam}:=\bi^{T^{\blam}}\in I^d$ its residue sequence.

\subsubsection{Permutation modules}
For $i\in I$ and $N\in\mathbb{Z}_{>0}$, define the \emph{segment}
\[
  s(i,N):=(i,i+1,\dots,i+N-1)\in I^N,
\]
where entries are taken modulo $e$ if $e>0$. Let $\alpha_{i,N}\in Q^+$ be the positive root associated to this segment, that is,
$\alpha_{i,N}=\alpha\big(s(i,N)\big)$; see \autoref{eq:associated-positive-root}.

Let $s=s(i,N)$ and set $\alpha_s=\alpha_{i,N}$. The \emph{segment module} $M(s)$ is the rank-one graded
$R_{\alpha_s}$-module generated by $m(s)$ in degree $0$, with action
\[
  e(\bj)m(s)=\delta_{\bj,s}\,m(s),\qquad y_r m(s)=0,\qquad \psi_r m(s)=0,
\]
for all admissible $\bj$ and $r$.

Let $\bms=(s(1),\dots,s(n))$ be an ordered tuple of segments. Write $\alpha_r$ for the positive root associated to $s(r)$ and set
$\alpha:=\alpha_1+\cdots+\alpha_n$. Define
\[
  M(\bms):=M(s(1))\circ\cdots\circ M(s(n))
  :=\Ind^{R_\alpha}_{R_{\alpha_1}\otimes\cdots\otimes R_{\alpha_n}}
  \bigl(M(s(1))\boxtimes\cdots\boxtimes M(s(n))\bigr),
\]
with cyclic generator
\[
  m(\bms):=1\otimes m(s(1))\otimes\cdots\otimes m(s(n)).
\]
Let $j(\bms):=s(1)\cdots s(n)\in I^d$ be the concatenation of the segments.

Now fix $\blam\in\Par[\bkappa]_\alpha$. List the nonempty rows of $\blam$ as
$R_1,\dots,R_g$ (from the first component to the $\ell$th component, and within each component from top to bottom).
If $R_a$ has length $N_a$ and the leftmost node of $R_a$ has residue $i_a$, set
\[
  r(a):=s(i_a,N_a).
\]
Define the associated row tuple
\[
  \bmr(\blam):=(r(1),\dots,r(g)).
\]

\begin{Definition}\label{def:row-permutation-module}
Define the \emph{row permutation module}
\[
  M^{\blam}:=M\big(\bmr(\blam)\big)\,\langle \deg(T^{\blam})\rangle,
\]
and write $m^{\blam}$ for its (shifted) cyclic generator.
\end{Definition}
\subsubsection{Homogeneous Garnir relations}
Fix a Garnir node $A\in[\blam]$. Let $\mathcal{B}^A$ be the Garnir belt of $A$, decomposed into
$A$-bricks as in \autoref{subsec:garnir-tableaux}. Let $k=k^A$ be the number of bricks in $\mathcal{B}^A$,
and let $\Sym^A\cong\Sym_k$ be the brick permutation group, with Coxeter generators
$w_1^A,\dots,w_{k-1}^A$ defined as \autoref{eq:brick-transposition} swapping adjacent bricks. Let $f=f^A$ be the number of bricks lying in the
first-row part of $\mathcal B^A$, so that $\Sym_f\times\Sym_{k-f}\le \Sym^A$.
Let $\mathscr D^A$ be the set of minimal-length left coset representatives of
$\Sym_f\times\Sym_{k-f}$ in $\Sym^A$.

Let $T^A$ be the unique maximal row-standard tableau in the Garnir set $\Gar^{A}$. Set
\[
  \bi^A:=\bi^{T^A}\in I^d,\qquad m^A:=\psi^{T^A}m^{\blam}\in M^{\blam}.
\]

Following \cite[\S5.4]{kmr-universal-specht-type-A}, we define the brick operators as follows.

\begin{Definition}\label{def:brick-operators-sigma-tau}
   For $1\le r\le k-1$, set
    \[
      \sigma_r^A:=\psi_{w_r^A}\,e(\bi^A),
      \qquad
      \tau_r^A:=(\sigma_r^A+1)\,e(\bi^A).
    \]
    If $u\in \Sym^A$ and $u=w_{r_1}^A\cdots w_{r_a}^A$ is a chosen reduced expression, define
    \[
      \sigma_u^A:=\sigma_{r_1}^A\cdots\sigma_{r_a}^A,
      \qquad
      \tau_u^A:=\tau_{r_1}^A\cdots\tau_{r_a}^A.
    \]
\end{Definition}

\begin{Lemma}[{\cite[\S5.4]{kmr-universal-specht-type-A}}]\label{lm:tau-well-defined}
    For every $u\in\mathscr D^A$, the element $\tau_u^A$ is independent of the choice of reduced expression of $u$.
\end{Lemma}

\begin{Definition}\label{def:garnir-relation}
    The homogeneous Garnir element associated to $A$ is
    \[
      g^A:=\sum_{u\in\mathscr D^A}\tau_u^A\,\psi^{T^A}\ \in R_\alpha.
    \]
    Equivalently, the corresponding \textbf{Garnir relation} in $M^{\blam}$ may be written as
    \[
      g^A m^{\blam}=\sum_{u\in\mathscr D^A}\tau_u^A\,m^A.
    \]
\end{Definition}

The KLR algebras admit a diagrammatic description; see \cite[Section 2]{khovanovlauda-klr-1}. We briefly introduce this in \autoref{subsec:subdivision-KLR}. We end this section with an example using this diagrammatic description to describe $\sigma_r^A$ and $\psi_{w^{T^A}}$.

\begin{Example}\label{eg:psi-permutation-of-garnir-maximal-tableau}
    Continue with \autoref{eg:garnir-bricks-garnir-sets}, $\sigma^A_{2}:=\psi_{w_r^A}e(\bi^A)$ is the following string diagram:
    \begin{center}
        \resizebox{0.9\linewidth}{!}
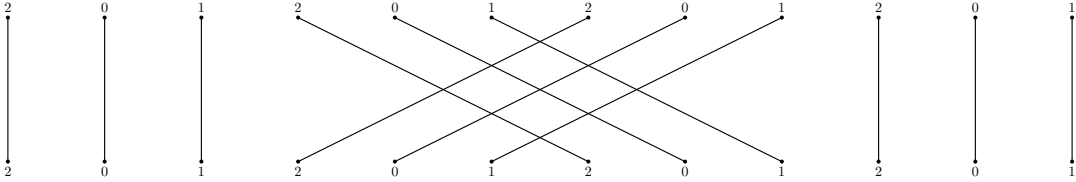
{%
        \begin{tikzpicture}[scale=1.2, thick, line cap=round]
        
        \def\yb{0}
        \def\yt{3}
        
        \coordinate (B26) at (-11,\yb); \coordinate (T26) at (-11,\yt);
        \coordinate (B27) at ( -9,\yb); \coordinate (T27) at ( -9,\yt);
        \coordinate (B28) at ( -7,\yb); \coordinate (T28) at ( -7,\yt);
        \coordinate (B29) at ( -5,\yb); \coordinate (T29) at ( -5,\yt);
        \coordinate (B30) at ( -3,\yb); \coordinate (T30) at ( -3,\yt);
        \coordinate (B31) at ( -1,\yb); \coordinate (T31) at ( -1,\yt);
        \coordinate (B32) at (  1,\yb); \coordinate (T32) at (  1,\yt);
        \coordinate (B33) at (  3,\yb); \coordinate (T33) at (  3,\yt);
        \coordinate (B34) at (  5,\yb); \coordinate (T34) at (  5,\yt);
        \coordinate (B35) at (  7,\yb); \coordinate (T35) at (  7,\yt);
        \coordinate (B36) at (  9,\yb); \coordinate (T36) at (  9,\yt);
        \coordinate (B37) at ( 11,\yb); \coordinate (T37) at ( 11,\yt);
        
        \fill (B26) circle (1.2pt); \fill (T26) circle (1.2pt);
        \node[below] at (B26) {\(2\)}; \node[above] at (T26) {\(2\)};
        
        \fill (B27) circle (1.2pt); \fill (T27) circle (1.2pt);
        \node[below] at (B27) {\(0\)}; \node[above] at (T27) {\(0\)};
        
        \fill (B28) circle (1.2pt); \fill (T28) circle (1.2pt);
        \node[below] at (B28) {\(1\)}; \node[above] at (T28) {\(1\)};
        
        \fill (B29) circle (1.2pt); \fill (T29) circle (1.2pt);
        \node[below] at (B29) {\(2\)}; \node[above] at (T29) {\(2\)};
        
        \fill (B30) circle (1.2pt); \fill (T30) circle (1.2pt);
        \node[below] at (B30) {\(0\)}; \node[above] at (T30) {\(0\)};
        
        \fill (B31) circle (1.2pt); \fill (T31) circle (1.2pt);
        \node[below] at (B31) {\(1\)}; \node[above] at (T31) {\(1\)};
        
        \fill (B32) circle (1.2pt); \fill (T32) circle (1.2pt);
        \node[below] at (B32) {\(2\)}; \node[above] at (T32) {\(2\)};
        
        \fill (B33) circle (1.2pt); \fill (T33) circle (1.2pt);
        \node[below] at (B33) {\(0\)}; \node[above] at (T33) {\(0\)};
        
        \fill (B34) circle (1.2pt); \fill (T34) circle (1.2pt);
        \node[below] at (B34) {\(1\)}; \node[above] at (T34) {\(1\)};
        
        \fill (B35) circle (1.2pt); \fill (T35) circle (1.2pt);
        \node[below] at (B35) {\(2\)}; \node[above] at (T35) {\(2\)};
        
        \fill (B36) circle (1.2pt); \fill (T36) circle (1.2pt);
        \node[below] at (B36) {\(0\)}; \node[above] at (T36) {\(0\)};
        
        \fill (B37) circle (1.2pt); \fill (T37) circle (1.2pt);
        \node[below] at (B37) {\(1\)}; \node[above] at (T37) {\(1\)};
        
        \draw (B26) -- (T26);
        \draw (B27) -- (T27);
        \draw (B28) -- (T28);
        \draw (B35) -- (T35);
        \draw (B36) -- (T36);
        \draw (B37) -- (T37);
        
        \draw (B29) -- (T32);
        \draw (B30) -- (T33);
        \draw (B31) -- (T34);
        
        \draw (B32) -- (T29);
        \draw (B33) -- (T30);
        \draw (B34) -- (T31);
        
        \end{tikzpicture}%
        }
    \end{center}
    These strands correspond to the entries from $26$ to $31$ in $T^{\blam}$, while all undrawn strands are taken to be vertical straight strands. Similarly, $\psi_{w^{T^A}}$ is as follows:

    \begin{center}
    \resizebox{0.9\linewidth}{!}
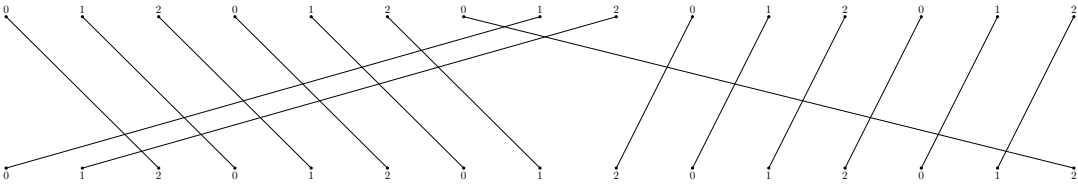
{%
        \begin{tikzpicture}[scale=1.2, thick, line cap=round]
        
            \def\yb{0}
            \def\yt{4}
            
            \coordinate (B24) at (-14,\yb); \coordinate (T24) at (-14,\yt);
            \coordinate (B25) at (-12,\yb); \coordinate (T25) at (-12,\yt);
            \coordinate (B26) at (-10,\yb); \coordinate (T26) at (-10,\yt);
            \coordinate (B27) at ( -8,\yb); \coordinate (T27) at ( -8,\yt);
            \coordinate (B28) at ( -6,\yb); \coordinate (T28) at ( -6,\yt);
            \coordinate (B29) at ( -4,\yb); \coordinate (T29) at ( -4,\yt);
            \coordinate (B30) at ( -2,\yb); \coordinate (T30) at ( -2,\yt);
            \coordinate (B31) at (  0,\yb); \coordinate (T31) at (  0,\yt);
            \coordinate (B32) at (  2,\yb); \coordinate (T32) at (  2,\yt);
            \coordinate (B33) at (  4,\yb); \coordinate (T33) at (  4,\yt);
            \coordinate (B34) at (  6,\yb); \coordinate (T34) at (  6,\yt);
            \coordinate (B35) at (  8,\yb); \coordinate (T35) at (  8,\yt);
            \coordinate (B36) at ( 10,\yb); \coordinate (T36) at ( 10,\yt);
            \coordinate (B37) at ( 12,\yb); \coordinate (T37) at ( 12,\yt);
            \coordinate (B38) at ( 14,\yb); \coordinate (T38) at ( 14,\yt);
            
            \fill (B24) circle (1.2pt); \fill (T24) circle (1.2pt);
            \node[below] at (B24) {\(0\)}; \node[above] at (T24) {\(0\)};
            
            \fill (B25) circle (1.2pt); \fill (T25) circle (1.2pt);
            \node[below] at (B25) {\(1\)}; \node[above] at (T25) {\(1\)};
            
            \fill (B26) circle (1.2pt); \fill (T26) circle (1.2pt);
            \node[below] at (B26) {\(2\)}; \node[above] at (T26) {\(2\)};
            
            \fill (B27) circle (1.2pt); \fill (T27) circle (1.2pt);
            \node[below] at (B27) {\(0\)}; \node[above] at (T27) {\(0\)};
            
            \fill (B28) circle (1.2pt); \fill (T28) circle (1.2pt);
            \node[below] at (B28) {\(1\)}; \node[above] at (T28) {\(1\)};
            
            \fill (B29) circle (1.2pt); \fill (T29) circle (1.2pt);
            \node[below] at (B29) {\(2\)}; \node[above] at (T29) {\(2\)};
            
            \fill (B30) circle (1.2pt); \fill (T30) circle (1.2pt);
            \node[below] at (B30) {\(0\)}; \node[above] at (T30) {\(0\)};
            
            \fill (B31) circle (1.2pt); \fill (T31) circle (1.2pt);
            \node[below] at (B31) {\(1\)}; \node[above] at (T31) {\(1\)};
            
            \fill (B32) circle (1.2pt); \fill (T32) circle (1.2pt);
            \node[below] at (B32) {\(2\)}; \node[above] at (T32) {\(2\)};
            
            \fill (B33) circle (1.2pt); \fill (T33) circle (1.2pt);
            \node[below] at (B33) {\(0\)}; \node[above] at (T33) {\(0\)};
            
            \fill (B34) circle (1.2pt); \fill (T34) circle (1.2pt);
            \node[below] at (B34) {\(1\)}; \node[above] at (T34) {\(1\)};
            
            \fill (B35) circle (1.2pt); \fill (T35) circle (1.2pt);
            \node[below] at (B35) {\(2\)}; \node[above] at (T35) {\(2\)};
            
            \fill (B36) circle (1.2pt); \fill (T36) circle (1.2pt);
            \node[below] at (B36) {\(0\)}; \node[above] at (T36) {\(0\)};
            
            \fill (B37) circle (1.2pt); \fill (T37) circle (1.2pt);
            \node[below] at (B37) {\(1\)}; \node[above] at (T37) {\(1\)};
            
            \fill (B38) circle (1.2pt); \fill (T38) circle (1.2pt);
            \node[below] at (B38) {\(2\)}; \node[above] at (T38) {\(2\)};
            
            \draw (B24) -- (T31);
            \draw (B25) -- (T32);
            
            \draw (B26) -- (T24);
            \draw (B27) -- (T25);
            \draw (B28) -- (T26);
            \draw (B29) -- (T27);
            \draw (B30) -- (T28);
            \draw (B31) -- (T29);
            
            \draw (B32) -- (T33);
            \draw (B33) -- (T34);
            \draw (B34) -- (T35);
            \draw (B35) -- (T36);
            \draw (B36) -- (T37);
            \draw (B37) -- (T38);
            
            \draw (B38) -- (T30);
        
        \end{tikzpicture}%
    }
    \end{center}
    These strands correspond to the entries from $24$ to $38$ in $T^{\blam}$ and $T^{A}$, while, as before, all undrawn strands are taken to be vertical straight strands.
\end{Example}
\subsubsection{Universal Specht modules}
We define the universal graded Specht modules via a highest-weight-module presentation:
\begin{Definition}\label{dfn-universal-specht}
    Let $\blam\in\Par[\bkappa]_\alpha$ and $d=\height(\alpha)$.
    The \emph{universal graded (row) Specht module} $S^{\blam}$ is the graded $R_\alpha$-module generated by
    $z^{\blam}$, which is homogeneous of degree $\deg(T^{\blam})$, subject to the relations
    \begin{enumerate}[label=(\alph*)]
        \item $e(\bj)z^{\blam}=\delta_{\bj,\bi^{\blam}}\,z^{\blam}$ for all $\bj\in I^{\alpha}$;
        \label{rel:idempotent}
    
        \item $y_r z^{\blam}=0$ for $r=1,\dots,d$;
        \label{rel:nilpotent}
    
        \item $\psi_r z^{\blam}=0$ whenever $r\rightarrow_{T^{\blam}}r+1$;
        \label{rel:psi}
    
        \item $g^A z^{\blam}=0$ for every Garnir node $A\in[\blam]$.
        \label{rel:garnir}
    \end{enumerate}
\end{Definition}

\begin{Remark}\label{rmk:highest-weight-presentaiton-of-permutation-module}
    Let $M$ be the cyclic module generated by $m^{\blam}$ subject only to the relations
    (1)--(3) in Definition~\ref{dfn-universal-specht}. Then $M\cong M^{\blam}$ as graded $R_{\alpha}$-modules and there is a natural surjection
    $M^{\blam}\twoheadrightarrow S^{\blam}$ sending $m^{\blam}\mapsto z^{\blam}$, and $S^{\blam}\ \cong\ M^{\blam}\Big/\Big\langle\, g^A m^{\blam}\ \Big|\ \forall A\in[\blam]\text{ is a Garnir node}\,\Big\rangle$.
\end{Remark}

In later sections, we call $m^{\blam}$ and $z^{\blam}$ the \emph{standard cyclic generators} of the modules $M^{\blam}$ and $S^{\blam}$, respectively.

\subsubsection{Bases}
For completeness, we record the standard basis results.

\begin{Theorem}[{\cite[Theorem 5.6]{kmr-universal-specht-type-A}}]\label{thm:basis-permutation}
    The permutation module $M^{\blam}$ has a $\bk$-basis 
    \[
        \{\,\psi^T m^{\blam}\mid T\in\RST{\blam}\,\}
    \]
\end{Theorem}

\begin{Corollary}\label{cor:iso-of-permutation-module}
    Let $\ell,\ell'$ be two positive integers, and take two charges $\bkappa\in I^{\ell}$ and $\bkappa'\in I^{\ell'}$. Let $\blam\in\Par[\bkappa]_{\alpha}$ and $\bmu\in\Par[\bkappa']_{\alpha}$. Assume that $\blam$ and $\bmu$ have the same
    ordered list of row segments (i.e. $\bmr(\blam)=\bmr(\bmu)$). Then there is a canonical graded
    $R_\alpha$-module isomorphism
    \[
        \Phi:\ M^{\blam}\xrightarrow{\ \sim\ } M^{\bmu},\qquad \Phi(r\,m^{\blam})=r\,m^{\bmu}\ \ (r\in R_\alpha).
    \]
    In particular, splitting any component into several consecutive components by cutting off rows (and preserving the row order) does not change the isomorphism class of the corresponding permutation module.
\end{Corollary}

\begin{Theorem}[{\cite[Corollary 6.24]{kmr-universal-specht-type-A}}]\label{thm:basis-specht}
    There is a homogeneous degree $0$ isomorphism between $S^{\blam}$ and the graded cell module of
    $R^\Lambda_\alpha$ constructed in \textup{\cite{humathas-graded-cellular}}. Moreover,
    $S^{\blam}$ has a $\bk$-basis $\{\,\psi^T z^{\blam}\mid T\in\Std(\blam)\,\}$.
\end{Theorem}

\begin{Theorem}[{\cite[Theorem 8.2]{kmr-universal-specht-type-A}}]\label{thm:specht-factorisation}
    Suppose that $\blam=(\blam^{(1)},\dots,\blam^{(\ell)})\in\Par[\bkappa]_{\alpha}$. Then
    \begin{equation}\label{eq:specht-factorisation}
    S^{\blam}\ \cong\ S^{\blam^{(1)}}\circ\cdots\circ S^{\blam^{(\ell)}}\langle d_{\blam}\rangle,
    \end{equation}
    as graded $R_{\alpha}$-modules, where
    $d_{\blam}:=\ \deg(T^{\blam})-\deg(T^{\blam^{(1)}})-\cdots-\deg(T^{\blam^{(\ell)}})$. 
\end{Theorem}

\section{Algebraic Subdivision}\label{sec:algebraic-subdivision}
In this section, we introduce the subdivision maps for KLR algebras following \cite{maksimau-subdivision}, with minor notational changes adapted to our combinatorial framework.

Since we will define subdivision maps on several kinds of objects---including quivers, positive roots, words, dominant weights, and (cyclotomic) KLR algebras---we will, by abuse of notation, use the same symbol $\Phi$ for all of them. When we apply these maps later, it will be clear from the objects involved which version of $\Phi$ is intended.
We will, however, verify along the way that these definitions are compatible with one another.
\subsection{Subdivision on quiver}\label{subsec:subdivision-quiver}
Assume that the quiver $\Gamma=\Aone[e-1]$ has vertex set $I=\Z/e\Z$, which we identify with $\{0,1,\ldots,e-1\}$ throughout the paper.
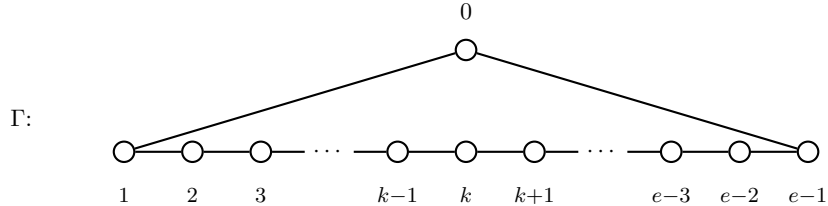
\begin{center}
\begin{tikzpicture}[
    scale=0.9,
    every node/.style={transform shape},
    dynnode/.style={circle, draw, thick, minimum size=3mm, inner sep=0pt, fill=white},
    label_style/.style={yshift=-2mm, font=\small}
]

    \node[dynnode, label={[yshift=1mm]above:$0$}] (n0) at (5, 1.5) {};


    \node[dynnode, label={[label_style]below:$1$}] (n1) at (0,0) {};
    \node[dynnode, label={[label_style]below:$2$}] (n2) at (1,0) {};
    \node[dynnode, label={[label_style]below:$3$}] (n3) at (2,0) {};

    \node (dots1) at (3,0) {$\cdots$};

    \node[dynnode, label={[label_style]below:$k{-}1$}] (nk_1) at (4,0) {};
    \node[dynnode, label={[label_style]below:$k$}] (nk) at (5,0) {};
    \node[dynnode, label={[label_style]below:$k{+}1$}] (nk_plus1) at (6,0) {};

    \node (dots2) at (7,0) {$\cdots$};

    \node[dynnode, label={[label_style]below:$e{-}3$}] (ne_3) at (8,0) {};
    \node[dynnode, label={[label_style]below:$e{-}2$}] (ne_2) at (9,0) {};
    \node[dynnode, label={[label_style]below:$e{-}1$}] (ne_1) at (10,0) {};

    
    \draw[thick] (n1) -- (n2) -- (n3) -- (dots1) -- (nk_1) -- (nk) -- (nk_plus1) -- (dots2) -- (ne_3) -- (ne_2) -- (ne_1);

    \draw[thick] (n0) -- (n1);
    \draw[thick] (n0) -- (ne_1);

    \node at (-1.5, 0.5) {$\Gamma$:};

\end{tikzpicture}
\end{center}
Fix an integer $k\in I$ and identify $e$ with $0$. The subdivision at the edge $k\to k+1$ replaces this single edge by the two-step path
$k\to \overline{k}\to k+1$, while leaving all other edges unchanged. We denote the resulting subdivision map by $\Phi_k$, and write $\overline{\Gamma}:=\Phi_k(\Gamma)$ for the resulting quiver:
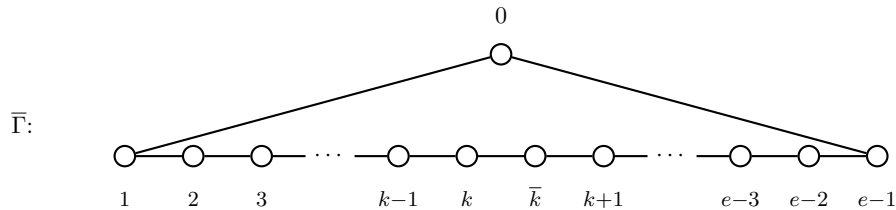
\begin{center}
\begin{tikzpicture}[
    scale=0.9, 
    every node/.style={transform shape},
    dynnode/.style={circle, draw, thick, minimum size=3mm, inner sep=0pt, fill=white},
    label_style/.style={yshift=-2mm, font=\small}
]

    \node[dynnode, label={[yshift=1mm]above:$0$}] (n0) at (5.5, 1.5) {};


    \node[dynnode, label={[label_style]below:$1$}] (n1) at (0,0) {};
    \node[dynnode, label={[label_style]below:$2$}] (n2) at (1,0) {};
    \node[dynnode, label={[label_style]below:$3$}] (n3) at (2,0) {};

    \node (dots1) at (3,0) {$\cdots$};

    \node[dynnode, label={[label_style]below:$k{-}1$}] (nk_1) at (4,0) {};
    \node[dynnode, label={[label_style]below:$k$}] (nk) at (5,0) {};
    \node[dynnode, label={[label_style]below:$\overline{k}$}] (nk_bar) at (6,0) {};
    \node[dynnode, label={[label_style]below:$k{+}1$}] (nk_plus1) at (7,0) {};

    \node (dots2) at (8,0) {$\cdots$};

    \node[dynnode, label={[label_style]below:$e{-}3$}] (ne_3) at (9,0) {};
    \node[dynnode, label={[label_style]below:$e{-}2$}] (ne_2) at (10,0) {};
    \node[dynnode, label={[label_style]below:$e{-}1$}] (ne_1) at (11,0) {};

    
    \draw[thick] (n1) -- (n2) -- (n3) -- (dots1) -- (nk_1) -- (nk) -- (nk_bar) -- (nk_plus1) -- (dots2) -- (ne_3) -- (ne_2) -- (ne_1);

    \draw[thick] (n0) -- (n1);
    \draw[thick] (n0) -- (ne_1);

    \node at (-1.5, 0.5) {$\overline{\Gamma}$:};

\end{tikzpicture}
\end{center}
As we will apply combinatorial arguments later, it is convenient to label the new quiver
$\overline{\Gamma}$, which is a cyclic quiver with $e+1$ vertices, using the standard
convention for type $\Aone[e]$. We therefore relabel the vertices as follows:
\begin{equation}\label{eq:relabel-rules}
\begin{aligned}
    &\bullet \text{ keep the label } i \text{ unchanged for } 0 \le i \le k; \\
    &\bullet \text{ replace the label } i \text{ by } i+1 \text{ for } k+1 \le i \le e-1; \\
    &\bullet \text{ replace the label } \overline{k} \text{ by } k+1.
\end{aligned}
\end{equation}
With this convention, the quiver $\overline{\Gamma}$ becomes:

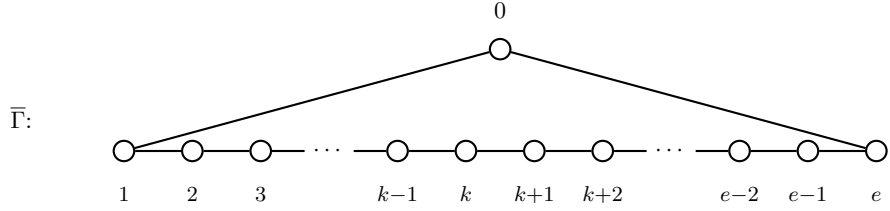
\begin{center}
\begin{tikzpicture}[
    scale=0.9,
    every node/.style={transform shape},
    dynnode/.style={circle, draw, thick, minimum size=3mm, inner sep=0pt, fill=white},
    label_style/.style={yshift=-2mm, font=\small}
]

    \node[dynnode, label={[yshift=1mm]above:$0$}] (n0) at (5.5, 1.5) {};


    \node[dynnode, label={[label_style]below:$1$}] (n1) at (0,0) {};
    \node[dynnode, label={[label_style]below:$2$}] (n2) at (1,0) {};
    \node[dynnode, label={[label_style]below:$3$}] (n3) at (2,0) {};

    \node (dots1) at (3,0) {$\cdots$};

    \node[dynnode, label={[label_style]below:$k{-}1$}] (nk_1) at (4,0) {};
    \node[dynnode, label={[label_style]below:$k$}] (nk) at (5,0) {};
    \node[dynnode, label={[label_style]below:$k{+}1$}] (nk_plus1) at (6,0) {};
    \node[dynnode, label={[label_style]below:$k{+}2$}] (nk_plus2) at (7,0) {};

    \node (dots2) at (8,0) {$\cdots$};

    \node[dynnode, label={[label_style]below:$e{-}2$}] (ne_2) at (9,0) {};
    \node[dynnode, label={[label_style]below:$e{-}1$}] (ne_1) at (10,0) {};
    \node[dynnode, label={[label_style]below:$e$}] (ne) at (11,0) {};

    
    \draw[thick] (n1) -- (n2) -- (n3) -- (dots1) -- (nk_1) -- (nk) -- (nk_plus1) -- (nk_plus2) -- (dots2) -- (ne_2) -- (ne_1) -- (ne);

    \draw[thick] (n0) -- (n1);
    \draw[thick] (n0) -- (ne);

    \node at (-1.5, 0.5) {$\overline{\Gamma}$:};

\end{tikzpicture}
\end{center}
By abuse of notation, we will continue to write this relabelled quiver as $\overline{\Gamma}$. The vertex set of $\oGamma$ is $\widebar{I}=\Z/(e+1)Z$ which is identified with $\{0,1,\cdots,e\}$.



\begin{Remark}
In \cite{maksimau-subdivision}, there is no relabelling step: the new vertex is kept as a distinguished symbol
$\overline{k}$.  Our relabelling identifies this new vertex with $k+1\in\overline I$ and shifts the old labels
$k+1,\dots,e-1$ up by one, so the two conventions differ only by this canonical relabelling.
\end{Remark}
\subsection{Subdivision on positive roots}\label{subsec:subdivision-root}
Let $Q^+=Q^+_I$ be the positive root lattice of type $\Aone[e-1]$, with simple roots
$\alpha_0,\dots,\alpha_{e-1}$. For $\beta\in Q^+$ write
\[
      \beta=\sum_{i=0}^{e-1} x_i\alpha_i \qquad\text{with } x_i\in\mathbb{Z}_{\ge 0}.
\]
For $k\in I=\{0,1,\cdots,e-1\}$, the subdivision map $\Phi_k$ on $Q^+$ at the edge $k\to k+1$ is defined by
\begin{equation}\label{eq:Phi-k-root}
      \Phi_k(\beta)
      \;=\;
      \sum_{i=0}^{k-1} x_i\alpha_i
      \;+\;
      x_k(\alpha_k+\alpha_{k+1})
      \;+\;
      \sum_{i=k+1}^{e-1} x_i\alpha_{i+1}.
\end{equation}
This is compatible with the subdivision map on quivers described in \autoref{subsec:subdivision-quiver}. Indeed, for $0\le i\le k-1$ we have $\Phi_k(\alpha_i)=\alpha_i$, for $k+1\le i\le e-1$ we have $\Phi_k(\alpha_i)=\alpha_{i+1}$, and $\Phi_k(\alpha_k)=\alpha_k+\alpha_{k+1}$.

The map $\Phi_k$ yields a bijection
\begin{equation}\label{eq:subdivision-bijection-roots}
    \Phi_k: Q^{+}_{I}\to \Bigl\{\beta=\sum_{i\in \widebar{I}}x_i\alpha_i\in Q^+_{\widebar{I}} \,\Bigm|\, x_k=x_{k+1}\Bigr\}, 
    \qquad \alpha\mapsto\Phi_k(\alpha).
\end{equation}
If there is no ambiguity in $k$, we write
$\overline{\beta}=\Phi(\beta)=\Phi_k(\beta)$.
\subsection{Subdivision on words}\label{subsec:subdivision-words}
Fix quiver type $\Aone[e-1]$ and let $I$ be its vertex set. Let $\bi\in I^\alpha$ with $n=\operatorname{ht}(\alpha)$, and fix $k\in I$. Let $\bar{I}$ be the vertex set of the quiver $\Aone[e]$, which is the image of $\Aone[e-1]$ under the subdivision map constructed in \autoref{subsec:subdivision-quiver}. Define $\Phi_k(\bi)$ to be the sequence with entries in $\bar{I}$ obtained from $\bi=(\bi_1,\ldots,\bi_n)$ by:
\begin{itemize}
    \item keeping $\bi_j$ unchanged if $0\leq \bi_j<k$;
    \item replacing $\bi_j$ by $\bi_j+1$ if $k<\bi_j\le e-1$;
    \item replacing each occurrence of $\bi_j=k$ by two consecutive entries $k,k+1$.
\end{itemize}
\begin{Example}\label{eg:subdivision-words}
    Fix type $\Aone[2]$ and $k=1$, let $\bi=(1200122012010)$ from \autoref{eg:multi-tableau-residue-sequence}, then $\Phi_k(\bi)=(12300123301230120)$.
\end{Example}

This definition is compatible with the subdivision map on positive roots in the following sense. Recall, for each word $\bi$, we can define the associated positive root $\alpha(\bi)$ by \autoref{eq:associated-positive-root}, then we have 
\[
    \Phi_k\big(\alpha(\bi)\big)=\alpha\big(\Phi_k(\bi)\big)
\]
More explicitly, write $\alpha(\bi)=\sum_{j=1}^n \alpha_{\bi_j}\in Q^+$ and let $m$ be the number of entries equal to $k$ in $\bi$. Then
\begin{equation}\label{eq:subdivision-residue-sequence-root}
    \sum_{j=1}^{n+m} \alpha_{(\Phi_k(\bi))_j} \;=\; \Phi_k\big(\alpha(\bi)\big).
\end{equation}

Similarly to \autoref{eq:subdivision-bijection-roots}, $\Phi_k$ gives the following bijection
\begin{equation}\label{eq:subdivision-bijection-words}
    \Phi_k: I^{\alpha}\to \{\bj\in I^{\Phi_k(\alpha)}\mid \text{$\bj_t=k$ if and only if $t<N$ and $\bj_{t+1}=k+1$}\},\quad \text{where }N=\height\big(\Phi_k(\alpha)\big)
\end{equation}

To trace the position of a letter in a word after applying subdivision, we introduce the following \emph{position-tracing function}.
\begin{Definition}\label{def:position-tracing-function}
    Fix $k\in I$ and let $\bi=(i_1,\dots,i_n)\in I_\alpha$. Write $m:=\#\{t\mid i_t=k\}$.
    The position-tracing function associated to $(k,\bi)$ is the strictly increasing map
    \[
        \phi_{\bi}:\{1,\dots,n\}\longrightarrow \{1,\dots,n+m\}
    \]
    defined by
    \[
        \phi_{\bi}(t):=t+\#\{\,1\le j<t\mid i_j=k\,\}.
    \]
    Equivalently, $\phi_{\bi}(1)=1$ and $\phi_{\bi}(t+1)=\phi_{\bi}(t)+1+\delta_{i_t,k}\qquad(1\le t<n)$.
\end{Definition}

The following result is immediate from the definition.

\begin{Lemma}
    With the same notations as in \autoref{def:position-tracing-function}, 
    \[
    \big(\Phi_k(\bi)\big)_{\phi_{\bi}(t)}=
    \begin{cases}
        \bi_t & \text{if }0\leq \bi_t< k,\\
        k & \text{if }\bi_t= k,\\
        \bi_t+1 & \text{if }k+1\leq \bi_t\leq e-1.
    \end{cases}
    \]
    Moreover, if $\big(\Phi_k(\bi)\big)_{\phi_{\bi}(t)}=k$, then $\big(\Phi_k(\bi)\big)_{\phi_{\bi}(t)+1}=k+1$.\hfill $\square$
\end{Lemma}

\subsection{Subdivision on dominant weights}\label{subsec:subdivision-weight}
Let $P^+$ be the (integral) dominant weight lattice of type $\Aone[e-1]$, and let $I$ be the vertex set of the corresponding $\Aone[e-1]$ quiver. We define the subdivision map on fundamental weights by
\begin{equation}\label{eq:def-subdivision-weight}
    \Phi \Lambda_i =
        \begin{cases}
            \Lambda_i, & \text{if } 0 \leq i \leq k, \\[1mm]
            \Lambda_{i+1}, & \text{if } i > k.
        \end{cases}
\end{equation}
For an arbitrary (integral dominant) weight $\Lambda=\sum_{i=0}^{e-1} x_i \Lambda_i$ with $x_i\geq 0$ for all $i$, we extend this map linearly by
\[
\Phi(\Lambda):=\sum_{i=0}^{e-1} x_i \Phi(\Lambda_i).
\]
In type $\Aone[e-1]$ (or type $\Ainfty$), this map preserves the \emph{level} of the dominant weight. The compatibility of this definition with the ones in \autoref{subsec:subdivision-root} is shown in the following results:
\begin{Lemma}\label{lm:cyclotomic-multiplicity}
    Fix $\alpha \in Q^+$, $\Lambda \in P^+$, and $k\in I$. Let $\bi \in I^\alpha$, and let $\Phi = \Phi_k$ be the subdivision map on positive roots or dominant weights. Then we have:
    \[
        \big( \Phi(\Lambda)\mid \alpha_{\Phi(\bi)_1} \big) = \big( \Lambda\mid \alpha_{\bi_1} \big).
    \]
\end{Lemma}
\begin{proof}
    It suffices to verify the equality for any fundamental weight $\Lambda_j$; the general case follows by linearity. Recall that $\left( \Lambda_j \mid \alpha_{\bi_1} \right) = \delta_{j, \bi_1}$. Set $\oLambda := \Phi(\Lambda)$ and $\overline{\bi} := \Phi(\bi)$. Then
    \[
        \left( \oLambda \mid \alpha_{\overline{\bi}_1} \right) =
        \begin{cases}
            \left( \Lambda_j \mid \alpha_{\bi_1} \right)= \delta_{j, \bi_1}& \text{if } j \leq k,\, \bi_1 \leq k,\\
            \left( \Lambda_j \mid \alpha_{\bi_1+1} \right) =0& \text{if } j \leq k,\, \bi_1 > k,\\
            \left( \Lambda_{j+1} \mid \alpha_{\bi_1+1} \right)=\delta_{j+1, \bi_1+1} & \text{if } j > k,\, \bi_1 > k,\\
            \left( \Lambda_{j+1} \mid \alpha_{\bi_1} \right)=0 & \text{if } j > k,\, \bi_1 \leq k.
        \end{cases}\quad
         = \delta_{j, \bi_1}
    \]
    Thus, the desired equality holds in all cases.
\end{proof}
\begin{Lemma}\label{lm:pairing-invariance}
    Fix $\alpha \in Q^+$, $\Lambda \in P^+$, and $k\in I$. Let $\Phi=\Phi_k$ be the subdivision map on positive roots or dominant weights. Then the pairing is preserved:
    \[
        (\Phi(\Lambda) \mid \Phi(\alpha)) = (\Lambda \mid \alpha).
    \]
\end{Lemma}
\begin{proof}
    By linearity, it suffices to verify the equality for fundamental weights $\Lambda_j$ and simple roots $\alpha_i$, i.e., to show $(\Phi(\Lambda_j) \mid \Phi(\alpha_i)) = \delta_{j,i}$. 
    We verify the cases based on the index $i$:
    \begin{enumerate}
        \item Case $i < k$: Then $\Phi(\alpha_i) = \alpha_i$. If $j \leq k$, then $\Phi(\Lambda_j) = \Lambda_j$ and the pairing is $\delta_{j,i}$. If $j > k$, then $\Phi(\Lambda_j) = \Lambda_{j{+}1}$, and $(\Lambda_{j{+}1} \mid \alpha_i) = 0$ (since $j{+}1 > i$), matching $\delta_{j,i}=0$.

        \item Case $i = k$: Then $\Phi(\alpha_k) = \alpha_k + \alpha_{k{+}1}$.
        \begin{itemize}
            \item If $j < k$: $(\Lambda_j \mid \alpha_k + \alpha_{k{+}1}) = 0$.
            \item If $j = k$: $(\Lambda_k \mid \alpha_k + \alpha_{k{+}1}) = 1 + 0 = 1$.
            \item If $j > k$: $(\Lambda_{j{+}1} \mid \alpha_k + \alpha_{k{+}1}) = 0$.
        \end{itemize}
        In all subcases, the result matches $\delta_{j,k}$.

        \item Case $i > k$: Then $\Phi(\alpha_i) = \alpha_{i{+}1}$. If $j \leq k$, $(\Lambda_j \mid \alpha_{i{+}1}) = 0$. If $j > k$, $(\Lambda_{j{+}1} \mid \alpha_{i{+}1}) = \delta_{j{+}1, i{+}1} = \delta_{j,i}$.
    \end{enumerate}
    Thus, the equality holds for all basis elements.
\end{proof}
If there is no ambiguity in $k$, we write
$\overline{\Lambda}=\Phi(\Lambda)=\Phi_k(\Lambda)$
without further explanation.
\subsection{Subdivision on KLR algebras}\label{subsec:subdivision-KLR}
Following \cite{maksimau-subdivision}, one may define the subdivision map on KLR algebras. In this section, we give a diagrammatic definition of the map (see \cite[Remark 2.13]{maksimau-subdivision} and \cite[Section 4]{mathastubbenhauer-klrw-ac}), which we think is more intuitive and clearer. For the algebraic description, see \cite[Section 2]{maksimau-subdivision}.

Fix a quiver $\Gamma$ of type $\Aone[e-1]$ with vertex set $I$, and let $\alpha\in Q^+$ have height $n$. Recall that the KLR algebra $R_\alpha$ is generated by elements subject to certain relations; see \autoref{def:klr-algebras}. Concretely, it is generated by
$e(\bi)$, $\psi_j e(\bi)$, and $y_i e(\bi)$, where $1\le i\le n$, $1\le j\le n-1$, and $\bi\in I^\alpha$.
These generators admit a diagrammatic interpretation as in \cite{khovanovlauda-klr-1}:
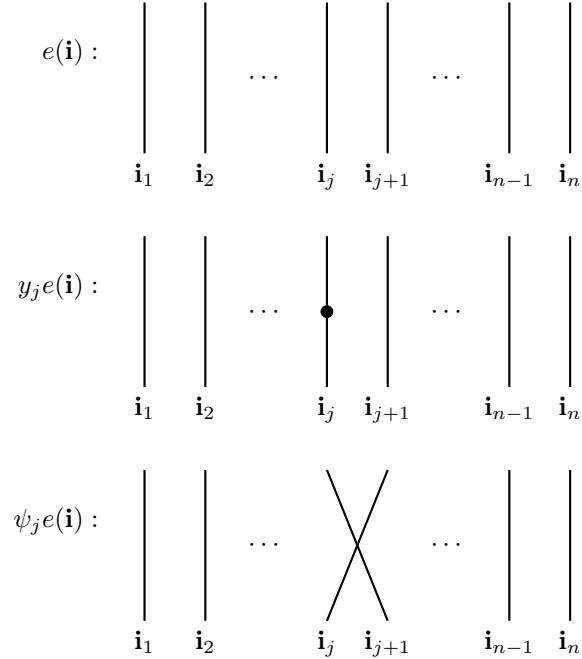
\begin{align*}
    e(\bi): &\quad 
    \begin{tikzpicture}[baseline={(0,1.25)}, xscale=0.8, yscale=0.8, thick]
        \def\h{2.5}
        \foreach \x/\lab in {1/1, 2/2, 4/j, 5/{j+1}, 7/{n-1}, 8/n} {
            \draw (\x,0) -- (\x,\h);
            \node[below] at (\x,0) {$\bi_{\lab}$};
        }
        \node at (3, \h/2) {$\cdots$};
        \node at (6, \h/2) {$\cdots$};
    \end{tikzpicture} 
    \\[1em] 
    y_je(\bi): &\quad 
    \begin{tikzpicture}[baseline={(0,1.25)}, xscale=0.8, yscale=0.8, thick]
        \def\h{2.5}
        \foreach \x/\lab in {1/1, 2/2, 4/j, 5/{j+1}, 7/{n-1}, 8/n} {
            \draw (\x,0) -- (\x,\h);
            \node[below] at (\x,0) {$\bi_{\lab}$};
        }
        \fill (4,\h/2) circle (3pt);
        \node at (3, \h/2) {$\cdots$};
        \node at (6, \h/2) {$\cdots$};
    \end{tikzpicture} 
    \\[1em]
    \psi_je(\bi): &\quad 
    \begin{tikzpicture}[baseline={(0,1.25)}, xscale=0.8, yscale=0.8, thick]
        \def\h{2.5}
        \foreach \x/\lab in {1/1, 2/2, 7/{n-1}, 8/n} {
            \draw (\x,0) -- (\x,\h);
            \node[below] at (\x,0) {$\bi_{\lab}$};
        }
        \draw (4,0) -- (5,\h);
        \node[below] at (4,0) {$\bi_j$};
        \draw (5,0) -- (4,\h);
        \node[below] at (5,0) {$\bi_{j+1}$};
        \node at (3, \h/2) {$\cdots$};
        \node at (6, \h/2) {$\cdots$};
    \end{tikzpicture}
\end{align*}

In this string-diagrammatic presentation, a string labeled by $i\in I$ is called an $i$-string. The multiplication is given by vertical concatenation of diagrams.

\begin{Example}\label{eg:stringdiagram}
The following string diagram represents the element $\psi_1\psi_4\psi_3\psi_4 y_2e(01212)$: 

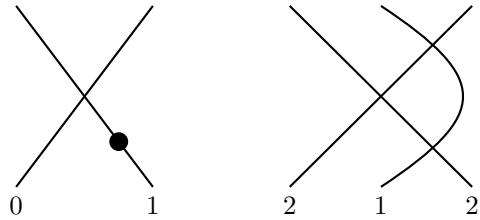
\begin{center}
\begin{tikzpicture}[scale=1.2, thick]
    \draw (0,0) -- (1.5,2);
    \node[below] at (0,0) {0};
    \draw (1.5,0) -- (0,2) node[pos=0.25, circle, fill, inner sep=2.5pt] {};
    \node[below] at (1.5,0) {1};
    \begin{scope}[shift={(3,0)}]
        \draw (0,0) -- (2,2);
        \node[below] at (0,0) {2};
        \draw (1,0) .. controls (2.2, 0.8) and (2.2, 1.2) .. (1,2); 
        \node[below] at (1,0) {1};
        \draw (2,0) -- (0,2);
        \node[below] at (2,0) {2};
    \end{scope}
\end{tikzpicture}
\end{center}

\end{Example}

Let $\overline{\Gamma}=\Phi_k(\Gamma)$ and $\overline{\alpha}=\Phi_k(\alpha)$. 
We construct a subdivision map $\Phi=\Phi_k$ from the KLR algebra $R_\alpha=R_\alpha(\Gamma)$ over $\Gamma=\Aone[e-1]$ to the KLR algebra $R_{\overline{\alpha}}=R_{\overline{\alpha}}(\overline{\Gamma})$ over $\overline{\Gamma}=\Aone[e]$. This map will be compatible with the subdivision maps on the quiver, positive roots, and dominant weights defined above.

It follows from \autoref{thm:base-of-klr} that the elements $\psi_w y_1^{a_1}\cdots y_n^{a_n} e(\bi)$, where $a_i\ge 0$, $w\in\Sym_n$, and $\bi\in I^\alpha$, form a basis of $R_\alpha$. We first define a map $\Theta_k$ on these basis elements:

\begin{enumerate}
      \item The element $\Theta_k\bigl(e(\bi)\bigr)$ is obtained by adding a new string labeled $\overline{k}$ immediately to the right of each $k$-string.
      \item  The newly added $\overline{k}$-string carries no dots even if the $k$-string carries dots. 
      \item For a string diagram of the form $\psi_w\,e(\bi)$, first add the new strings as in (1), and then extend each new string so that it follows the shape of the corresponding $k$-string.
      \item Relabel the strings as follows: replace $i$ by $i+1$ for $k+1\le i\le e-1$, replace $\overline{k}$ by $k+1$, and leave all other labels unchanged.
\end{enumerate}

\begin{Example}\label{eg:subdivision-stringdiagram}
The image of the string diagram from \autoref{eg:stringdiagram} under subdivision $\Phi_1$ is the following string diagram:
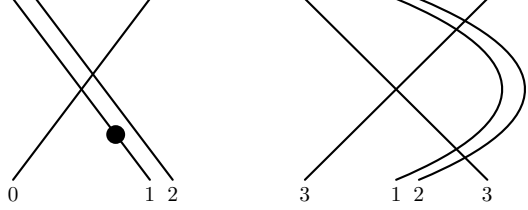
\begin{center}
    \begin{tikzpicture}[scale=1.2, thick, line cap=round, line join=round]
        \draw (0,0) -- (1.5,2);
        \node[below, scale=0.8] at (0,0) {0};
        \draw (1.5,0) -- (0,2)
          node[pos=0.25, circle, fill, inner sep=2.5pt] {};
        \node[below, scale=0.8] at (1.5,0) {1};
        \draw (1.75,0) -- (0.25,2);
        \node[below, scale=0.8] at (1.75,0) {2};
        \begin{scope}[shift={(3.2,0)}]
            \draw (0,0) -- (2,2);
            \node[below, scale=0.8] at (0,0) {3};
    
            \draw (2,0) -- (0,2);
            \node[below, scale=0.8] at (2,0) {3};
            \def\bulge{1.55} 
            \def\yA{0.65}
            \def\yB{1.35}
            
            \foreach \x/\lab in {1.00/1, 1.25/2} {
              \draw (\x,0) .. controls (\x+\bulge,\yA) and (\x+\bulge,\yB) .. (\x,2);
              \node[below, scale=0.8] at (\x,0) {\lab};
            }
        \end{scope}
    \end{tikzpicture}
\end{center}






    

\end{Example}
We caution the reader that the map $\Theta$ is defined only on basis elements and does not extend to a $\bk$-algebra homomorphism. Indeed, $\Theta$ does not preserve the defining relations of KLR algebras. As described below, after quotienting by a suitable ideal and applying an idempotent truncation, one obtains a well-defined homomorphism, which is in fact an isomorphism.

To describe this isomorphism, we introduce some additional notation. The notation below is taken from \cite[Section 2]{maksimau-subdivision}. The reader is encouraged to consult that source for the more general definitions.

Recall that we subdivide at the edge $k\to k+1$. Let $d=\height(\oalpha)$. A sequence $\bi=(\bi_1,\ldots,\bi_d)\in\widebar{I}^{\oalpha}$ is called:
\begin{itemize}
    \item \emph{unordered} if there exists an index $r\in\{1,2,\ldots,d\}$ such that the number of occurrences of $k+1$ in $(\bi_1,\ldots,\bi_r)$ is strictly greater than the number of occurrences of $k$ in $(\bi_1,\ldots,\bi_r)$;
    \item \emph{ordered}\footnote{In \cite{maksimau-subdivision}, the terminology \emph{well-ordered} is used; we prefer \emph{ordered} here, since \emph{well-ordered} already has a standard meaning.} if, for every index $a$ with $\bi_a=k$, we have $a<d$ and $\bi_{a+1}=k+1$;
    \item \emph{almost-ordered} if there exist an ordered sequence $\bj\in\widebar{I}^{\oalpha}$ and an index $r\in\{1,2,\ldots,d-1\}$ such that $\bj_r=k$ and $\bi=\sigma_r(\bj)$.
\end{itemize}
We write $\wellorder$, $\unorder$, and $\almostorder$ for the subsets of ordered, unordered, and almost-ordered sequences in $\widebar{I}^{\oalpha}$. Every almost-ordered sequence is unordered, so $\almostorder\subseteq\unorder$. By the definition of subdivision on words, for any $\bi\in\widebar{I}^{\oalpha}$, the number of indices $j$ with $\bi_j=k$ is the same as the number with $\bi_j=k+1$. In particular, the definition of an ordered sequence can be restated as follows: for each admissible $a$, we have $\bi_a=k$ if and only if $a<d$ and $\bi_{a+1}=k+1$. In other words, all occurrences of $k$ and $k+1$ appear in adjacent pairs. In view of \autoref{eq:subdivision-bijection-words}, $\Phi_k$ gives a bijection between $I^\alpha$ and $\wellorder$.

We define the \emph{truncation idempotent} by
\begin{equation}\label{eq:truncation-idempotent}
\be=\sum_{\bi\in\wellorder} e(\bi)\in R_{\oalpha}(\oGamma).
\end{equation}

\begin{Definition}[{\cite[Definition 2.6]{maksimau-subdivision}}]\label{def:balanced-klr}
  The \textbf{balanced KLR algebra} is the algebra
  $$
    S_{\oalpha}(\oGamma)
      = \be R_{\oalpha}(\oGamma)\be \Big/ \sum_{\bj\in \unorder}\be R_{\oalpha}(\oGamma)e(\bj)R_{\oalpha}(\oGamma)\be.
  $$
\end{Definition}

For simplicity, we omit the quiver $\oGamma$ from the notation, and write $S_{\oalpha}$ for $S_{\oalpha}(\oGamma)$. Let
\[
  \badideal=\sum_{\bj\in \unorder}R_{\oalpha}e(\bj)R_{\oalpha}
\]
be the two-sided ideal of $R_{\oalpha}$, called the \emph{bad ideal}. Then
\begin{equation}\label{eq:balanced-klr-truncation}
  S_{\oalpha}
    = \be R_{\oalpha}\be/\be\badideal\be
    \cong \be(R_{\oalpha}/\badideal)\be.
\end{equation}

It is possible to reduce the number of idempotents appearing in the definition of $\be\badideal\be$.

\begin{Lemma}[{\cite[Lemma 3.7]{maksimau-subdivision}}]
  $\be\badideal\be=\sum\limits_{\bj\in \almostorder}\be R_{\oalpha}e(\bj)R_{\oalpha}\be$.
\end{Lemma}
\begin{Theorem}[{\cite[Theorem 2.12]{maksimau-subdivision}}]\label{thm:subdivision-iso}
    Fix a quiver $\Gamma$ of type $\Aone[e-1]$, and let $\alpha\in Q^+$ and $k\in\{0,1,\ldots,e-1\}$. Then there is a graded $\bk$-algebra isomorphism
    \[
      \Phi_k:R_\alpha\to S_{\oalpha},
    \]
    induced by the map $\Theta$ on the basis of $R_{\alpha}$, as described above.
\end{Theorem}

\begin{proof}
    By \cite[Theorem 2.12]{maksimau-subdivision}, $\Phi_k$ is a $\bk$-algebra isomorphism. It remains to show that $\Phi_k$ preserves the grading. This follows from a straightforward diagrammatic check that $\Theta_i$ is homogeneous on the generators.
\end{proof}
\subsection{Subdivision on cyclotomic KLR algebras}\label{subsec:subdivision-cyclotomic-KLR}
By the results of \autoref{subsec:subdivision-KLR}, we have the following (commutative) diagram:
\begin{center}
\begin{tikzcd}[row sep=3em, column sep=4em]
    & \be R_{\overline{\alpha}} \be \arrow[d, "\pi"] \\
    R_\alpha \arrow[ru, dashed, "\Theta"] \arrow[r, "\Phi", "\cong"'] 
    &  S_{\oalpha}
\end{tikzcd}
\end{center}

Since $\Theta$ is defined only on basis elements of $R_\alpha$, we draw it as a dashed arrow to emphasize this point. Let $\pi$ be the canonical quotient map from $\be R_{\oalpha}\be$ to
$S_{\oalpha}=\be R_{\oalpha}\be\big/\be \badideal\be$, and let $\Phi=\Phi_k$ be the subdivision isomorphism from \autoref{thm:subdivision-iso}. As usual, we write $\oalpha=\Phi(\alpha)$, $\oLambda=\Phi(\Lambda)$, and $\overline{\bi}=\Phi(\bi)$ for $\alpha\in Q^+$, $\Lambda\in P^+$, and $\bi\in I^{\alpha}$.

The cyclotomic ideal $J^\Lambda_\alpha$ of $R_\alpha$ is the two-sided ideal generated by the elements
\[
  \bigl\{\,y_1^{(\Lambda\mid \alpha_{\bi_1})}e(\bi)\ \bigm|\ \bi \in I^\alpha\,\bigr\}.
\]
Equivalently, set
\[
  e^{\Lambda}(\alpha):=\sum_{\bi \in I^\alpha}y_1^{(\Lambda\mid \alpha_{\bi_1})}e(\bi),
\]
so that $J^\Lambda_\alpha=R_\alpha\, e^{\Lambda}(\alpha)\, R_\alpha$.

Similarly, the cyclotomic ideal $J^{\overline{\Lambda}}_{\overline{\alpha}}$ of $R_{\overline{\alpha}}$ is the two-sided ideal generated by the elements
\[
  \bigl\{\,y_1^{(\overline{\Lambda}\mid \alpha_{\bj_1})}e(\bj)\ \bigm|\ \bj \in I^{\overline{\alpha}}\,\bigr\}.
\]
Equivalently, set
\[
  e^{\oLambda}(\oalpha):=\sum_{\bj \in \widebar{I}^{\oalpha}}y_1^{(\oLambda\mid \alpha_{\bj_1})}e(\bj),
\]
so that $J^{\oLambda}_{\oalpha}=R_{\oalpha}\, e^{\oLambda}(\oalpha)\, R_{\oalpha}$.

The image of $J^{\Lambda}_{\alpha}$ under $\Phi$ is the two-sided ideal of $S_{\oalpha}$ generated by $\Phi\bigl(e^{\Lambda}(\alpha)\bigr)$. In $S_{\oalpha}$ we have
\[
  \Phi\bigl(e^{\Lambda}(\alpha)\bigr)
  =\sum_{\bi \in I^\alpha}\Phi(y_1)^{(\Lambda\mid \alpha_{\bi_1})}\,\Phi\bigl(e(\bi)\bigr)
  =\sum_{\bi \in I^\alpha}y_1^{(\Lambda\mid \alpha_{\bi_1})}e(\Phi(\bi))+\be\badideal\be
  =\sum_{\bi \in I^\alpha}y_1^{(\oLambda\mid \alpha_{\overline{\bi}_1})}e(\overline{\bi})+\be\badideal\be,
\]
where the second equality follows from the definition of subdivision on KLR algebras and on positive roots, and the last equality follows from \autoref{lm:cyclotomic-multiplicity}. On the other hand,
\[
  \be e^{\oLambda}(\oalpha)\be
  =\be \sum_{\bj \in \widebar{I}^{\oalpha}}y_1^{(\oLambda\mid \alpha_{\bj_1})}e(\bj)\be
  =\sum_{\bj \in \wellorder}y_1^{(\oLambda\mid \alpha_{\bj_1})}e(\bj).
\]
Since the subdivision map induces a bijection between $I^\alpha$ and $\wellorder$, it follows that
\[
  \Phi\bigl(e^{\Lambda}(\alpha)\bigr)=\be e^{\oLambda}(\oalpha)\be+\be\badideal\be.
\]

Since $\Phi:R_\alpha\xrightarrow{\sim} S_{\oalpha}$ is an algebra isomorphism, it follows that
\[
\Phi(J^\Lambda_\alpha)
= S_{\oalpha}\,\Phi\bigl(e^\Lambda(\alpha)\bigr)\,S_{\oalpha}
= S_{\oalpha}\bigl(\be e^{\oLambda}(\oalpha)\be+\be\badideal\be\bigr)S_{\oalpha}.
\]

Viewing ideals in $S_{\oalpha}=\be R_{\oalpha}\be/\be\badideal\be$, this is the same as
\[
\Phi(J^\Lambda_\alpha)
= \bigl(\be J^{\oLambda}_{\oalpha}\be+\be\badideal\be\bigr)\big/\be\badideal\be
\subseteq \be R_{\oalpha}\be\big/\be\badideal\be.
\]

Therefore $\Phi$ induces an isomorphism on cyclotomic quotients:
\[
R_\alpha^\Lambda
=R_\alpha/J^\Lambda_\alpha
\xrightarrow{\ \sim\ }
S_{\oalpha}/\Phi(J^\Lambda_\alpha)
\xrightarrow{\ \sim\ }
\be R_{\oalpha}\be\big/\be\bigl(\badideal+J^{\oLambda}_{\oalpha}\bigr)\be.
\]

Now the subdivision isomorphism induces the following isomorphism between the cyclotomic KLR algebra and a (cyclotomic) quotient of balanced KLR algebra:
\begin{equation}\label{iso-cyclotomic-subdivision}
\overline{\Phi}: R_\alpha^\Lambda \xrightarrow{\sim} \be R^{\overline{\Lambda}}_{\overline{\alpha}}\be/\be\big(\badideal+J^{\oLambda}_{\oalpha}\big)\be.
\end{equation}
This isomorphism will be called the \textbf{cyclotomic subdivision isomorphism}.

\subsection{Defect invariance}
It is well known that the blocks of cyclotomic KLR algebras of type $\Aone[e-1]$ are the algebras $R^\Lambda_\beta$ for $\beta\in Q^+$ and $\Lambda\in P^+$. In particular, they are indecomposable.
\begin{Definition}\label{def:defect}
    Fix $\beta\in Q^+$ and $\Lambda\in P^+$. The defect of $\beta$ is the integer:
    \begin{equation}
        \defe_\Lambda \beta=(\Lambda\mid\beta)-\frac{1}{2}(\beta\mid\beta)
    \end{equation}
\end{Definition}
The defect of a nonzero block (i.e.\ $R^{\Lambda}_{\beta}\neq 0$) is non-negative (see \cite[Lemma 11.13.2]{kac-infinite-dimensional-lie-algebras}) and measures the complexity of the corresponding block. We now show that the subdivision map preserves the defect of the block.

\begin{Proposition}\label{prop:equality-of-defects}
    Fix the quiver of type $\Aone[e-1]$ with vertex set $I$ and fix $k\in I$. Take $\alpha\in Q^+$ and $\Lambda\in P^+$. Set $\oalpha:=\Phi_k(\alpha)$ and $\oLambda:=\Phi_k(\Lambda)$, where $\Phi_k$ is the corresponding subdivision map. Then
    $\defe_\Lambda \beta = \defe_{\bar{\Lambda}}\bar{\beta}$.
\end{Proposition}
\begin{proof}
    Assume $\beta=\sum_{0\leq i\leq e{-}1}x_i\alpha_i$. Then, recall from \autoref{subsec:subdivision-root} that $\bar{\beta}$ is of the form:
    \[
        \bar{\beta}=\sum_{0\leq i\leq k}x_i\alpha_i + x_k\alpha_{k{+}1} + \sum_{k{+}1\leq i\leq e{-}1}x_i\alpha_{i{+}1}.
    \]
    Calculating the inner product $(\beta\mid \beta)$ gives:
    \begin{align*}
        (\beta\mid\beta) &= (2x_0^2-x_0x_1-x_0x_{e{-}1}) \quad + \sum_{1\leq i\leq k{-}1}(2x_i^2-x_ix_{i{-}1}-x_ix_{i{+}1}) \\
        &\quad + (2x_k^2-x_kx_{k{+}1}-x_kx_{k{-}1}) \quad + \sum_{k{+}1\leq i\leq e{-}2}(2x_{i}^2-x_{i}x_{i{-}1}-x_ix_{i{+}1}) \\
        &\quad + (2x_{e{-}1}^2-x_{e{-}1}x_{e{-}2}-x_{e{-}1}x_{0}).
    \end{align*}
    whereas, calculating the inner product $(\bar{\beta}\mid\bar{\beta})$ gives: 
    \begin{align*}
        (\bar{\beta}\mid\bar{\beta}) &= (2x_0^2-x_0x_1-x_0x_{e{-}1}) \quad + \sum_{1\leq i\leq k{-}1}(2x_i^2-x_ix_{i{-}1}-x_ix_{i{+}1}) \\
        &\quad + (2x_k^2-x_k^2-x_kx_{k{-}1}) + (2x_k^2-x_k^2-x_kx_{k{+}1}) \quad + \sum_{k{+}2\leq i\leq e{-}1}(2x_{i{-}1}^2-x_{i{-}2}x_{i{-}1}-x_ix_{i{-}1}) \\
        &\quad + (2x_{e{-}1}^2-x_{e{-}1}x_{e{-}2}-x_{e{-}1}x_{0}).
    \end{align*}
    Simplifying the middle terms, we observe $(\beta\mid\beta)=(\bar{\beta}\mid\bar{\beta})$. It remains to show $(\Lambda\mid\beta)=(\bar{\Lambda}\mid\bar{\beta})$, which follows by \autoref{lm:pairing-invariance}.
\end{proof}
\section{Combinatorial Subdivision}\label{sec:combinatorial-subdivision}
In this section, we first define the subdivision map on partitions, which we use in \autoref{subsec:image-of-idempotent} to describe the image of idempotents in a KLR algebra under the subdivision map $\Phi_k$. We give two equivalent definitions, using Young diagrams (\autoref{subsec:subdivision-young-diagram}) and abaci (\autoref{subsec:subdivision-abacus}), since each is convenient for different arguments.

We then extend the construction to row-standard tableaux (\autoref{subsec:subdivision-standard-tableaux}). Its restriction to standard tableaux sends standard tableaux to standard tableaux, and we show that it preserves degree (\autoref{thm:degree-invariance-standard-tableau}). This construction is used for describing the bases of permutation modules and Specht modules under the isomorphism induced by the subdivision on KLR algebras; see \autoref{subsec:subdivision-permutation-modules} and \autoref{subsec:subdivision-specht-modules}.

A preliminary partition-level sketch in a restricted specialization appeared in \cite{qin-subdivision-klrw} (namely $k=0$ in the subdivision datum; see \autoref{def:subdivision-datum}). The present section is self-contained and substantially extends this to a complete picture, and the tableaux-level construction is brand new.
\begin{Definition}\label{def:subdivision-datum}
    A \textbf{subdivision datum} is a tuple $\subdatum$, where $e\in\Z_{\ge 3}$, $I$ is the vertex set of the  quiver $\Aone[e-1]$ (identified with $\{0,1,\dots,e-1\}$), $\Lambda\in P^+$ is a dominant weight of level $\ell\in\Z_{\ge 1}$, $\alpha\in Q^+$ is a positive root, $\bkappa\in I^\ell$ is a charge of $\Lambda$, and $k\in I$.
\end{Definition}

Fix a subdivision datum $\subdatum$, let $\Phi=\Phi_k$ be the subdivision map (on quivers, positive roots, dominant weights) introduced in the earlier sections, and set $\oGamma=\Phi(\Gamma)$, $\oalpha=\Phi(\alpha)$, and $\oLambda=\Phi(\Lambda)$. Let $\overline{\bkappa}\in \widebar{I}^\ell$ be such that $\Phi_k(\Lambda_{\kappa_i})=\Lambda_{\okappa_i}$ for each $1\le i\le \ell$.\footnote{The reader should be cautious that this $\okappa$ does not coincide with the subdivision of the word $\kappa\in I^\ell$.}

Recall that $\Par[\bkappa]_{\alpha}$ is the set of $\ell$-partitions of residue content $\alpha$ with respect to the charge $\bkappa$. Similarly, $\Par[\overline{\bkappa}]_{\oalpha}$ is the set of $\ell$-partitions of residue content $\oalpha$ with respect to the charge $\overline{\bkappa}$. Our aim is to construct a subdivision map $\Phi_k:\Par[\bkappa]_{\alpha}\to \Par[\overline{\bkappa}]_{\oalpha}$ that is compatible with other subdivision maps defined in the earlier sections.

To define $\Phi_k(\blam)$, we first define $\Phi_k(\lambda)$ for an arbitrary partition $\lambda$, and then apply the construction componentwise. More precisely, if $\blam=(\blam^{(1)},\ldots,\blam^{(\ell)})\in\Par[\bkappa]_{\alpha}$, then
\[
  \Phi_k(\blam):=\bigl(\Phi_k(\blam^{(1)}),\ldots,\Phi_k(\blam^{(\ell)})\bigr)\in \Par[\overline{\bkappa}]_{\oalpha}
\]
For this reason, it suffices to consider the case $\Lambda=\Lambda_i$ for some $i\in I$.
\subsection{Subdivision on partitions via Young diagrams}\label{subsec:subdivision-young-diagram}
Fix a subdivision datum $\subdatum[\Lambda_i][i]$\footnote{Strictly speaking, it should be written as $\subdatum[\Lambda_i][(i)]$, but we abuse notation and identify $(i)$ with $i$.} and let $\Lambda:=\Lambda_i$. Recall from \autoref{subsec:multipartition-charge-residue} that $[\lambda]$ is the Young diagram of a partition $\lambda$, and that $[\lambda]_\Lambda$ is the Young diagram filled with residues from $I$. As usual, since the choice of $\Lambda=\Lambda_i$ is clear from the context, we abuse notation and write $[\lambda]$ for the Young diagram filled with residues.

\begin{Definition}\label{def:k-strip}
    Fix a subdivision datum $\subdatum[\Lambda=\Lambda_i][i]$ and $\lambda\in\Par[\Lambda]_{\alpha}$.
    A \textbf{\boldmath $(k,k+1)$-strip of length $m$} in $[\lambda]$ is a finite sequence of nodes in $[\lambda]$,
    $(A_1,A_2,\ldots,A_m)$, such that:
    \begin{enumerate}
      \item $\res(A_{2j+1})=k$ and $\res(A_{2j})=k+1$ for all admissible $j$;
      \item if $A_{2j+1}$ lies in row $r$ and column $c$, then $A_{2j+2}$ lies in row $r$ and column $c+1$;
      \item if $A_{2j}$ lies in row $r$ and column $c$, then $A_{2j+1}$ lies in row $r+1$ and column $c$.
    \end{enumerate}
    
    Similarly, a \textbf{\boldmath $(k+1,k)$-strip of length $m$} in $[\lambda]$ is a finite sequence of nodes in $[\lambda]$, $(A_1,A_2,\ldots,A_m)$, such that:
    \begin{enumerate}
      \item $\res(A_{2j+1})=k+1$ and $\res(A_{2j})=k$ for all admissible $j$;
      \item if $A_{2j+1}$ lies in row $r$ and column $c$, then $A_{2j+2}$ lies in row $r+1$ and column $c$;
      \item if $A_{2j}$ lies in row $r$ and column $c$, then $A_{2j+1}$ lies in row $r$ and column $c+1$.
    \end{enumerate}
    
    A $(k,k+1)$-strip or $(k+1,k)$-strip is called \textbf{maximal} if it cannot be extended to a strip of greater length. If the length $m=1$, we call the strip \textbf{trivial}. We call $A_1$ the \textbf{initial node} and $A_m$ the \textbf{terminal node}.
\end{Definition}

\begin{Example}\label{eg:k-strip}
    Take the subdivision datum to be $\subdatum=(5,I,\Lambda_1,\alpha,1,1)$ and $\lambda=(11^7)$, we form the Young diagram $[\lambda]$ and color the maximal $(1,2)$-strips in {\color{cyan}cyan}, color the maximal $(2,1)$-strip in {\color{OrangeRed}orange}:

    \begin{center}
        \tikzset{
            C/.style={fill=cyan,text=white},
            O/.style={fill=OrangeRed, text=white}
                }
        \Tableau[no border,box height=0.5,box width=0.5]{[C]1[C]2340[C]1[C]2340[C]1,0[C]1[C]2340[C]1[C]2340,40[C]1[C]2340[C]1[C]234,340[C]1[C]2340[C]1[C]23,[O]2340[C]1[C]2340[C]1[C]2,[O]1[O]2340[C]1[C]2340[C]1,0[O]1[O]2340[C]1[C]2340}
    \end{center}    
\end{Example}
Since we work in type $\Aone[e-1]$, the following observation is immediate.

\begin{Lemma}\label{lm:initial-nodes}
Every maximal $(k,k+1)$-strip in $[\lambda]$ has its initial node in the first row of $[\lambda]$, and every maximal $(k+1,k)$-strip in $[\lambda]$ has its initial node in the first column of $[\lambda]$. \hfill $\square$
\end{Lemma}

The definition in \autoref{def:k-strip} extends in a natural way to triples $(k,k+1,k+2)$ and $(k+2,k+1,k)$. Indeed, we will define the image of a partition $\lambda$ under subdivision by replacing its maximal $(k,k+1)$-strips and maximal $(k+1,k)$-strips with $(k,k+1,k+2)$-strips and $(k+2,k+1,k)$-strips, respectively. We make this precise as follows:

\begin{Definition}\label{def:subdivision-young-diagram}
    Fix a subdivision datum $\subdatum[\Lambda_i][i]$ and take a partition $\lambda\in \Par[\Lambda_i]_{\alpha}$.
    For each maximal $(k,k+1)$-strip $S$ in  $[\lambda]$, insert a node labelled with $\overline{k}$ to the right of every $k$-node in $S$. Similarly, for each maximal $(k+1,k)$-strip $S$ in $[\lambda]$, insert a node labelled with $\overline{k}$ above every $k$-node in $S$. After performing these two procedures for all maximal strips, apply the relabelling from \autoref{eq:relabel-rules} to each node. The resulting Young diagram is denoted by $\Phi^{Y}_k([\lambda])$.
\end{Definition}

There are two points that are not immediate from the definition: (1) the resulting diagram is the Young diagram of a partition; (2) the new label on each node is exactly its residue with respect to the new charge (or dominant weight) for $\Aone[e]$. We now prove both statements.

\smallskip
By \autoref{def:subdivision-young-diagram}, all four kinds of maximal strips (except for the trivial $(k+1,k)$-strip, that is, a single $(k+1)$-node, which simply becomes a $(k+2)$-node) are treated as follows:

For any maximal $(k,k+1)$-strip, 
\begin{enumerate}[label=(\roman*)]
    \item if it is of the form 
    \begin{center}
        \RibbonTableau[no border, box font=\tiny,styles={D={draw=aTableauInner}},tikz after={
        \draw[dotted,aTableauMain]([shift={(0,-0.15)}]A-1-2.south)--([shift={(0,0.15)}]A-4-4.north);
        \draw[aTableauInner](A-1-1.south east)--++(0,-0.5)--++(0.5,0)--++(0,-0.5)--++(0.5,0);
        \draw[aTableauInner](A-1-2.south east)--++(0.5,0)--++(0,-0.5)--++(0.5,0)--++(0,-0.5);}]
        {[D]11_{k},[D]12_{k{+}1},
        [D]44_{k},[D]45_{k{+}1}}
    \end{center}
    then replace it by 
    \begin{center}
       \RibbonTableau[no border, box font=\tiny,styles={D={draw=aTableauInner}},tikz after={
        \draw[dotted,aTableauMain]([shift={(0,-0.15)}]A-1-2.south)--([shift={(0,0.15)}]A-4-5.north);
        \draw[aTableauInner](A-1-1.south east)--++(0,-0.5)--++(0.5,0)--++(0,-0.5)--++(0.5,0);
        \draw[aTableauInner](A-1-3.south east)--++(0.5,0)--++(0,-0.5)--++(0.5,0)--++(0,-0.5);}]
        {[D]11_{k},[D]12_{k{+}1},[D]13_{k{+}2},
        [D]44_{k},[D]45_{k{+}1},[D]46_{k{+}2}}
    \end{center}
    \item if it is of the form 
    \begin{center}
        \RibbonTableau[no border, box font=\tiny,styles={D={draw=aTableauInner}},tikz after={
        \draw[dotted,aTableauMain]([shift={(0,-0.15)}]A-1-2.south)--([shift={(0,0.15)}]A-4-4.north);
        \draw[aTableauInner](A-1-1.south east)--++(0,-0.5)--++(0.5,0)--++(0,-0.5)--++(0.5,0);
        \draw[aTableauInner](A-1-2.south east)--++(0.5,0)--++(0,-0.5)--++(0.5,0)--++(0,-0.5);}]
        {[D]11_{k},[D]12_{k{+}1},
        [D]44_{k}}
    \end{center}
    then replace it by 
    \begin{center}
       \RibbonTableau[no border, box font=\tiny,styles={D={draw=aTableauInner}},tikz after={
        \draw[dotted,aTableauMain]([shift={(0,-0.15)}]A-1-2.south)--([shift={(0,0.15)}]A-4-5.north);
        \draw[aTableauInner](A-1-1.south east)--++(0,-0.5)--++(0.5,0)--++(0,-0.5)--++(0.5,0);
        \draw[aTableauInner](A-1-3.south east)--++(0.5,0)--++(0,-0.5)--++(0.5,0)--++(0,-0.5);}]
        {[D]11_{k},[D]12_{k{+}1},[D]13_{k{+}2},
        [D]44_{k},[D]45_{k{+}1}}
    \end{center}
\end{enumerate}

For any non-trivial maximal $(k+1,k)$-strip: 
\begin{enumerate}
    \item if it is of the form 
    \begin{center}
        \RibbonTableau[no border, box font=\tiny,styles={D={draw=aTableauInner}},tikz after={
        \draw[dotted,aTableauMain]([shift={(0.15,0)}]A-2-1.east)--([shift={(-0.15,0)}]A-4-4.west);
        \draw[aTableauInner](A-2-1.south east)--++(0,-0.5)--++(0.5,0)--++(0,-0.5)--++(0.5,0);
        \draw[aTableauInner](A-1-1.south east)--++(0.5,0)--++(0,-0.5)--++(0.5,0)--++(0,-0.5);}]
        {[D]11_{k{+}1},[D]21_{k},
        [D]44_{k{+}1},[D]54_{k}}
    \end{center}
    then replace it by 
    \begin{center}
       \RibbonTableau[no border, box font=\tiny,styles={D={draw=aTableauInner}},tikz after={
        \draw[dotted,aTableauMain]([shift={(0.15,0)}]A-2-1.east)--([shift={(-0.15,0)}]A-5-4.west);
        \draw[aTableauInner](A-3-1.south east)--++(0,-0.5)--++(0.5,0)--++(0,-0.5)--++(0.5,0);
        \draw[aTableauInner](A-1-1.south east)--++(0.5,0)--++(0,-0.5)--++(0.5,0)--++(0,-0.5);}]
        {[D]11_{k{+}2},[D]21_{k{+}1},[D]31_{k},
        [D]44_{k{+}2},[D]54_{k{+}1},[D]64_{k}}
    \end{center}
    \item if it is of the form 
    \begin{center}
        \RibbonTableau[no border, box font=\tiny,styles={D={draw=aTableauInner}},tikz after={
        \draw[dotted,aTableauMain]([shift={(0.15,0)}]A-2-1.east)--([shift={(-0.15,0)}]A-4-4.west);
        \draw[aTableauInner](A-2-1.south east)--++(0,-0.5)--++(0.5,0)--++(0,-0.5)--++(0.5,0);
        \draw[aTableauInner](A-1-1.south east)--++(0.5,0)--++(0,-0.5)--++(0.5,0)--++(0,-0.5);}]
        {[D]11_{k{+}1},[D]21_{k},
        [D]44_{k{+}1}}
    \end{center}
    then replace it by 
    \begin{center}
       \RibbonTableau[no border, box font=\tiny,styles={D={draw=aTableauInner}},tikz after={
        \draw[dotted,aTableauMain]([shift={(0.15,0)}]A-2-1.east)--([shift={(-0.15,0)}]A-5-4.west);
        \draw[aTableauInner](A-3-1.south east)--++(0,-0.5)--++(0.5,0)--++(0,-0.5)--++(0.5,0)--++(0,0.5);
        \draw[aTableauInner](A-1-1.south east)--++(0.5,0)--++(0,-0.5)--++(0.5,0)--++(0,-0.5);}]
        {[D]11_{k{+}2},[D]21_{k{+}1},[D]31_{k},
        [D]44_{k{+}2},}
    \end{center}
\end{enumerate}

To show that $\Phi^{Y}_k[\lambda]$ is the Young diagram of a partition, it suffices to check that each of the four steps still produces the Young diagram of a partition. Consider a maximal $(k,k+1)$-strip of the form:
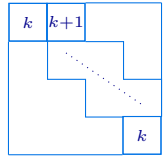
\begin{center}
        \RibbonTableau[no border, box font=\tiny,styles={D={draw=aTableauInner}},tikz after={
        \draw[dotted,aTableauMain]([shift={(0,-0.15)}]A-1-2.south)--([shift={(0,0.15)}]A-4-4.north);
        \draw[aTableauInner](A-1-1.south west)--++(0,-1.5);
        \draw[aTableauInner](A-4-4.south west)--++(-1.5,0);
        \draw[aTableauInner](A-1-2.north east)--++(1,0);
        \draw[aTableauInner](A-4-4.north east)--++(0,1.5);
        \draw[aTableauInner](A-1-1.south east)--++(0,-0.5)--++(0.5,0)--++(0,-0.5)--++(0.5,0);
        \draw[aTableauInner](A-1-2.south east)--++(0.5,0)--++(0,-0.5)--++(0.5,0)--++(0,-0.5);}]
        {[D]11_{k},[D]12_{k{+}1},
        [D]44_{k}}
\end{center}
As $\lambda$ is a partition, this strip must lie within a minimal rectangle (as shown in the figure above) in $[\lambda]$. Hence, by replacing the strip, we obtain the following:
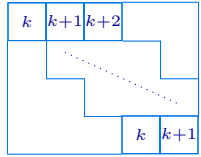
\begin{center}
       \RibbonTableau[no border, box font=\tiny,styles={D={draw=aTableauInner}},tikz after={
        \draw[dotted,aTableauMain]([shift={(0,-0.15)}]A-1-2.south)--([shift={(0,0.15)}]A-4-5.north);
        \draw[aTableauInner](A-1-1.south west)--++(0,-1.5);
        \draw[aTableauInner](A-4-4.south west)--++(-1.5,0);
        \draw[aTableauInner](A-1-3.north east)--++(1,0);
        \draw[aTableauInner](A-4-5.north east)--++(0,1.5);
        \draw[aTableauInner](A-1-1.south east)--++(0,-0.5)--++(0.5,0)--++(0,-0.5)--++(0.5,0);
        \draw[aTableauInner](A-1-3.south east)--++(0.5,0)--++(0,-0.5)--++(0.5,0)--++(0,-0.5);}]
        {[D]11_{k},[D]12_{k{+}1},[D]13_{k{+}2},
        [D]44_{k},[D]45_{k{+}1}}
\end{center}
This transformation still results in a partition by modifying such a rectangle. The other three operations can be verified in a similar manner. Hence, after replacing every maximal $(k,k+1)$-strip and $(k+1,k)$-strip, we obtain a new partition.

Let $\Phi^{Y}_k(\lambda)$ be the partition such that $[\Phi^{Y}_k(\lambda)]=\Phi^{Y}_k([\lambda])$. In this way, $\Phi^{Y}_k$ is realized as a map on the set of partitions (rather than merely on Young diagrams) via Young diagrams.

\smallskip
gWe next consider residues. Assume $\lambda\in\Par[\Lambda_i]_{\alpha}$. For any node $A=(r,c)\in[\lambda]$, the residue of $A$ is $i+c-r\pmod{e}$. We call the node $A=(1,1)$ the \emph{first node} of $[\lambda]$. In particular, the residue of the first node is $i$, corresponding to the dominant weight $\Lambda_i$. If $0\leq i\leq k$, then by our construction in \autoref{def:subdivision-young-diagram}, the first node in $\Phi^{Y}_k([\lambda])$ has residue $i$; if $k+1\leq i\leq e-1$, the first node in $\Phi^{Y}_k([\lambda])$ has residue $i+1$. This agrees with the definition of the subdivision map on dominant weights in \autoref{eq:def-subdivision-weight}. Hence the new dominant weight is $\Phi_k(\Lambda_i)$, and the new partition $\Phi^Y_k(\lambda)$ lies in $\Par[\Phi_k(\Lambda_i)]$. 

\begin{Lemma}\label{lem:residue-characterization}
    Fix quiver type $\Aone[e-1]$, let $\lambda\in\Par[\Lambda_i]$ and let $[\lambda]$ be its Young diagram.  A function $f: [\lambda] \longrightarrow \Z/e\Z$
    coincides with the residue function on $[\lambda]$ with respect to $\Lambda_i$ if and only if it satisfies the following two conditions:
    \begin{enumerate}[label=(\roman*)]
        \item $f(1,1)=i$;
        \item for any two horizontally adjacent nodes $(r,c),(r,c+1)\in [\lambda]$ one has
              $f(r,c+1)=f(r,c)+1$, and for any two vertically adjacent nodes $(r,c),(r+1,c)\in [\lambda]$ one has
              $f(r+1,c)=f(r,c)-1$, with both equalities taken in $\Z/e\Z$.
    \end{enumerate}
\end{Lemma}
\begin{proof}
    If $f$ is the residue function, then the two properties follow immediately from
    definition.
    Conversely, assume $f$ satisfies \emph{(i)} and \emph{(ii)}.  We prove by induction on $r+c$ that
    $f(r,c)\equiv i+c-r\pmod{e}$ for all $(r,c)\in [\lambda]$.  The base case $(1,1)$ is \emph{(i)}.
    For $(r,c)\neq(1,1)$, either $(r,c-1)\in [\lambda]$ (if $c>1$) or $(r-1,c)\in [\lambda]$ (if $r>1$), and then \emph{(ii)} gives
    \[
        f(r,c)\equiv f(r,c-1)+1\pmod{e} \quad\text{or}\quad f(r,c)\equiv f(r-1,c)-1\pmod{e}
    \]
    Applying the induction hypothesis to $(r,c-1)$ or $(r-1,c)$ yields
    $f(r,c)\equiv i+c-r\pmod{e}$, as required.  Hence $f$ agrees with the residue function.
\end{proof}

To show that the label associated to each node in $[\Phi^{Y}_k(\lambda)]$ coincides with the residue of that node with respect to $\Lambda_j:=\Phi_k(\Lambda_i)$, note that we already know the first node of $[\Phi^Y_k(\lambda)]$ has residue $j$. Thus it suffices to show that \emph{(ii)} in \autoref{lem:residue-characterization} holds for type $\Aone[e]$.

For each of the four steps above applied to a maximal strip, the diagram is modified only inside the minimal rectangle containing that strip.  Before the relabeling procedure, the labels on all original nodes are unchanged (hence still equal to their residues in $[\lambda]$), and the only new labels that appear are the labels $\overline{k}$ on the nodes inserted to subdivide the strip: these are inserted either immediately to the right of a $k$-node (for a maximal $(k,k{+}1)$-strip) or immediately above a $k$-node (for a maximal $(k{+}1,k)$-strip).

Let $\operatorname{L}$ be the labeling on the set of nodes of $[\Phi^Y_k(\lambda)]$ after performing the local modifications and the final relabeling.  Consider any pair of horizontally adjacent nodes $(r,c),(r,c+1)\in[\Phi_k(\lambda)]$ (respectively, vertically adjacent nodes $(r,c),(r+1,c)\in[\Phi_k(\lambda)]$). If both endpoints are original nodes of $[\lambda]$, then their labels were not altered during the local modification inside the supporting rectangles, and the global relabeling acts compatibly on both labels; hence the relations
\[
    \operatorname{L}(r,c+1)\equiv \operatorname{L}(r,c)+1 \pmod{e+1}
    \qquad\text{and}\qquad
    \operatorname{L}(r+1,c)\equiv \operatorname{L}(r,c)-1 \pmod{e+1}
\]
continue to hold.

Therefore it remains only to check the horizontal and vertical adjacency relations for pairs of nodes in which at least one node is newly inserted.  By construction, every inserted node is labeled $\overline{k}$ before relabeling and lies on a modified maximal strip.  Such a node is inserted only in the following local situation: a node labeled $k$ is adjacent (horizontally or vertically, according to the type of the strip) to another node on the strip, and the insertion places $\overline{k}$ immediately to the right of that $k$-node (in the horizontal case) or immediately above that $k$-node (in the vertical case).  When the adjacent node on the strip exists and is labeled $k+1$, the insertion replaces the original adjacency by a chain of two adjacencies
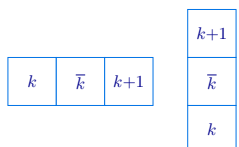
\begin{center}
    \resizebox{0.2\linewidth}{!}{
    \Tableau[no border,box height=1,box width=1]{k{\overline{k}}{k{+}1}}\qquad \Tableau[no border,box height=1,box width=1]{{k{+}1},{\overline{k}},k}
    }
\end{center}
respectively.  Near an endpoint of the strip, the corresponding neighboring $k+1$-node may be absent; in that case $\overline{k}$ has only one adjacent neighbor along the strip, and there is no second adjacency relation to verify there.

After applying the relabeling, $k$ and $k+1$ are sent to $k$ and $k+2$, respectively, while $\overline{k}$ is sent to $k+1$. Consequently, every horizontal (respectively, vertical) adjacency involving an inserted node satisfies the required increasing (respectively, decreasing) relation modulo $e+1$ as well. Hence \emph{(ii)} of \autoref{lem:residue-characterization} holds for $\operatorname{L}$ in type $\Aone[e]$. Since the first node has residue $j$, it follows from \autoref{lem:residue-characterization} that $\operatorname{L}$ coincides with the residue function on $[\Phi^Y_k(\lambda)]$ with respect to $\Lambda_j$.

\begin{Example}\label{eg:subdivision-tableau}
    Continue with \autoref{eg:k-strip}, the image of $\lambda=(11)^7$ under the subdivision map $\Phi^{Y}_1$ is the one corresponding to the following Young diagram:
    \begin{center}
        \tikzset{
            C/.style={fill=cyan,text=white},
            O/.style={fill=OrangeRed, text=white}
        }
        \Tableau[no border,box height=0.5,box width=0.5]{
            [C]1[C]2[C]3450[C]1[C]2[C]3450[C]1[C]2,
            0[C]1[C]2[C]3450[C]1[C]2[C]3450,
            50[C]1[C]2[C]3450[C]1[C]2[C]345,
            450[C]1[C]2[C]3450[C]1[C]2[C]34,
            [O]3450[C]1[C]2[C]3450[C]1[C]2[C]3,
            [O]2[O]3450[C]1[C]2[C]3450[C]1[C]2,
            [O]1[O]2[O]3450[C]1[C]2[C]3450,
            0[O]1
        }
    \end{center}    
    The $(1,2)$-strips become the $(1,2,3)$-strips and the $(2,1)$-strips become the $(3,2,1)$-strips. The partition $\Phi^{Y}_1(\lambda)=(14,13^5,12,2)$
\end{Example}
\subsection{Subdivision on partitions via abaci}\label{subsec:subdivision-abacus}
The definition of the subdivision map on partitions in \autoref{subsec:subdivision-young-diagram} is natural, but it is difficult to compute in practice for large partitions, since it requires working with the entire Young diagram. In this section, we give a second, more computable description of the map, and in \autoref{subsec:equivalence-of-two-definitions} we show that the two definitions agree.

We use the abacus combinatorics from \autoref{subsec:young-diagram-abaci} and \autoref{subsec:multipartition-charge-residue}. 

\begin{Definition}\label{def:abacus-subdivision-datum}
    Fix a subdivision datum $\subdatum[\Lambda_i][i]$. For each partition $\lambda\in\Par[\Lambda_i]_{\alpha}$, an \emph{abacus subdivision datum} for $\lambda$
    is a tuple $(a,c,d,a')\in\mathbb Z^{4}$ such that
    \[
    a\ge \max\{k,\ell(\lambda)\},\qquad a\equiv i\pmod e,\qquad a+d=ce+k,\qquad d\in I,\qquad a'=a+c.
    \]
\end{Definition}

The subdivision map on the set of partitions is easier to define in terms of abaci combinatorics:

\begin{Definition}\label{def:subdivision-abacus}
    Fix a subdivision datum $\subdatum[\Lambda_i][i]$ and let $\lambda\in\Par[\Lambda_i]_{\alpha}$, choose an abacus subdivision datum $\absubdatum$ for $\lambda$. Form the $e$-abacus of $\lambda$ with $a$ beads. Insert a new runner immediately to the left of the $k$-runner, and place $c$ beads on this runner in the top rows.
    Finally, relabel the runners from left to right by $0,1,\ldots,e-1,e$.
    The partition corresponding to the resulting $(e+1)$-abacus with $a'$ beads is defined to be $\Phi^A_k(\lambda)$.
\end{Definition}

It is useful to record the positions of the beads added in \autoref{def:subdivision-abacus}.
Define the set $\bT_{e,k,c}$ of non-negative integers by
\begin{equation}\label{eq:T-ekc}
    \bT_{e,k,c}:=\{ie+k\mid 0\leq i\leq c-1\}.
\end{equation}
Let $\subdatum[\Lambda_i][i]$ be a subdivision datum and let $(a,c,d,a')$ be an abacus subdivision datum
for $\lambda\in \Par[\Lambda_i]_{\alpha}$.
Then, in the resulting $(e+1)$-abacus for $\Phi^A_k(\lambda)$, the positions of the newly added beads are exactly the set $\bT_{e+1,k,c}$.

\begin{Example}\label{eg:subdivision-abacus}
    Returning to \autoref{eg:k-strip}, the subdivision datum is $\subdatum=(5,I,\Lambda_1,\alpha,1,1)$ and $\lambda=(11^7)$. Choose an abacus subdivision datum for $\lambda$ to be $\absubdatum=(11,2,0,13)$. Form the $e$-abacus of $\lambda$ with $11$ beads:
    \begin{center}
      \Abacus[runner labels={0,1,2,3,4}]{5}{11^7,0^4}
    \end{center}
    Apply $\Phi^{A}_k$ to $\lambda$: insert a new runner immediately to the left of the $1$-runner, and place two beads on this runner as high as possible. The resulting $(e+1)$-abacus with $13$ beads is:
    \begin{center}
      \Abacus[runner labels={0,1,2,3,4,5}]{6}{14,13^5,12,*2,0,0,0,*0,0}
    \end{center}
    This abacus corresponds to the partition $(14,13^5,12,2)$. Comparing with \autoref{eg:subdivision-tableau}, we see that the two definitions agree. The two newly added beads are located at $\{1,7\}$, which is precisely $\bT_{6,1,2}$.
\end{Example}

\begin{Lemma}\label{lm:e-regularity-preserving}
    Take $\lambda\in\Par[\Lambda]_{\alpha}$ $e$-regular, then $\Phi^A_k(\lambda)$ is $(e+1)$-regular.
\end{Lemma}
\begin{proof}
    More generally, inserting a flush runner preserves $e$-regularity. Recall that an $e$-abacus display of $\lambda$ is $e$-regular if and only if it contains no string of $e$ consecutive beads with no gap between them. Suppose that $\Phi^A_k(\lambda)$ is not $(e+1)$-regular. Then the $(e+1)$-abacus display of $\Phi^A_k(\lambda)$ contains $(e+1)$ consecutive beads with no gap between them. Exactly one of these beads lies on the newly inserted runner, so the remaining $e$ beads come from the original $e$-abacus display and form $e$ consecutive beads with no gap between them. This shows that $\lambda$ is not $e$-regular, a contradiction. 
\end{proof}
In view of the discussion in \autoref{sec:categorification}, \autoref{lm:e-regularity-preserving} is the same as \cite[Lemma 4.18]{alice-full-runner-removal}.
\subsection{Equivalence of the two definitions}\label{subsec:equivalence-of-two-definitions}
In this section, we prove that for a partition $\lambda$, the map $\Phi_k^Y(\lambda)$ defined in \autoref{subsec:subdivision-young-diagram} using Young diagrams coincides with $\Phi_k^A(\lambda)$ defined in \autoref{subsec:subdivision-abacus} using abaci. Consequently, in later sections we may simply write $\Phi_k(\lambda)$ and refer to it as the subdivision map on partitions.

\begin{Definition}\label{def:k_lambda}
    Fix a subdivision datum $\subdatum[\Lambda=\Lambda_x][x]$ and take $\lambda\in \Par[\Lambda]_{\alpha}$.
    Let $k(\lambda)$ be the number of non-trivial maximal $(k+1, k)$-strips in the Young diagram $[\lambda]$. By \autoref{lm:initial-nodes}, such strips correspond to adjacent pairs in the first column $(r, 1)$ and $(r+1, 1)$ with residues $k+1$ and $k$ respectively. Thus:
    $$
    k(\lambda) := \# \{ 1 \le r < \ell(\lambda) \mid \operatorname{res}(r, 1) \equiv k+1 \pmod e \}.
    $$
\end{Definition}
In \autoref{def:k_lambda}, the maximal $(k{+}1,k)$-strip is required to be non-trivial. Indeed, by \autoref{def:subdivision-young-diagram}, a trivial $(k{+}1,k)$-strip (i.e.\ a single $(k{+}1)$-node) simply becomes a $(k{+}2)$-node under $\Phi^{Y}_k$ and hence does not change the Young diagram under subdivision.

The case $k(\lambda)=0$ is of key importance in \autoref{subsec:subdivision-standard-tableaux} and \autoref{sec:categorical-subdivisions}. For later use, we extend this terminology to multipartitions as follows.

\begin{Definition}\label{def:k-horizontal}
    Fix a subdivision datum $\subdatum$, an $\ell$-partition $\blam\in \Par[\bkappa]$ is \textbf{$\bm{k}$-horizontal} if each component $\blam^{(m)}$ satisfies $k(\blam^{(m)})=0$, equivalently, if in every component of $[\blam]$ there is no non-trivial maximal $(k{+}1,k)$-strip.
\end{Definition}

We now study some properties of $k(\lambda)$. Let $\lfloor\bullet\rfloor$ be the floor function, i.e.  for any real number $x$, $\lfloor x \rfloor$ is the maximal integer $N$ such that $x\geq N$.
\begin{Lemma}\label{lm:counting_formula}
    Fix a subdivision datum $\subdatum[\Lambda_x][x]$ and take $\lambda\in \Par[\Lambda_x]_{\alpha}$.
    Let $\rho\in I$ be the unique integer satisfying $\rho\equiv x-k\pmod e$.
    Then
    \begin{equation}\label{eq:k-lambda-counting-formula}
        k(\lambda)
        = \max\Bigl(0,\ \Bigl\lfloor \frac{\ell(\lambda)-1-\rho}{e}\Bigr\rfloor + 1\Bigr).
    \end{equation}
\end{Lemma}

\begin{proof}
    In the first column of $[\lambda]$, $\operatorname{res}(r,1)\equiv x+1-r\pmod e$.
    Hence $\operatorname{res}(r,1)\equiv k+1\pmod e$ is equivalent to
    \begin{equation}\label{eq:counting-congruence}
    x+1-r \equiv k+1 \pmod e
    \quad\iff\quad
    r \equiv x-k \pmod e.
    \end{equation}
    We count integers $r$ with $1\le r\le \ell(\lambda)-1$ satisfying \eqref{eq:counting-congruence}.
    The smallest positive solution is $\rho$, and every solution is of the form $\rho+te$ for $t\in\Z_{\ge 0}$.
    Therefore $k(\lambda)$ is the number of non-negative integers $t$ such that $\rho+te\le \ell(\lambda)-1$.
    
    If $\ell(\lambda)-1<\rho$, there are no solutions and $k(\lambda)=0$.
    Otherwise, the largest admissible $t$ is $\bigl\lfloor(\ell(\lambda)-1-\rho)/e\bigr\rfloor$, so the number of solutions is
    $\bigl\lfloor(\ell(\lambda)-1-\rho)/e\bigr\rfloor+1$.
    This gives \eqref{eq:k-lambda-counting-formula}.
\end{proof}

\begin{Lemma}\label{lm:length-subdivision}
    Fix a subdivision datum $\subdatum[\Lambda_x][x]$ and take $\lambda\in \Par[\Lambda_x]_{\alpha}$. Then 
    \[
        \ell\big(\Phi_k^Y(\lambda)\big) = \ell(\lambda) + k(\lambda).
    \]
\end{Lemma}

\begin{proof}
    The length of a partition is the number of nodes in the first column of its Young diagram. Maximal $(k+1, k)$-strips start from the first column. By \autoref{def:subdivision-young-diagram}, the map $\Phi_k^Y$ replaces:
    \begin{enumerate}
        \item A non-trivial maximal $(k+1,k)$-strip \Tableau[no border,box font=\tiny]{{k{+}1},{k}} with a vertical triple \Tableau[no border, box font=\tiny]{{k{+}2},{k{+}1},{k}}. This increases the column height by $1$.
        \item A trivial maximal $(k+1,k)$-strip (single $k+1$) with a single $k+2$. This preserves height.
    \end{enumerate}
    Thus, the total length increases by exactly the number of non-trivial maximal  $(k+1,k)$-strips, $k(\lambda)$.
\end{proof}
\begin{Lemma}\label{lm:k-lambda-zero-bound}
    Fix a subdivision datum $\subdatum[\Lambda_x][x]$ and take $\lambda\in \Par[\Lambda_x]_{\alpha}$. 
    Let $\rho \in I$ be the unique integer satisfying $\rho \equiv x-k \pmod e$.
    Then $k(\lambda)=0$ if and only if $\ell(\lambda) \le \rho$.
\end{Lemma}

\begin{proof}
By \autoref{lm:counting_formula}, the condition $k(\lambda)=0$ is equivalent to 
$$
\frac{\ell(\lambda)-1-\rho}{e} < 0.
$$
Since $e > 0$, this simplifies to $\ell(\lambda) - 1 - \rho < 0$, or $\ell(\lambda) \le \rho$.
\end{proof}
The following corollary will be useful later.
\begin{Corollary}\label{cor:k-lambda-zero-bound-2}
    Fix a subdivision datum $\subdatum[\Lambda_x][x]$ and take $\lambda\in\Par[\Lambda_x]_{\alpha}$. Take an abacus subdivision datum $\absubdatum$ for $\lambda$. Then $k(\lambda)=0$ if and only if $\ell(\lambda)\leq e-d$.
\end{Corollary}
\begin{proof}
    Let $\rho \in I$ be the unique integer satisfying $\rho \equiv x-k \pmod e$.
    By definition, $a+d=ce+k$ and $a\equiv x\pmod e$. Hence $x-k\equiv ce-d\equiv -d\pmod e$, and it follows that $\rho=e-d$. The statement then follows from \autoref{lm:k-lambda-zero-bound}.
\end{proof}

To compare $\Phi^A_k$ and $\Phi^{Y}_k$ explicitly, we analyze the row lengths. We introduce three useful functions, which will be used to describe the subdivision map in terms of beta numbers. We fix a subdivision datum $\subdatum[\Lambda_x][x]$ for convenience.

\begin{Definition}\label{def:counting-functions-NkM}
    Take $M\in \Z_{\geq 0}$ and let $N_k(M)$ be the number of integers $y$ such that $0 \le y \le M$ and $y \equiv k \pmod e$. An explicit formula is
    \begin{equation}\label{eq:NkM-formula}
        N_k(M)=\left\lfloor \frac{M-k}{e} \right\rfloor + 1.
    \end{equation}
\end{Definition}
\begin{Definition}\label{def:eps-d}
    For $u\in I=\{0,1,\dots,e-1\}$, define the step function on $I$ by
    \begin{equation}\label{eq:eps-def}
      \varepsilon_u(r):=
      \begin{cases}
        0,& 0\le r<u,\\
        1,& u\le r\le e-1,
      \end{cases}
      \qquad(r\in I).
    \end{equation}
\end{Definition}

\begin{Definition}\label{def:iota}
    For $u\in I=\{0,1,\dots,e-1\}$, define a map $\iota_u:\Z_{\ge0}\to \Z_{\ge0}$ as follows. For $n\in\Z_{\ge0}$, write $n=qe+r$ with $q\in\Z_{\ge0}$ and $r\in I$, and set    \begin{equation}\label{eq:iota-def}
      \iota_u(n):=q(e+1)+r+\varepsilon_u(r).
    \end{equation}
\end{Definition}

The subdivision map $\Phi_k^A$ acts on the corresponding beta set essentially by applying $\iota_k$. More precisely, we have the following.
\begin{Lemma}\label{lm:subdivision-beta-numbers}
    Fix a subdivision datum $\subdatum[\Lambda_x][x]$. Take $\lambda\in\Par[\Lambda_x]_{\alpha}$ and an abacus subdivision datum $\absubdatum$ for $\lambda$.
    Let $B(\lambda;a)$ and $B(\Phi^A_k(\lambda);a')$ be the corresponding beta sets (see \autoref{eq:beta-set}). Set $\bT:=\bT_{e+1,k,c}$ as in \autoref{eq:T-ekc}. Then
    $$
          B(\Phi^A_k(\lambda);a')=\iota_k\bigl(B(\lambda;a)\bigr)\sqcup \bT.
    $$
\end{Lemma}
\begin{proof}
    This is a direct translation of \autoref{def:subdivision-abacus}, using the equivalence between the $e$-abacus with $a'$ beads and the $a'$-beta numbers of the corresponding partition.
\end{proof}
\begin{Lemma}\label{lm:Nk-epsilon}
    Write $M = qe + r$ with $0 \le r < e$, then $N_k(M) = q + \varepsilon_k(r)$.
\end{Lemma}

\begin{proof}
    By definition, $N_k(M) = \lfloor \frac{M-k}{e} \rfloor + 1$. Substituting $M=qe+r$:
    $$
    N_k(M) = q + \left\lfloor \frac{r-k}{e} \right\rfloor + 1.
    $$
    If $r \ge k$, the floor term is $0$, yielding $q+1$. If $r < k$, the floor term is $-1$, yielding $q$.
    This matches the definition of $q + \varepsilon_k(r)$.
\end{proof}

\begin{Corollary}\label{cor:relation-NkM-iota}
    Take $n\in\Z_{\geq 0}$ and $k\in\{0,1,\cdots,e-1\}$, we have $\iota_k(n)=n+N_k(n)$. \hfill$\square$
\end{Corollary}

For $\lambda\in\Par[\Lambda_x]_{\alpha}$, define $m_r:=m_{r,k}(\lambda)$ to be the number of $k$-nodes in the $r$-th row of $[\lambda]$. Take an abacus subdivision datum $\absubdatum$ for $\lambda$. Then, since $a\equiv x\pmod{e}$, $m_r$ equals the number of integers in $\{a-r+1,a-r+2,\cdots,a-r+\lambda_r=\beta_r\}$ that are congruent to $k$ modulo $e$. Thus,
\begin{equation}\label{eq:m_r-NkM-expression}
    m_r=N_k(\beta_r)-N_k(a-r).
\end{equation}

We split the proof of the equivalence of the two definitions (\autoref{thm:equivalence-of-two-definitions-subdivision}) into two parts, depending on whether $k(\lambda)=0$.
\subsubsection{\texorpdfstring{The case \(k(\lambda)=0\)}{k(lambda)=0 case}}\label{subsubsec:klam-zero-case}
Throughout this section, we fix a subdivision datum $\subdatum[\Lambda_x][x]$.
\begin{Lemma}\label{lm:node-count}
    Take $\lambda\in \Par[\Lambda_x]_{\alpha}$ and assume $k(\lambda)=0$. Let $\absubdatum$ be an abacus subdivision datum for $\lambda$.
    For $1 \le r \le \ell(\lambda)$, let $\beta_r = \lambda_r - r + a = q_r e + t_r$ with $0 \le t_r < e$. Then $m_r = q_r - c + \varepsilon_k(t_r)$.
\end{Lemma}

\begin{proof}
    The residues in $[\lambda]_r$ are $\{a-r+1,\cdots,a-r+\lambda_r=\beta_r\}$ mod $e$. Hence the number of $k$-nodes is $m_r = N_k(\beta_r) - N_k(a-r)$.
    By \autoref{lm:Nk-epsilon}, the first term is $N_k(\beta_r) = q_r + \varepsilon_k(t_r)$.
    
    For the second term $N_k(a-r)$, since $a-k = ce - d$, we have:
    $$
    \frac{a-r-k}{e} = \frac{ce - d - r}{e} = c + \frac{-(d+r)}{e}.
    $$
    Since $k(\lambda)=0$, by \autoref{cor:k-lambda-zero-bound-2}, this implies $\ell(\lambda) \le e-d$. Since $1 \le r \le \ell(\lambda)$, we have $d+1 \le d+r \le e$.
    Consequently,
    $$
    -1 \le \frac{-(d+r)}{e} < 0 \implies \left\lfloor \frac{-(d+r)}{e} \right\rfloor = -1.
    $$
    Thus by \autoref{eq:NkM-formula},
    $$
    N_k(a-r) = \left( c - 1 \right) + 1 = c.
    $$
    Substituting these back into the expression for $m_r$ proves the lemma. 
\end{proof}


\begin{Lemma}\label{lm:beta-mapping}
    Take $\lambda\in\Par[\Lambda_x]_{\alpha}$ and an abacus subdivision datum $\absubdatum$ for $\lambda$. Let $\{\beta_r\}_{r\ge 1}$ and $\{\beta'_r\}_{r\ge 1}$ be the sets of $a$-beta (respectively $a'$-beta) numbers for $\lambda$ and $\Phi^A_k(\lambda)$. If $k(\lambda)=0$, then
    $$
    \beta'_r = \iota_k(\beta_r), \qquad 1 \le r \le \ell(\lambda).
    $$
\end{Lemma}

\begin{proof}
    By \autoref{lm:subdivision-beta-numbers}, the set of beta numbers $\{\beta'_j\}_{j\ge 1}$ for $\Phi^A_k(\lambda)$ is the union of the image set $S=\{\iota_k(\beta_r)\}_{r\ge 1}$ and the set $\bT:=\bT_{e+1,k,c}$ defined in \autoref{eq:T-ekc}, which corresponds to the beads on the newly inserted runner. We only need to prove that the order of the beta numbers is preserved.
    
    There are two trivial cases. If $c=0$, then $\bT=\emptyset$, so the result holds. If $c=1$ and $k=0$, then $\bT=\{0\}$ and the conclusion holds naturally.
    
    Assume now that we are not in these trivial cases. Let $X = (c-1)e+k-1\geq 0$ and set $r_0 = a - X$. Since $ce = a+d-k$, we have:
    $$
    \begin{aligned}
    r_0 &= a - X \\
    &= a - (a+d-k - e + k - 1) \\
    &= a - a - d + e + 1 \\
    &= e - d + 1.
    \end{aligned}
    $$
    By \autoref{cor:k-lambda-zero-bound-2}, $k(\lambda)=0$ implies $\ell(\lambda) \le e-d$. Thus $\ell(\lambda) < r_0$.
    Since $r_0 > \ell(\lambda)$, the corresponding part $\lambda_{r_0}$ is zero. Therefore, the beta number  $\beta_{r_0}=\lambda_{r_0}+a-r_0 = a - r_0 = X$.
    
    Since the sequence of beta numbers is strictly decreasing, for any $1 \le r \le \ell(\lambda)$, we have $r < r_0$, which implies:
    $$
    \beta_r > \beta_{r_0} = (c-1)e + k - 1.
    $$
    Therefore, $\beta_r \ge (c-1)e + k$. Applying the strictly increasing map $\iota_k$, we get:
    $$
    \iota_k(\beta_r) \ge \iota_k((c-1)e + k) = (c-1)(e+1) + k + 1 > \max(\bT).
    $$
    This shows that the images of the first $\ell(\lambda)$ beta numbers are strictly larger than any element in $\bT$. Consequently, they occupy the first $\ell(\lambda)$ positions in the sorted set $\{\beta'_j\}$, proving $\beta'_r = \iota_k(\beta_r)$ for these indices.
\end{proof}
\begin{Corollary}\label{cor:length_preserved}
    Take $\lambda\in\Par[\Lambda_x]_{\alpha}$ and assume $k(\lambda)=0$. Then $\ell(\Phi^A_k(\lambda)) = \ell(\lambda)$.
\end{Corollary}

\begin{proof}
    Set up as in the proof of \autoref{lm:beta-mapping}. Define the values $X = (c-1)e + k - 1$ and $Y = (c-1)(e+1) + k$.
    Let $B = \{\beta_r\}_{1\le r \le a}$ be the set of $a$-beta numbers for $\lambda$. We partition $B$ into a "head" $B_{>X} = \{ \beta \in B \mid \beta > X \}$ and a "tail" $B_{\le X} = \{ \beta \in B \mid \beta \le X \}$.
    Similarly, let $B' = \{\beta'_r\}_{1\le r \le a'}$ be the set of $a'$-beta numbers for $\mu:=\Phi^A_k(\lambda)$, partitioned into $B'_{>Y}$ and $B'_{\le Y}$.
    
    Let $\bT=\bT_{e+1,k,c}$ as in \autoref{eq:T-ekc}. By definition, the map $\iota_k$ induces a bijection between $\{0, \dots, X\}$ and $\{0, \dots, Y\} \setminus \bT$, while $\bT$ fills the gaps. Thus, $\iota_k(\{0, \dots, X\}) \cup \bT = \{0, \dots, Y\}$.
    
    First, consider the tail. Let $r_{0} = e - d + 1$. In the proof of \autoref{lm:beta-mapping}, we established that $X = a - r_{0}$.
    Since $k(\lambda)=0 \implies \ell(\lambda) < r_{0}$, we have $\lambda_r = 0$ for all $r \ge r_{0}$ and thus $\beta_r = a - r \le X$.
    Therefore, $B_{\le X}$ is exactly the set of integers $\{0, \dots, X\}$ and
    $$
    B'_{\le Y} = \iota_k(B_{\le X}) \cup \bT = \iota_k(\{0, \dots, X\}) \cup \bT = \{0, \dots, Y\}.
    $$
    Since $Y = a' - r_{0}$, this set $\{0, \dots, a' - r_{0}\}$ corresponds exactly to the beta numbers $a-r$ for indices $r \ge r_{0}$.
    Thus, $\mu_r = 0$ for all $r \ge r_{0}$.
    
    Next, consider the head. For $r < r_{0}$, we have $\beta_r \in B_{>X}$. By \autoref{lm:beta-mapping}, $\beta'_r = \iota_k(\beta_r)$ for $1\leq r\leq \ell(\lambda)$. We remark that the proof shows that this is actually true for all $1\leq r\leq r_0-1$.
    
    
    Since $1 \le r \le e-d$, we have $d+1 \le d+r \le e$, which implies $\lfloor -(d+r)/e \rfloor = -1$.
    Thus $N_k(a-r) = c - 1 + 1 = c$.
    By \autoref{cor:relation-NkM-iota}, and using $a'=a+c$:
    $$
    \iota_k(a-r) = (a-r) + N_k(a-r) = a - r + c = a' - r.
    $$
    Using the fact that $\iota_k$ is strictly increasing:
    $$
    \lambda_r > 0 \iff \beta_r > a - r \iff \iota_k(\beta_r) > \iota_k(a - r) \iff \beta'_r > a' - r \iff \mu_r > 0.
    $$
    This equivalence holds for all $r < r_{0}$. Since $\ell(\lambda) < r_{0}$, this confirms that exactly the first $\ell(\lambda)$ parts of $\Phi^A_k(\lambda)$ are non-zero.
    Thus $\ell(\Phi^A_k(\lambda)) = \ell(\lambda)$.
\end{proof}


\begin{Proposition}\label{prop:equivalence-abacus-tableau-klam-zero}
    Take $\lambda\in\Par[\Lambda_x]_{\alpha}$ and assume $k(\lambda)=0$. Then $\Phi^A_k(\lambda) = \Phi^{Y}_k(\lambda)$.
\end{Proposition}

\begin{proof}
    First, consider $\Phi^{Y}_k(\lambda)$. Since $k(\lambda)=0$, by \autoref{lm:length-subdivision}, $\ell(\Phi^{Y}_k(\lambda))=\ell(\lambda)$. Moreover, let $m_r$ be the number of $k$-nodes in the $r$-th row of $[\lambda]$, it follows directly from \autoref{def:subdivision-young-diagram} that $\Phi^{Y}_k(\lambda)_r=\lambda_r+m_r$ for $1\le r\le \ell(\lambda)$.

    By \autoref{cor:length_preserved}, we also have $\ell(\Phi^A_k(\lambda))=\ell(\lambda)$. It remains to show that the parts agree for $1\le r\le \ell(\lambda)$.

    Fix an abacus subdivision datum $\absubdatum$ for $\lambda$. Let $\mu=\Phi^A_k(\lambda)$. The $r$-th part is $\mu_r=\beta'_r+r-a'$. By \autoref{lm:beta-mapping}, $\beta'_r=\iota_k(\beta_r)$ for $1\le r\le \ell(\lambda)$. Since $a'=a+c$, we have
    $$
        \mu_r-\lambda_r=\big(\iota_k(\beta_r)+r-(a+c)\big)-(\beta_r+r-a)=\iota_k(\beta_r)-\beta_r-c.
    $$
    Write $\beta_r=q_re+t_r$. By \autoref{def:iota}, $\iota_k(\beta_r)-\beta_r=q_r+\varepsilon_k(t_r)$. Therefore,
    $$
    \mu_r-\lambda_r=q_r+\varepsilon_k(t_r)-c.
    $$
    By \autoref{lm:node-count}, this is exactly $m_r$. Hence $\mu_r=\lambda_r+m_r=\Phi^{Y}_k(\lambda)_r$ for $1\le r\le \ell(\lambda)$, and so $\Phi^A_k(\lambda)=\Phi^{Y}_k(\lambda)$.
\end{proof}

\subsubsection{\texorpdfstring{The case $k(\lambda)>0$}{k(lambda)>0 case}}\label{subsubsec:klam-nonzero-case}\mbox{}

Throughout this section, we fix a subdivision datum $\subdatum[\Lambda_x][x]$ and a partition $\lambda\in\Par[\Lambda_x]_{\alpha}$. We assume $s:=k(\lambda)>0$. Hence, the Young diagram $[\lambda]$ contains $s$ non-trivial maximal $(k+1,k)$-strips. By \autoref{lm:initial-nodes}, these strips start in the first column.

For a non-trivial maximal $(k+1,k)$-strip $S$, let $k_S$ and $r_S$ be the row indices of the initial node and the terminal node of $S$, respectively. Non-triviality of $S$ is equivalent to $r_S>k_S$.

Let $S_1,\dots,S_s$ be the non-trivial maximal $(k+1,k)$-strips in $[\lambda]$. Set $k_i:=k_{S_i}$ and $r_i:=r_{S_i}$, and reorder the strips so that
\[
    k_1<k_2<\cdots<k_s
    \qquad\text{and}\qquad
    r_1\le r_2\le\cdots\le r_s.
\]

\begin{Definition}\label{def:tracing-functions-k}
    We define the \textit{strip-tracing functions} $f, g: \Z_{\ge 1} \to \Z_{\ge 0}$ by:
    \begin{equation}\label{eq:g-def-k}
        g(j) = \#\{i \mid r_i < j\}
    \end{equation}
    \begin{equation}\label{eq:f-def-k}
        f(j) =
        \begin{cases}
            0 & \text{if } 1 \le j \le k_1, \\
            i & \text{if } k_i < j \le k_{i+1},
        \end{cases}
    \end{equation}
    where we set $k_{s+1} = \infty$.
\end{Definition}

By definition, $g(j)$ is the number of maximal $(k+1,k)$-strips in $[\lambda]$ that end strictly before row $j$. Recall also that $m_j$ denotes the number of $k$-nodes in the $j$-th row of $[\lambda]$ (as defined in \autoref{subsubsec:klam-zero-case}).

By \autoref{lm:length-subdivision}, the subdivision map $\Phi^{Y}_k$ increases the number of rows by $s=k(\lambda)$. The rows of the new diagram $[\Phi^{Y}_k(\lambda)]$ can be naturally classified by the residue of the last node in each row:
\begin{itemize}
    \item The rows that \textbf{do not} end with residue $k$ correspond to the original rows of $\lambda$.
    \item The rows that \textbf{do} end with residue $k$ correspond to the new rows created by enlarging the strips.
\end{itemize}

\begin{Lemma}\label{lm:tableau-row-expansion-k}
    Take $\lambda\in\Par[\Lambda_x]_{\alpha}$ and let $\lambda^+ = \Phi^{Y}_k(\lambda)$. For $1 \le j \le \ell(\lambda)$, the following hold:
    \begin{enumerate}
        \item\label{item:old-row} The row in $\lambda^+$ that does not end with residue $k$ is located at index $j+g(j)$ and has length
        \begin{equation}\label{eq:original_row_length_k}
            \lambda^+_{j+g(j)} = \lambda_j + m_j + g(j) - f(j).
        \end{equation}
        and we say this row is corresponding to the original $j$-th row of $\lambda$.
        \item\label{item:new-row} The $s$ new rows that end with residue $k$ are located at indices $r_i + i$ for $1 \le i \le s$.
    \end{enumerate}
\end{Lemma}
\begin{proof}
    We first record the only way that a row of $[\lambda^+]$ can end with residue $k$.
    Let $A$ be the last node of a row of $[\lambda]$. Tracking $A$ by the local rules of
    $\Phi^{Y}_k$ in \autoref{def:subdivision-young-diagram}, one checks:
    
    (i) if $\res(A)\notin\{k,k+1\}$ then its image has residue $\Phi_k\big(\res(A)\big)\neq k$;
    
    (ii) if $A$ is the terminal node of a maximal $(k,k+1)$-strip and $\res(A)=k$, then $A$ expands
    to a horizontal pair \Tableau[no border,box font=\tiny]{{k}{k{+}1}}, so the row ends with residue $k+1$;
    
    (iii) if $\res(A)=k+1$ and $A$ is the terminal node of a maximal $(k,k+1)$-strip or a maximal
    $(k+1,k)$-strip, then $A$ is replaced by a single $(k+2)$-node, so the row ends with
    residue $k+2$;
    
    (iv) if $A$ is the terminal node of a non-trivial maximal $(k+1,k)$-strip and $\res(A)=k$, then
    $A$ is replaced by a vertical pair \Tableau[no border,box font=\tiny]{{k{+}1},{k}}, which
    creates a new row whose last node has residue $k$.
    
    Hence the rows of $\lambda^+$ ending with residue $k$ are exactly the new rows created from
    the terminal nodes of non-trivial maximal $(k+1,k)$-strips.
        
    For \autoref{item:new-row}, consider the $i$th non-trivial maximal $(k+1,k)$-strip $S_i$.
    Its terminal node lies in row $r_i$ of $[\lambda]$, and by (iv) it produces exactly one new row.
    The order of the original rows is preserved, and the new row is inserted immediately after
    the image of row $r_i$. Since exactly $i-1$ new rows are inserted strictly above row $r_i$,
    the image of row $r_i$ sits at index $r_i+(i-1)$, so the new row sits at index $r_i+i$.
    
    For \autoref{item:old-row}, fix $1\le j\le \ell(\lambda)$. The image of the original $j$th row is shifted
    downward by the number of new rows inserted strictly above it, which is $g(j)$ by definition \autoref{eq:g-def-k},
    so it is located at index $j+g(j)$.
    
    To compute the length of the image of the $j$th row. We firstly ignore the maximal $(k+1,k)$-strips for a moment, every $k$-node in row $j$ would expand to a
    horizontal pair $(k,k+1)$ in the same row, so the length would increase by $1$ for each
    $k$-node. This gives the tentative value $\lambda_j+m_j$.
    
    However, this tentative value $\lambda_j+m_j$ is too large, because not every $k$-node in
    row $j$ expands horizontally. The problematic ones are those $k$-nodes in row $j$ that lie on a non-trivial
    maximal $(k{+}1,k)$-strip that passes through row $j$. Indeed, for each such non-trivial maximal
    $(k{+}1,k)$-strip $S$ one has $k_S<j\le r_S$, and $S\cap [\lambda]_j$
    contains exactly one $k$-node (it is the left node of a horizontal \Tableau[no border,box font=\tiny]{{k}{k{+}1}} if
    $k_S<j<r_S$, and it is the single terminal $k$-node if $j=r_S$). 
    
    In the construction of
    $\Phi_k^Y$ by \autoref{def:subdivision-young-diagram}, this particular $k$-node does not contribute an extra node to the image of row $j$: as \Tableau[no border,box font=\tiny]{{k{+}1},{k}} is replaced by \Tableau[no border,box font=\tiny]{{k{+}2},{k{+}1},{k}}, the $k$-node in this row is replaced by a single $(k+1)$-node
    
    Hence one must subtract $1$ for each such strip that passes through row $j$.
    By \autoref{def:tracing-functions-k}, the number of such strips is $f(j)-g(j)$.
    Therefore
    \[
    \lambda^+_{j+g(j)}
    =(\lambda_j+m_j)-(f(j)-g(j))
    =\lambda_j+m_j+g(j)-f(j),
    \]
    as required.
\end{proof}

The row indices $k_i$ can be computed explicitly as follows. Take an abacus subdivision datum for $\lambda$. Then, by \autoref{def:k_lambda} and \autoref{cor:k-lambda-zero-bound-2}, we have $k_1=e-d$ and hence $k_j=je-d$ for $1\le j\le s=k(\lambda)$.

Moreover, $f(t)$ can be written in terms of the function $N_k(M)$ as follows. Suppose that $f(t)=j$. Then, by \autoref{def:tracing-functions-k}, we have $je-d<t\le (j+1)e-d$, and hence $je<t+d\le (j+1)e$. We compute
$$
N_k(a-t)=\left\lfloor \frac{a-t-k}{e}\right\rfloor+1
=\left\lfloor \frac{ce-(d+t)}{e}\right\rfloor+1
=c+\left\lfloor \frac{-(t+d)}{e}\right\rfloor+1
=c-(j+1)+1=c-j.
$$
Hence,
\begin{equation}\label{eq:f-equal-NkM}
    f(t)=c-N_k(a-t),\quad (t\geq 1).
\end{equation}

\begin{Remark}\label{rmk:relation-of-bead-node}
     Consider the $e$-abacus of $\lambda$ with $a$ beads, and the Young diagram $[\lambda]$. For $1\le i\le \ell(\lambda)$, the bead corresponding to the beta number $\beta_i=\lambda_i+a-i$ lies on the $k$-runner if and only if $\lambda_i+a-i\equiv k\pmod{e}$. The last node of the $i$-th row in $[\lambda]$ has residue $a+\lambda_i-i\pmod{e}$. Hence the runner label agrees with the residue of the last node in the corresponding row.
\end{Remark}

\begin{Proposition}\label{prop:equivalence-abacus-tableau-klam-nonzero-general}
    Take $\lambda\in\Par[\Lambda_x]_{\alpha}$ and assume $k(\lambda)>0$. Then $\Phi^A_k(\lambda) = \Phi^{Y}_k(\lambda)$.
\end{Proposition}

\begin{proof}
    Fix an abacus subdivision datum $\absubdatum$ for $\lambda$.
    Let $\lambda^+=\Phi_k^Y(\lambda)$ and $\mu=\Phi_k^A(\lambda)$. 
    
    For $1\le r\le a$, set $\beta_r:=\lambda_r+a-r$, and let $B(\lambda;a)=\{\beta_r\mid 1\le r\le a\}$ be the $a$-beta set of $\lambda$.
    For $1\le r\le a'$, set $\beta_r^+:=\lambda_r^+ + a' - r$, and let $B(\lambda^+;a')=\{\beta_r^+\mid 1\le r\le a'\}$ be the $a'$-beta set of $\lambda^+$. Similarly, $B(\mu;a')$ is the $a'$-beta set of $\mu$.  
    
    To prove $\lambda^+=\mu$, it suffices to show their $a'$-beta sets $B(\mu;a')$ and $B(\lambda^+;a')$ are equal. By \autoref{lm:subdivision-beta-numbers}, we have
    \[
    B(\mu;a')=\iota_k\bigl(B(\lambda;a)\bigr)\cup \bT,
    \]
    where $\bT:=\bT_{e+1,k,c}=\{q(e+1)+k\mid 0\le q\le c-1\}$ as in \autoref{eq:T-ekc}. 
    
    \medskip
    \noindent \textit{Step 1.} We first show $\iota_k\bigl(B(\lambda;a)\bigr)\subseteq B(\lambda^+;a')$.
    
    Fix $1\le j\le a$. By \autoref{lm:tableau-row-expansion-k}, the image of the original $j$th row lies
    in row $j+g(j)$ of $[\lambda^+]$ and has length $\lambda_j+m_j+g(j)-f(j)$. Hence
    \[
    \beta_{j+g(j)}^+
    =\bigl(\lambda_j+m_j+g(j)-f(j)\bigr)+a'-(j+g(j))
    =\beta_j+m_j-f(j)+c.
    \]
    Using \eqref{eq:m_r-NkM-expression} (applied with $r=j$) gives
    $m_j=N_k(\beta_j)-N_k(a-j)$, so
    \[
    \beta_{j+g(j)}^+
    =\beta_j+N_k(\beta_j)-N_k(a-j)-f(j)+c.
    \]
    By \autoref{eq:f-equal-NkM}, $N_k(a-j)=c-f(j)$.
    Substituting this yields
    \[
    \beta_{j+g(j)}^+=\beta_j+N_k(\beta_j)=\iota_k(\beta_j).
    \]
    where the last equality follows from \autoref{cor:relation-NkM-iota}.
    Thus every element of $\iota_k(B(\lambda;a))$ occurs among the beta numbers of $\lambda^+$, so
    $\iota_k(B(\lambda;a))\subseteq B(\lambda^+;a')$.

    \medskip\noindent \textit{Step 2.} Secondly, we show that $\bT\subseteq B(\lambda^+;a')$.
    
    Let $M:=a'-\ell(\lambda^+)-1$.  Since $\lambda_r^+=0$ for all $r>\ell(\lambda^+)$, we have
    \[
    \{\beta_r^+\mid \ell(\lambda^+)<r\leq a'\}=\{a'-r\mid \ell(\lambda^+)<r\leq a'\}=\{0,1,\dots,M\}.
    \]
    Hence every element of $\bT$ that is at most $M$ lies in $B(\lambda^+;a')$.
    
    It remains to consider the elements of $\bT$ that are larger than $M$.
    Let $t=q(e+1)+k\in \bT$ and assume $t>M$.
    So $t$ cannot come from a zero row of $\lambda^+$ and must occur as
    $\beta_r^+$ for some $1\le r\le \ell(\lambda^+)$. Moreover, by \autoref{rmk:relation-of-bead-node}, $t$ can correspond only to a row whose last node has residue $k$.
    
    By \autoref{lm:tableau-row-expansion-k}, the rows of $[\lambda^+]$ ending with residue $k$ are precisely
    the newly inserted rows, and the $i$th inserted row occurs at index $r_i+i$. Moreover, this row is
    created from the terminal $k$-node of some non-trivial maximal $(k+1,k)$-strip: this $k$-node
    is replaced by a vertical pair with residues $k+1$ above $k$. Hence the inserted row $r_i+i$ has the
    same length as the row immediately above it (the image of row $r_i$), and therefore its $a'$-beta
    number is one less than the $a'$-beta number of that row. Since $g(r_i)=i-1$, the row above $r_i+i$
    is $r_i+i-1$, and Step~1 gives $\beta^+_{r_i+i-1}=\iota_k(\beta_{r_i})$. Thus
    \[
        \beta^+_{r_i+i}=\beta^+_{r_i+i-1}-1=\iota_k(\beta_{r_i})-1.
    \]
    Because the strip ends at a $k$-node, the last node of row $r_i$ has residue $k$ in $[\lambda]$, so
    $\beta_{r_i}\equiv k\pmod e$. Writing $\beta_{r_i}=qe+k$ gives
    $\iota_k(\beta_{r_i})=q(e+1)+k+1$, and therefore
    \[
        \beta^+_{r_i+i}=q(e+1)+k\in \bT.
    \]
    Hence every inserted row contributes an element of $\bT$ to $B(\lambda^+;a')$. In total, there are $s=k(\lambda)$ such elements. It remains to prove that there are exactly $s$ elements of the form $t=q(e+1)+k>M$ where $0\leq q\leq c-1$. We have:
    \[
        q(e+1)+k>M=a'-\ell(\lambda^+)-1=(a+c)-\big(\ell(\lambda)+s\big)-1=(ce+k-d)+c-\big(\ell(\lambda)+s\big)-1
    \]
    which is equivalent to
    \begin{equation}\label{eq:number-of-new-created-row}
        X:=\frac{d+\ell(\lambda)+s+1}{e+1} >c-q
    \end{equation}
    Let $\rho\in I$ be the unique integer satisfying $\rho\equiv x-k\pmod{e}$. Then by \autoref{lm:counting_formula}, we have:
    \[
        (s-1)e\leq \ell(\lambda)-1-\rho< se
    \]
    and, by the proof of \autoref{cor:k-lambda-zero-bound-2}, $\rho=e-d$. Hence
    \[
        s(e+1)\leq \ell(\lambda)-1+d+s< s(e+1)+e
    \]
    which implies
    \[
        s(e+1)+2\leq \ell(\lambda)+1+d+s\leq s(e+1)+e+1
    \]
    which implies
    \[
        s+\frac{2}{e+1}\leq \frac{d+\ell(\lambda)+s+1}{e+1}\leq  s+1.
    \]
    This means $X\in (s,s+1]$.
    In particular, since $1\leq c-q\leq c$ and $c\ge s$ (as the inserted rows correspond to a subset of $\bT$), there are exactly $s=k(\lambda)$ values satisfying \autoref{eq:number-of-new-created-row}, as desired. Hence $\bT\subset B(\lambda^+;a')$.
    
    \medskip\textit{Step 3} We have shown that $\iota_k(B(\lambda;a))\subseteq B(\lambda^+;a')$ and $\bT\subseteq B(\lambda^+;a')$ by the last two steps.
    By \autoref{lm:subdivision-beta-numbers},
    \[
        B(\mu;a')=\iota_k(B(\lambda;a))\sqcup \bT\subseteq B(\lambda^+;a').
    \]
    Comparing cardinalities gives equality of the two beta sets, and consequently $\lambda^+=\mu$.
\end{proof}

We can now state the main theorem of this section.
\begin{Theorem}\label{thm:equivalence-of-two-definitions-subdivision}
    \autoref{def:subdivision-young-diagram} and \autoref{def:subdivision-abacus} are equivalent.
\end{Theorem}
\begin{proof}
    This follows from \autoref{prop:equivalence-abacus-tableau-klam-zero} and \autoref{prop:equivalence-abacus-tableau-klam-nonzero-general}.
\end{proof}
\subsection{Subdivision on standard tableaux}\label{subsec:subdivision-standard-tableaux}
In this section, we show that if a multipartition is $k$-horizontal (see \autoref{def:k-horizontal}), then the subdivision on partitions in \autoref{def:subdivision-young-diagram} extends to a map on the set of (row)-standard tableaux. Moreover, it preserves the degree of a standard tableau; see \autoref{thm:degree-invariance-standard-tableau}.

\smallskip
Fix a subdivision datum $\subdatum$ and a $k$-horizontal $\ell$-partition $\blam\in\Par[\bkappa]_{\alpha}$. Let $\bmu:=\Phi_k(\blam)$ be the image. By \autoref{def:subdivision-young-diagram}, in this case $\bmu$ is obtained by inserting an extra node immediately to the right of every $k$-node in $[\blam]$. That is, if a $k$-node is not the last node in its row, then
\begin{center}
    \tikzset{
            C/.style={fill=cyan,text=white},
            O/.style={fill=OrangeRed, text=white}
                }
    \Tableau[no border,box height=0.7,box width=0.7]{[C]k[O]{k{+}1}}$\longrightarrow$\Tableau[no border,box height=0.7,box width=0.7]{[C]k{k{+}1}[O]{k{+}2}}
\end{center}
and otherwise,
\begin{center}
    \tikzset{
            C/.style={fill=cyan,text=white},
            O/.style={fill=OrangeRed, text=white}
                }
    \Tableau[no border,box height=0.7,box width=0.7]{[C]k}$\longrightarrow$\Tableau[no border,box height=0.7,box width=0.7]{[C]k{k{+}1}}\\
\end{center}

We distinguish the nodes of $[\bmu]$ coming from $[\blam]$ from those inserted in the construction. Consider the initial tableaux $T^{\blam}$ and $T^{\bmu}$, and set $\bi^{\blam}:=\bi^{T^{\blam}}$.
Let $\phi:=\phi_{\bi^{\blam}}$ be the position-tracing function associated to $(k,\bi^{\blam})$;
see \autoref{def:position-tracing-function}.
For each node $A\in[\blam]$ with $T^{\blam}(A)=t$ (where $1\le t\le d=\height(\alpha)$), define $A^{+}\in[\bmu]$ by
$T^{\bmu}(A^{+})=\phi(t)$.

The nodes of the form $A^{+}$ for some $A\in [\blam]$ are called \emph{old nodes}, and the remaining nodes of $[\bmu]$ are \emph{new nodes}.
Every new node is a $(k+1)$-node, lying immediately to the right of an old $k$-node; if $B$ is such a new node and
$A^{+}$ is the old $k$-node immediately to its left, write $B=A^{\sharp}$.
In particular, if $A\in[\blam]$ is a $k$-node with $T^{\blam}(A)=t$, then $A^{\sharp}$ satisfies $T^{\bmu}(A^{\sharp})=\phi(t)+1$.

We can be more precise. Let $A=(m,r,c)\in[\blam]$. When we apply $\Phi_k$, since the only non-trivial maximal strips are $(k,k+1)$-strips, the row index of every node is unchanged, and the column index increases by the number of $k$-nodes in the same row that occur strictly before $A$. This number is $\phi(t)-\phi(t_{m,r})$ where
\[
    t_{m,r}
    :=T^{\blam}(m,r,1)=
    1+\sum\limits_{1\le i\le m-1}\bigl|\blam^{(i)}\bigr|
    +\sum\limits_{1\le i\le r-1}\blam^{(m)}_{i}.
\]
Hence
\[
    A^+=(m,r,c+\phi(t)-\phi(t_r))\in[\bmu].
\]

\medskip
Let $\oalpha:=\Phi_k(\alpha)=\alpha(\bmu)$ be the residue content of $\bmu$ and let $d':=\height(\oalpha)$. Take any $T\in\RST{\blam}$, let $\bi^T$ be the residue sequence of $T$ and $\phi^T:=\phi_{\bi^T}$ be the position-tracing function associated to $(k,\bi^{T})$. We
construct a map $T'$ from $[\bmu]$ to $\{1,\cdots,d'\}$ as follows:
\begin{enumerate}
    \item if $B\in [\bmu]$ is a old node of the form $A^+$ for $A\in [\blam]$. Assume $T(A)=t$, define $T'(B):=\phi^T(t)$.
    \item if $B\in [\bmu]$ is a new node of the form $A^{\sharp}$ for $A\in [\blam]$. Assume $T(A)=t$, define $T'(B):=\phi^T(t)+1$.
\end{enumerate}
\begin{Lemma}\label{lm:image-bmu-tableau}
    Fix a subdivision datum $\subdatum$, $k$-horizontal $\blam\in\Par[\bkappa]_{\alpha}$ such that $T\in\RST{\blam}$ and let $\bmu:=\Phi_k(\blam)$, then $T'$ is a $\bmu$-tableau.
\end{Lemma}
\begin{proof}
    It suffices to show that $T'$ is a bijection from $[\bmu]$ to $\{1,\dots,d'\}$. This follows from \autoref{def:position-tracing-function} and the definition of $T'$. Indeed, $\phi^T$ is an embedding from $\{1,\cdots,d\}$ to $\{1,\cdots,d'\}$, and the elements of $\{1,\cdots,d'\}$ not in its image are precisely those of the form $\phi^T(t)+1$ with $\bi^T_t=k$.
\end{proof}

The following is an easy consequence of \autoref{def:position-tracing-function}.
\begin{Lemma}\label{lm:row-standardness-preserving}
    Fix a subdivision datum $\subdatum$, $k$-horizontal $\blam\in\Par[\bkappa]_{\alpha}$ such that $T\in\RST{\blam}$, if $\bi^T_t=k$, then $\phi^T(t+1)=\phi^T(t)+2$. In particular, for any $t'>t$, we have $\phi^T(t')>\phi^T(t)+1$.\hfill $\square$
\end{Lemma}

\begin{Proposition}\label{prop:subdivision-of-standard-tableau-is-standard}
    Fix a subdivision datum $\subdatum$. Let $\blam\in\Par[\bkappa]_{\alpha}$ be $k$-horizontal, and let $T\in\RST{\blam}$. Set $\bmu:=\Phi_k(\blam)$. Then $T'$ is row-standard, that is, $T'\in\RST{\bmu}$. Moreover, if $T\in\Std(\blam)$, then $T'$ is standard, that is, $T'\in\Std(\bmu)$.
\end{Proposition}
\begin{proof}
    By \autoref{lm:image-bmu-tableau}, $T'$ is a $\bmu$-tableau. It remains to show that $T'$ is row-standard, and, if $T$ is column-standard, that $T'$ is column-standard as well. Both properties can be checked locally.
    
    Let $A\in[\blam]$ be a node such that $T(A)=t$, and suppose that the surrounding nodes (some of which may not exist) are as follows (we draw the nodes in $[\blam]$ on the left and the corresponding entries of $T$ on the right):
    \begin{center}
        \resizebox{0.4\linewidth}{!}{
        \RibbonTableau[no border, math entries, box height=0.9, box width=0.9]{12_{A^u},21_{A^l},22_{A},23_{A^r},32_{A^d}}\qquad
        \RibbonTableau[no border, math entries, box height=0.9, box width=0.9]{12_{t^u},21_{t^l},22_{t},23_{t^r},32_{t^d}}
        }
    \end{center}
    \smallskip
    where $T(A^{*})=t^{*}$ for $*\in\{l,r,u,d\}$. Apply the construction of $[\bmu]$ and $T'$, there are several cases:
    \begin{enumerate}
        \item If $\res(A)=k$, then the corresponding nodes in $[\bmu]$ and the entries of $T'$ are as follows:
        \begin{center}
            \resizebox{0.6\linewidth}{!}{
            \RibbonTableau[no border, math entries, box height=1.5, box width=1.45]{
              13_{{(A^u)}^+},21_{{(A^l)}^+},22_{A^+},23_{A^{\sharp}},24_{{(A^r)}^+},32_{{(A^d)}^+}
            }
            \qquad
            \RibbonTableau[no border, math entries, box height=1.5, box width=1.45]{
              13_{\phi^T(t^u)},
              21_{\phi^T(t^l)},
              22_{\phi^T(t)},
              23_{\phi^T(t){+}1},
              24_{\phi^T(t^r)},
              32_{\phi^T(t^d)}
            }
            }
        \end{center}
        \smallskip
        Since $T$ is row-standard, we have $t^l<t<t^r$. The inequality $\phi^T(t^l)<\phi^T(t)$ follows from the fact that $\phi^T$ is strictly increasing, and $\phi^T(t)+1<\phi^T(t^r)$ follows from \autoref{lm:row-standardness-preserving} together with the assumption $\res(A)=\bi^T_t=k$. If, moreover, $T$ is standard, then $t^u<t<t^d$. Both $\phi^T(t)<\phi^T(t^d)$ and $\phi^T(t^u)<\phi^T(t)+1$ follow from the fact that $\phi^T$ is strictly increasing.
    
        \item If $\res(A)=k+1$, then $\res(A^l)=\res(A^d)=k$, and the corresponding nodes in $[\bmu]$ and the entries of $T'$ are as follows:
        \begin{center}
            \resizebox{0.6\linewidth}{!}{
            \RibbonTableau[no border, math entries, box height=1.5, box width=1.5]{
              13_{{(A^u)}^+},21_{{(A^l)}^+},23_{A^+},22_{{(A^{l})}^{\sharp}},24_{{(A^r)}^+},32_{{(A^d)}^+},33_{{(A^{d})}^{\sharp}}}
            \qquad
            \RibbonTableau[no border, math entries, box height=1.5, box width=1.5]{
              13_{\phi^T(t^u)},
              21_{\phi^T(t^l)},
              23_{\phi^T(t)},
              22_{\phi^T(t^l){+}1},
              24_{\phi^T(t^r)},
              32_{\phi^T(t^d)},
              33_{\phi^T(t^d){+}1}
            }
            }
        \end{center}
        \smallskip
        Since $T$ is row-standard, we have $t^l<t<t^r$. The inequality $\phi^T(t)<\phi^T(t^r)$ follows from the fact that $\phi^T$ is strictly increasing, and $\phi^T(t^l)+1<\phi^T(t)$ follows from \autoref{lm:row-standardness-preserving} together with the condition $\res(A^l)=\bi^T_{t^l}=k$. 
        If, moreover, $T$ is standard, then $t^u<t<t^d$. Both $\phi^T(t^u)<\phi^T(t)$ and $\phi^T(t)<\phi^T(t^d)+1$ follow from the fact that $\phi^T$ is strictly increasing, while $\phi^T(t^l)+1<\phi^T(t^d)$ again follows from \autoref{lm:row-standardness-preserving}.
            
        \item If $\res(A)\neq k,k+1$, then the corresponding nodes in $[\bmu]$ and the entries of $T'$ are as follows:
        \begin{center}
            \resizebox{0.5\linewidth}{!}{
            \RibbonTableau[no border, math entries, box height=1.5, box width=1.5]{
              12_{{(A^u)}^+},21_{{(A^l)}^+},22_{A^+},23_{{(A^r)}^+},32_{{(A^d)}^+}}
            \qquad
            \RibbonTableau[no border, math entries, box height=1.5, box width=1.5]{
              12_{\phi^T(t^u)},
              21_{\phi^T(t^l)},
              22_{\phi^T(t)},
              23_{\phi^T(t^r)},
              32_{\phi^T(t^d)}
            }
            }
        \end{center}
        \smallskip
        In this case, all inequalities follow immediately from the fact that $\phi^T$ is strictly increasing.
    \end{enumerate}
    Since every node in $[\bmu]$ is either of the form $A^+$ or of the form $A^{\sharp}$, the cases above exhaust all possibilities. Hence $T'$ is a row-standard $\bmu$-tableau, and is standard if $T$ is.
\end{proof}

By \autoref{prop:subdivision-of-standard-tableau-is-standard}, we define the subdivision map on row-standard tableaux by
\[
    \Phi_k:\RST{\blam}\to\RST{\bmu},\qquad T\mapsto \Phi_k(T):=T',
\]
which restricts to the set of standard tableaux:
\[
    \Phi_k:\Std(\blam)\to\Std(\bmu),\qquad T\mapsto \Phi_k(T):=T'.
\]
By construction, $\Phi_k(T^{\blam})=T^{\bmu}$ always holds and hence $\Phi_k(\bi^{\blam})=\bi^{\bmu}$.
\begin{Example}\label{eg:subdivision-on-standard-tableau}
    Let the subdivision datum be $\subdatum=(3,I,\Lambda_0+\Lambda_1,\alpha,(0,1),1)$, and consider the multipartition $\blam=(8,5\mid 4,3,2)\in\Par[\bkappa]$. Then $\blam$ is $k$-horizontal, and $\bmu:=\Phi_k(\blam)=(11,6\mid 6,4,2)$. The Young diagrams filled with residues are as follows. We draw the $k$-nodes and the $(k+1)$-nodes in the maximal $(k,k+1)$-strips of $[\blam]$ in {\color{cyan}cyan} and {\color{OrangeRed}orange}, respectively. Similarly, we draw the corresponding nodes in $[\bmu]$ in the same colors:
    \begin{center}
        \tikzset{
            C/.style={fill=cyan,text=white},
            O/.style={fill=OrangeRed, text=white}
                }
        $[\blam]:$\qquad\Multitableau[no border]{0[C]1[O]20[C]1[O]20[C]1,20[C]1[O]20|[C]1[O]20[C]1,0[C]1[O]2,20}\\
        $[\bmu]:$\qquad\Multitableau[no border]{0[C]12[O]30[C]12[O]30[C]12,30[C]12[O]30|[C]12[O]30[C]12,0[C]12[O]3,30}
    \end{center}
    Take $T\in\Std(\blam)$ to be the following standard tableau:
    \begin{center}
        \tikzset{
            C/.style={fill=cyan,text=white},
            O/.style={fill=OrangeRed, text=white}
                }
        \Multitableau[no border]{1[C]2[O]68[C]{13}[O]{17}{21}[C]{22},35[C]{12}[O]{16}{18}|[C]4[O]9{10}[C]{20},7[C]{14}[O]{19},{11}{15}}
    \end{center}
    Applying $\Phi_k$, we obtain the following tableau $\Phi_k(T)$:
    \begin{center}
        \tikzset{
            C/.style={fill=cyan,text=white},
            O/.style={fill=OrangeRed, text=white}
                }
        \Multitableau[no border]{1[C]23[O]8{10}[C]{16}{17}[O]{22}{27}[C]{28}{29},47[C]{14}{15}[O]{21}{23}|[C]56[O]{11}{12}[C]{25}{26},9[C]{18}{19}[O]{24},{13}{20}}
    \end{center}
    It is easy to check (for example, using SageMath) that $\deg T=12=\deg \Phi_k(T)$.
\end{Example}

We want to prove that the subdivision map we have just constructed preserves the degree of a standard $\blam$-tableau. Before stating the theorem, we first show that removable nodes and addable nodes behave well under subdivision (on partitions).

\begin{Lemma}\label{lm:removable-correspondence}
    Fix a subdivision datum $\subdatum$. Let $\blam\in \Par[\bkappa]$ be $k$-horizontal and $\bmu:=\Phi_k(\blam)$. Let $A\in[\blam]$ be a node.    \begin{enumerate}
        \item If $\res(A)=i\neq k$, let $j:=\Phi_k(i)$, where $\Phi_k(i)$ is the subdivision on words defined in \autoref{subsec:subdivision-words}. Then $A$ is a removable $i$-node of $[\blam]$ if and only if $A^+$ is a removable $j$-node of $[\bmu]$.
        \item If $\res(A)=k$, then $A$ is a removable $k$-node of $[\blam]$ if and only if $A^{\sharp}$ is a removable $(k+1)$-node of $[\bmu]$ and $A^+$ is a removable $k$-node of $[\bmu]\setminus \{A^{\sharp}\}$.
    \end{enumerate}
\end{Lemma}
\begin{proof}
    Assume first that $\res(A)=i\neq k$, and set $j:=\Phi_k(i)$. Let $A=(m,r,c)$, and let $A^+=(m,r,c^{+})\in[\bmu]$ be the (unique) $j$-node corresponding to $A$.

    Recall that $A$ is removable in $[\blam]$ if and only if $c=\blam^{(m)}_r$ and either $r=\ell(\blam^{(m)})$ or $\blam^{(m)}_r>\blam^{(m)}_{r+1}$. Similarly, $A^+$ is removable in $[\bmu]$ if and only if $c^+=\bmu^{(m)}_r$ and either $r=\ell(\bmu^{(m)})$ or $\bmu^{(m)}_r>\bmu^{(m)}_{r+1}$.
    
    Let $p^m_t$ denote the number of $k$-nodes in the $t$-th row of $[\blam^{(m)}]$, for $1\le t\le \ell(\blam^{(m)})$. By construction of $\Phi_k$ and since $\blam$ is $k$-horizontal, we have
    \[
        \bmu^{(m)}_t=\blam^{(m)}_t+p^m_t\qquad\text{for all }t.
    \]
    
    Since $A$ is not a $k$-node and $\blam$ is $k$-horizontal, $A$ is the last node of $[\blam^{(m)}_r]$ if and only if $A^+$ is the last node of $[\bmu^{(m)}_r]$. In other words, $c^+=\bmu^{(m)}_r$ if and only if $c=\blam^{(m)}_r$.
    Thus removability of $A$ and $A^+$ reduces to comparing adjacent row lengths. In particular, if $r=\ell(\blam^{(m)})$, there is nothing to prove. So we assume $r<\ell(\blam^{(m)})$.

    \noindent\emph{($\Leftarrow$)} Suppose that $A^+$ is removable. Then $\bmu^{(m)}_r>\bmu^{(m)}_{r+1}$ and hence
    \[
        \blam^{(m)}_r-\blam^{(m)}_{r+1}
        =\big(\bmu^{(m)}_r-\bmu^{(m)}_{r+1}\big)-\big(p^m_r-p^m_{r+1}\big).
    \]
    Since $p^m_r-p^m_{r+1}\ge 0$, it follows that $\blam^{(m)}_r-\blam^{(m)}_{r+1}\ge 1-(p^m_r-p^m_{r+1})$. If $p^m_r-p^m_{r+1}=0$, then $\blam^{(m)}_r-\blam^{(m)}_{r+1}\ge 1$, so $\blam^{(m)}_r>\blam^{(m)}_{r+1}$ and $A$ is removable. If $p^m_r-p^m_{r+1}>0$, assume that $\blam^{(m)}_r=\blam^{(m)}_{r+1}$. Since the last node $A$ of row $r$ of $[\blam^{(m)}]$ is not a $k$-node, this implies that the number of $k$-nodes in rows $r$ and $r+1$ must be the same, which contradicts $p^m_r>p^m_{r+1}$. Hence $\blam^{(m)}_r>\blam^{(m)}_{r+1}$ and $A$ is removable in all cases.
    
    \noindent\emph{($\Rightarrow$)} If $A$ is removable then $\blam^{(m)}_r>\blam^{(m)}_{r+1}$, using $p^m_r\ge p^m_{r+1}$, we get
    \[
    \bmu^{(m)}_r-\bmu^{(m)}_{r+1}
    =\big(\blam^{(m)}_r-\blam^{(m)}_{r+1}\big)+\big(p^m_r-p^m_{r+1}\big)\ge 1,
    \]
    hence $A^+$ is removable.
    
    \medskip
    If $\res(A)=k$, then under subdivision the node $A=(m,r,c)$ corresponds to a horizontally adjacent pair in $[\bmu^{(m)}_r]$:
    an old $k$-node $A^+$ and a new $(k+1)$-node $A^{\sharp}$ immediately to the right of $A^+$.
    We want to prove that $A$ is removable in $[\blam]$ if and only if $A^{\sharp}$ is removable in $[\bmu]$ and $A^+$ is removable in $[\bmu]\setminus\{A^{\sharp}\}$.

    Since $A^{\sharp}$ is immediately to the right of $A^+$ in the same row, it is the last node in $[\bmu^{(m)}_r]$ if and only if $A$ is the last node of $[\blam^{(m)}_r]$. By the same argument as in the last case, it suffices to compare the adjacent row lengths in $[\blam^{(m)}]$ and in $[\bmu^{(m)}]$. We assume $r<\ell(\blam^{(m)})$.
    
    \noindent\emph{($\Rightarrow$)} If $A$ is removable, so $c=\blam^{(m)}_r$. Since $A$ is a $k$-node, the node immediately below $A$ has residue $k-1$, in other words, it is of the form:
    \begin{center}
        \resizebox{0.5\linewidth}{!}{
                    \tikzset{
                C/.style={fill=cyan,text=white},
                O/.style={fill=OrangeRed, text=white}
                    }
                    \Tableau[no border,box height=1,box width=1,dotted cols={1,4,5,8,9,10,11}]{0[C]{k}{\empty}00[C]{k}{\empty}0000[C]{k},0{\empty}[C]{k}00{\empty}[C]{k}0000{k{-}1}}
        }
    \end{center}
    as a result, whether or not this $(k-1)$-node lies in $[\blam]$, we always have $p^m_{r}>p^m_{r+1}$. Hence
    \[
        \bmu^{(m)}_{r}-\bmu^{(m)}_{r+1}
        =\blam^{(m)}_{r}-\blam^{(m)}_{r+1}+p^m_r-p^m_{r+1}
        >1+0=1.
    \]
    Hence $B$ is removable of $[\bmu]$ and $A^+$ is removable of $[\bmu]\setminus\{A^{\sharp}\}$. 
    
    \noindent\emph{($\Leftarrow$)} Conversely, if $[\bmu]\setminus\{A^{\sharp},A^+\}$ is still (the Young diagram of) an $\ell$-partition, then $\bmu^{(m)}_r\ge \bmu^{(m)}_{r+1}+2$. If $p^m_r>p^m_{r+1}+1$, then at least the node immediately below $A$ is not contained in $[\blam]$, and hence $A$ is removable. Otherwise,
    \[
        \blam^{(m)}_r-\blam^{(m)}_{r+1}
        =\big(\bmu^{(m)}_r-\bmu^{(m)}_{r+1}\big)-\big(p^m_r-p^m_{r+1}\big)
        \ge 2-1=1,
    \]
    and hence $A$ is also removable.
\end{proof}

We also need to prove the analogous statements for addable nodes of $[\blam]$ and $[\bmu]$. To do this, we extend the definition of subdivision on partitions.

Fix a subdivision datum $\subdatum$. For each node $A=(m,r,c)$, we define the residue $\res(A):=\res_{\bkappa}(A)$ as in \autoref{subsec:multipartition-charge-residue}; if $\res(A)=i$, we refer to $A$ as an $i$-node. Equivalently, we may choose an $\ell$-partition $\blam\in\Par[\bkappa]$ such that $A\in [\blam]$ and then define the residue of $A$ in the usual way. By definition of the residue function, it is independent of the choice of $\blam$ containing $A$.

Let $p_k(A)$ be the number of $k$-nodes of the form $(m,r,x)$ such that $1\leq x<c$. We define the node $A^{+}=\big(m,r,c+p_k(A)\big)$. If there is a $k$-horizontal $\ell$-partition $\blam$ containing $A$ and $A$ is an $i$-node, then under the subdivision map, $A^+\in [\Phi_k(\blam)]$ is the corresponding $\Phi_k(i)$-node, as before.




\begin{Lemma}\label{lm:addable-correspondence}
    Fix a subdivision datum $\subdatum$. Let $\blam\in \Par[\bkappa]$ be $k$-horizontal and set $\bmu:=\Phi_k(\blam)$. Let $A$ be a node, not necessarily in $[\blam]$.
    \begin{enumerate}
        \item If $\res(A)=i\neq k$, let $j:=\Phi_k(i)$. Then $A$ is an addable $i$-node of $[\blam]$ if and only if $A^+$ is an addable $j$-node of $[\bmu]$.
        \item If $\res(A)=k$, then $A$ is an addable $k$-node of $[\blam]$ if and only if $A^{+}$ is an addable $k$-node of $[\bmu]$ and $A^{\sharp}$ is an addable $(k+1)$-node of $[\bmu]\cup \{A^{+}\}$.
    \end{enumerate}
\end{Lemma}
\begin{proof}
    Let $A=(m,r,c)$ be a node with $\res(A)=i\neq k$. We want to prove that $A$ is an addable node of $[\blam]$ if and only if $A^+$ is an addable node of $[\bmu]$. 
    
    By definition, $A$ is addable of $[\blam]$ if and only if $\blam^{(m)}_r+1=c$ and $\blam^{(m)}_{r-1}\ge c$ whenever $r\neq 1$. Similarly, $A^+=(m,r,c^+)$ is addable of $[\bmu]$ if and only if $\bmu^{(m)}_r+1=c^+$ and $\bmu^{(m)}_{r-1}\ge c^+$ whenever $r\neq 1$. Let $p^m_t$ denote the number of $k$-nodes in the $t$-th row of $[\blam^{(m)}]$, for $1\le t\le \ell(\blam^{(m)})$.
    
    Suppose that $A$ is addable. Then $c^+=c+p^m_r=\blam^{(m)}_r+1+p^m_r=1+\bmu^{(m)}_r$, as required. Moreover, if $r\neq 1$, then
    $\bmu^{(m)}_{r-1}-c^+=\big(\blam^{(m)}_{r-1}+p^{m}_{r-1}\big)-\big(c+p^m_{r}\big)
    =\big(\blam^{(m)}_{r-1}-c\big)+\big(p^{m}_{r-1}-p^{m}_{r}\big)\ge 0$,
    so $\bmu^{(m)}_{r-1}\ge c^{+}$. Hence $A^+$ is addable in $[\bmu]$. 
    
    Conversely, suppose that $A^+$ is addable in $[\bmu]$. Then
    $c^{+}=\bmu^{(m)}_r+1=\blam^{(m)}_r+p^m_r+1$.
    Since the number of $k$-nodes before $c^{+}$ is the same as the number of $k$-nodes in $[\bmu^{(m)}_r]$, which is the same as the number of $k$-nodes in $[\blam^{(m)}_r]$, this number is $p^m_r$ by definition. Hence $c^{+}=c+p^m_r$ and $c=\blam^{(m)}_r+1$, as required. Now assume $r\neq 1$. If $p^m_{r-1}>p^m_r$, then $\blam^{(m)}_{r-1}>\blam^{(m)}_r=c-1$ since $\res(A)\neq k$, and hence $\blam^{(m)}_{r-1}\ge c$. Otherwise, $p^m_{r-1}=p^m_r$ and
    $\blam^{(m)}_{r-1}-c=\big(\bmu^{(m)}_{r-1}-p^m_{r-1}\big)-c
    =\big(\bmu^{(m)}_{r-1}-c^{+}\big)-\big(p^m_{r-1}-p^m_r\big)\ge 0$,
    as desired. Therefore $A$ is addable in $[\blam]$, and the claim follows.
    
    \smallskip
    Let $A=(m,r,c)$ be a $k$-node. Under the subdivision map, there is a unique $k$-node $A^+=(m,r,c^{+})$ and a unique $(k+1)$-node $A^{\sharp}=(m,r,c^{+}+1)$ corresponding to $A$. Namely, it is of the following form (we draw the node labels on the left and their residues on the right; the same convention applies to all figures below in the proof of \autoref{thm:degree-invariance-standard-tableau}):
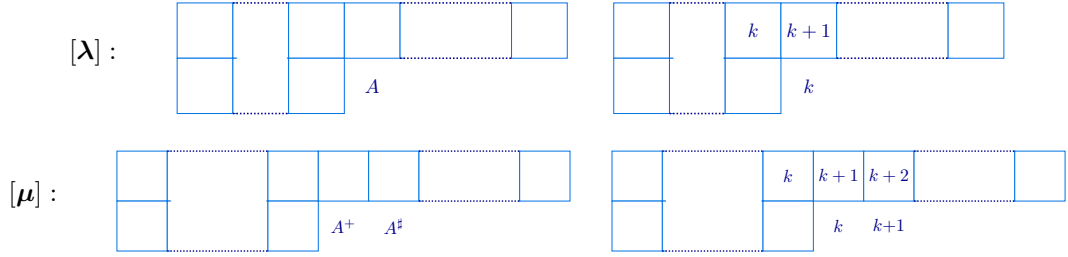
    \begin{center}
        $[\blam]:$\qquad\resizebox{0.7\linewidth}{!}{
        \Tableau[no border,box height=1,box width=1,snob style={draw =none},snobs={24_{A}},dotted cols={2,5,6}]{{\empty}0{\empty}{\empty}{\empty}0{\empty},{\empty}0{\empty}}
        \qquad
        \Tableau[no border,box height=1,box width=1,snob style={draw =none},snobs={24_{k}},dotted cols={2,5,6}]{{\empty}0{k}{k+1}00{\empty},{\empty}0{\empty}}
        }
    \end{center}
    \begin{center}
        $[\bmu]:$\qquad\resizebox{0.8\linewidth}{!}{
        \Tableau[no border,box height=1,box width=1,snob style={draw =none},snobs={25_{A^+},26_{A^{\sharp}}},dotted cols={2,3,7,8}]{{\empty}00{\empty}{\empty}{\empty}00{\empty},{\empty}00{\empty}}
        \qquad
        \Tableau[no border,box height=1,box width=1,snob style={draw =none},snobs={25_{k},26_{k{+}1}},dotted cols={2,3,7,8}]{{\empty}00{k}{k+1}{k+2}00{\empty},{\empty}00{\empty}}
        }
    \end{center}

    We want to prove that $A$ is an addable node of $[\blam]$ if and only $A^+$ is an addable node of $[\bmu]$ and $A^{\sharp}$ is an addable node of $[\bmu]\cup\{A^+\}$.
    
    Suppose that $A$ is an addable $k$-node of $[\blam]$. Then $c=\blam^{(m)}_r+1$ and $\blam^{(m)}_{r-1}\ge c$ whenever $r\neq 1$. By construction, $c^{+}=p^m_r+c=p^m_r+\blam^{(m)}_r+1=\bmu^{(m)}_r+1$. It remains to show that $\bmu^{(m)}_{r-1}\ge c^{+}+1$. Since $\res(A)=k$ and $A\notin [\blam]$, the maximal $(k,k+1)$-strip containing $A$ has no intersection with $[\blam^{(m)}_{r}]$ but does intersect $[\blam^{(m)}_{r-1}]$. In particular, $p^m_{r-1}>p^m_{r}$. Hence
    \[
        \bmu^{(m)}_{r-1}-c^{+}=\blam^{(m)}_{r-1}+p^m_{r-1}-c-p^m_r\ge 1.
    \]
    
    Conversely, suppose that $[\bmu]\cup \{A^+,A^{\sharp}\}$ is still the Young diagram of an $\ell$-partition. Then $c^{+}=\bmu^{(m)}_{r}+1=\blam^{(m)}_r+p^m_r+1$ and $\bmu^{(m)}_{r-1}\ge c^{+}+1$. Since the number of $k$-nodes before $A^+$ in $[\bmu^{(m)}_r]$ is the same as the number of $k$-nodes before $A$ in $[\blam^{(m)}_r]$, which is $p^m_r$, we have $c^{+}=p^m_r+c$ and hence $c=\blam^{(m)}_r+1$. It remains to show that $\blam^{(m)}_{r-1}\ge c$. If $p^m_{r-1}>p^m_r+1$, then $\blam^{(m)}_{r-1}>\blam^{(m)}_{r}=c-1$. Otherwise $p^m_{r-1}=p^m_r+1$, and
    $\blam^{(m)}_{r-1}-c=\bmu^{(m)}_{r-1}-p^m_{r-1}-c^{+}+p^m_r=\bmu^{(m)}_{r-1}-c^{+}-1\ge 0$,
    as required.
\end{proof}

\smallskip
For an $\ell$-partition $\blam=(\blam^{(1)},\cdots,\blam^{(\ell)})$, we say that the component $\blam^{(i)}$ is earlier than $\blam^{(j)}$ if $i<j$, and we call $\blam^{(i)}$ an \emph{earlier component} of $\blam$. Equivalently, we say that $\blam^{(j)}$ is later than $\blam^{(i)}$, and we call $\blam^{(j)}$ a \emph{later component} of $\blam$. 





For $\blam\in\Par[\bkappa]_n$ and $T\in \Std(\blam)$, the \emph{last node} of $T$ is the node $A\in[\blam]$ such that $T(A)=n$. In particular, since $T$ is standard, $A$ is a removable node of $[\blam]$. Hence we can define $d_{A}(\blam)$ as in \autoref{eq:degree-of-last-node}.

\begin{Lemma}\label{lm:induction-k-node}
    Fix a subdivision datum $\subdatum$. Take $\blam\in\Par[\bkappa]_{\alpha}$ to be $k$-horizontal and set $\bmu:=\Phi_k(\blam)$. Let $A$ be the last node of $T$. Suppose that $\res(A)=k$. Then
    \[
        d_{A}(\blam)=d_{A^{\sharp}}(\bmu)+d_{A^{+}}([\bmu]\setminus\{A^{\sharp}\}).
    \]
\end{Lemma}
\begin{proof}
    By the definition of $d_{\bullet}(\bullet)$ in \autoref{subsec:degree-of-tableaux}, it is enough to prove the following two statements:
        (1) the number of addable $k$-nodes below $A$ of $[\blam]$ is equal to the sum of the number of addable $(k+1)$-nodes below $A^{\sharp}$ of $[\bmu]$ and the number of addable $k$-nodes below $A^+$ of $[\bmu]\setminus \{A^{\sharp}\}$;
        (2) the number of removable $k$-nodes below $A$ of $[\blam]$ is equal to the sum of the number of removable $(k+1)$-nodes below $A^{\sharp}$ of $[\bmu]$ and the number of removable $k$-nodes below $A^+$ of $[\bmu]\setminus \{A^{\sharp}\}$.

        The subdivision map on $\ell$-partitions is defined componentwise and hence preserves the component index of every node. Since $\blam$ is $k$-horizontal, the row index of each node is also preserved. As a result, the order (above/below) of nodes defined in \autoref{subsec:degree-of-tableaux} is preserved. Let $\bar{I}$ be the vertex set of the new quiver after subdivision. 

        \medskip
        We firstly treat the addable node case. Consider any addable $(k+1)$-node $M$ of $[\bmu]$ below $A^{\sharp}$, there are three possibilities: 
        \begin{enumerate}[label=(\roman*)]
            \item In some row below $A^{\sharp}$, the last node $M'$ of that row in $[\bmu]$ is a $k$-node, $M$ is immediately to the right of $M'$, and $[\bmu]\cup\{M\}$ is (the Young diagram of) an $\ell$-partition. In other words, it is of the following form:      
            \begin{center}
                \resizebox{0.6\linewidth}{!}{
                \Tableau[no border,box height=1,box width=1,snob style={draw =none},snobs={15_{M}},dotted cols={2,3}]{{\empty}{\cdots}{\cdots}{M'}}
                \qquad
                \Tableau[no border,box height=1,box width=1,snob style={draw =none},snobs={15_{k{+}1}},dotted cols={2,3}]{{\empty}{\cdots}{\cdots}{k}}
                }
            \end{center}
            
            \item In some component $[\bmu^{(i)}]$ that is no earlier than the component containing $A^{\sharp}$, the first node $M'$ of the last row has residue $k+2$, and the node $M$ lies immediately below $M'$. In other words, it is of the following form:
            \begin{center}
                \resizebox{0.5\linewidth}{!}{
                \Tableau[no border,box height=1,box width=1,snob style={draw =none},snobs={21_{M}},dotted cols={2,3}]{{M'}{\cdots}{\cdots}{\empty}}
                \qquad
                \Tableau[no border,box height=1,box width=1,snob style={draw =none},snobs={21_{k{+}1}},dotted cols={2,3}]{{k{+}2}{\cdots}{\cdots}{\empty}}
                }
            \end{center}
            
            \item $M$ is the first node of a component later than the component containing $A^{\sharp}$, and the charge of this component is $k+1\in \bar{I}$.
        \end{enumerate}
        
        The first case is impossible, since after subdivision any $k$-node appears with a $(k+1)$-node immediately to its right, and vice versa. The third case is also impossible because, under subdivision, no component can have charge $k+1$; see \autoref{subsec:subdivision-weight}. For each such node $M$ in the second case, it corresponds uniquely to an addable $k$-node $A_M$ in $[\blam]$ below $A$, of the following form:
        \begin{center}
                \resizebox{0.5\linewidth}{!}{
                \Tableau[no border,box height=1,box width=1,snob style={draw =none},snobs={21_{A_M}},dotted cols={2,3}]{{A_M'}{\cdots}{\cdots}{\empty}}
                \qquad
                \Tableau[no border,box height=1,box width=1,snob style={draw =none},snobs={21_{k}},dotted cols={2,3}]{{k{+}1}{\cdots}{\cdots}{\empty}}
                }
            \end{center}
            
        Any other addable $k$-node $A_0$ of $[\blam]$ below $A$ is of one of the following forms: 
        \begin{enumerate}[label=(\roman*)]
            \item $A_0$ is the first node of a later component than the component containing $A$, and the charge of this component is $k$.

            \item In some row below $A$, the last node $A_0'$ of that row in $[\blam]$ is a $(k-1)$-node, $A_0$ is immediately to the right of $A_0'$, and $[\blam]\cup\{A_0\}$ is an $\ell$-partition. In other words, it is of the following form:
            \begin{center}
                \resizebox{0.6\linewidth}{!}{
                \Tableau[no border,box height=1,box width=1,snob style={draw =none},snobs={15_{A_0}},dotted cols={2,3}]{{\empty}{\cdots}{\cdots}{A_0'}}
                \qquad
                \Tableau[no border,box height=1,box width=1,snob style={draw =none},snobs={15_{k}},dotted cols={2,3}]{{\empty}{\cdots}{\cdots}{k{-}1}}
                }
            \end{center}            
        \end{enumerate}
        
        On the other hand, any addable $k$-node $M_0$ below $A^+$ of $[\bmu]\setminus \{A^{\sharp}\}$ is of one of the following forms:
        \begin{enumerate}[label=(\roman*)]
            \item $M_0$ is the first node of a later component than the component containing $A^+$ and the charge of this component is $k$.
            
            \item In some row below $A^+$, the last node $M_0'$ of that row in $[\bmu]\setminus \{A^{\sharp}\}$ is a $(k-1)$-node, $M_0$ is immediately to the right of $M_0'$, and $[\bmu]\setminus\{A^{\sharp}\}\cup\{M_0\}$ is an $\ell$-partition. In other words, it is of the following form:
            \begin{center}
                \resizebox{0.6\linewidth}{!}{
                \Tableau[no border,box height=1,box width=1,snob style={draw =none},snobs={15_{M_0}},dotted cols={2,3}]{{\empty}{\cdots}{\cdots}{M_0'}}
                \qquad
                \Tableau[no border,box height=1,box width=1,snob style={draw =none},snobs={15_{k}},dotted cols={2,3}]{{\empty}{\cdots}{\cdots}{k{-}1}}
                }
            \end{center}    
        \end{enumerate}
        
        Under the subdivision map $\Phi_k$, and by the assumption that $\blam$ is $k$-horizontal, the correspondence $M_0\leftrightarrow A_0$ is one-to-one in each case.

        By \autoref{lm:addable-correspondence}, the statement for the number of addable nodes follows.
        
        \medskip
        We next consider the removable-node case.
        Every removable $(k+1)$-node $N$ of $[\bmu]$ below $A^{\sharp}$ is a new node $B^{\sharp}$ immediately to the right of a $k$-node $B^{+}$ for some $k$-node $B$ in $[\blam]$, and $B^{\sharp}$ is the last node of some row below $A^{\sharp}$. In other words, it is of the following form:
        \begin{center}
                \resizebox{0.6\linewidth}{!}{
                \Tableau[no border,box height=1,box width=1,dotted cols={2,3}]{{\empty}{\cdots}{\cdots}{B^+}{B^{\sharp}}}
                \qquad
                \Tableau[no border,box height=1,box width=1,dotted cols={2,3}]{{\empty}{\cdots}{\cdots}{k}{k{+}1}}
                }
            \end{center} 
        Every removable $k$-node $B$ of $[\blam]$ below $A$ is the last node of some row below $A$, of the form:
        \begin{center}
            \resizebox{0.5\linewidth}{!}{
                \Tableau[no border,box height=1,box width=1,dotted cols={2,3}]{{\empty}{\cdots}{\cdots}{B}}
                \qquad
                \Tableau[no border,box height=1,box width=1,dotted cols={2,3}]{{\empty}{\cdots}{\cdots}{k}}
                }
        \end{center}
        Under the subdivision map, as discussed at the beginning of this section, $B\leftrightarrow B^{\sharp}$ is a one-to-one correspondence. On the other hand, there is no removable $k$-node in $[\bmu]\setminus \{A^{\sharp}\}$ below $A^+$, since every $k$-node in $[\bmu]$ appears to the left of a $(k+1)$-node and hence cannot be removable.
        Therefore, by \autoref{lm:removable-correspondence}, the desired equality for the number of removable nodes also holds.

        \medskip
        Combining the addable-node case and the removable-node case then yields the desired equality
        $d_A(\blam)=d_{A^{\sharp}}(\bmu)+d_{A^+}([\bmu]\setminus \{A^{\sharp}\})$.
\end{proof}

The proof of the next result is analogous to that of \autoref{lm:induction-k-node}.

\begin{Lemma}\label{lm:induction-non-k-node}
    Fix a subdivision datum $\subdatum$. Take $\blam\in\Par[\bkappa]_{\alpha}$ to be $k$-horizontal and set $\bmu:=\Phi_k(\blam)$. Let $A$ be the last node of $T$. Suppose that $\res(A)\neq k$. Then
    \[
        d_{A}(\blam)=d_{A^{+}}(\bmu).
    \]
\end{Lemma}
\begin{proof}
    Let $j=\Phi_k(i)$, where $\Phi_k$ is the subdivision map on words introduced in \autoref{subsec:subdivision-words}. Then $A^+$ is a $j$-node in $[\bmu]$. It is enough to prove that the number of addable $i$-nodes of $[\blam]$ below $A$ is equal to the number of addable $j$-nodes of $[\bmu]$ below $A^+$, and that the number of removable $i$-nodes of $[\blam]$ below $A$ is equal to the number of removable $j$-nodes of $[\bmu]$ below $A^+$. Let $\bar{I}$ be the vertex set of the new quiver after subdivision.  

    \smallskip
    Any addable $i$-nodes $M$ of $[\blam]$ below $A$ is of the following form:
    \begin{enumerate}[label=(\roman*)]
        \item In some row below $A$,
        the last node $M'$ of that row in $[\blam]$ is a $k$-node, $M$ is immediately to the right of $M'$, and $[\blam]\cup\{M\}$ is an $\ell$-partition.
                    
        \item In some component $\blam^{(m)}$ that is no earlier than the component containing $A$, the first node $M'$ of the last row has residue $i+1$, and the node $M$ lies immediately below $M'$.
        
        \item $M$ is the first node of later component than the component containing $A$ and the charge of this component is $i\in I$.
    \end{enumerate}
    
    While any addable $j$-node $A_M$ is of the following form:
    \begin{enumerate}[label=(\roman*)]
        \item In some row below $A^+$, there exists a $k$-node $A_M'\in [\bmu]$ such that $A_M'$ is the last node in that row, $A_M$ is immediately to the right of $A_M'$, and $[\bmu]\cup\{A_M\}$ is an $\ell$-partition.
                    
        \item In some component $\bmu^{(m)}$ that is no earlier than the component containing $A^+$, the first node $A_M'$ of the last row has residue $j+1$, and the node $A_M$ lies immediately below $A_M'$.
        
        \item $A_M$ is the first node of later component than the component where $A^+$ lives and the charge of this component is $j\in I$.
    \end{enumerate}
    
    Under the subdivision $\Phi_k$, $M\leftrightarrow A_M$ is thus a natural one-to-one correspondence and $A_M=M^+$ in our generalized definition of subdivision above \autoref{lm:removable-correspondence}. Hence by \autoref{lm:addable-correspondence}, the desired equality of addable nodes holds.

    \smallskip
    Similarly, applying \autoref{lm:removable-correspondence}, one can show that the removable $i$-nodes of $[\blam]$ below $A$ are in one-to-one correspondence with the removable $j$-nodes of $[\bmu]$ below $A^+$. 
    
    Combining these two statements gives the desired equality $d_A(\blam)=d_{A^+}(\bmu)$.
\end{proof}

\begin{Theorem}\label{thm:degree-invariance-standard-tableau}
    Fix a subdivision datum $\subdatum$,  $\blam\in\Par[\bkappa]_{\alpha}$ and $T\in\Std(\blam)$. Suppose $\blam$ is $k$-horizontal, then $\deg T=\deg \Phi_k(T)$.
\end{Theorem}
\begin{proof}
    Let $\bmu:=\Phi_k(\blam)$. Set $\height(\alpha)=d$ and $\height\big(\Phi_k(\alpha)\big)=d'$. Let $A=T^{-1}(d)$ be the last node of $T$. We proceed by induction on $d$. If $d=0$, then $\blam=\emptyset=\bmu$ and the claim holds by definition.

    For $d>0$, by definition of the degree of standard tableaux, we have 
    \[
        \deg T=d_{A}(\blam)+\deg \big( T\downarrow (d-1)\big)
    \]
    By inductive hypothesis, we have:
    \[
        \deg \big( T\downarrow (d-1)\big)=\deg \Big( \Phi_k \big(T\downarrow (d-1)\big)\Big)
    \]
    There are two cases, depending on whether $\res(A)=k$. We compute the degree in those two cases. 

    \begin{enumerate}
        \item If $\res (A)=k$, then applying the subdivision map to $[\blam]$ produces, in $[\bmu]$, an old node $A^+$ together with a new $(k+1)$-node $A^{\sharp}$ immediately to the right of $A^+$. In this case, by our definition of $\Phi_k(T)$, we have $\big(\Phi_k(T)\big)(A^+)=d'-1$ and $\big(\Phi_k(T)\big)(A^{\sharp})=d'$. Removing these two nodes from $[\bmu]$, it follows that
        $[\bmu]\setminus \{A^{\sharp},A^+\}=\Shape \Big(\Phi_k \big(T\downarrow (d-1) \big)\Big)$
        and $\Phi_k(T)\downarrow (d'-2)=\Phi_k \big(T\downarrow (d-1) \big)$. Hence:
        \begin{align*}
            \deg \big(\Phi_k( T)\big) &= d_{A^{\sharp}}(\bmu)+d_{A^+}([\bmu]\setminus \{A^{\sharp}\})+\deg \Big( \Phi_k(T)\downarrow (d'-2) \Big)\\
            &= d_{A^{\sharp}}(\bmu)+d_{A^+}([\bmu]\setminus \{A^{\sharp}\})+\deg \Big( \Phi_k \big(T\downarrow (d-1)\big)\Big)
        \end{align*}        
        It suffices to show that $d_A(\blam)=d_{A^{\sharp}}(\bmu)+d_{A^+}([\bmu]\setminus \{A^{\sharp}\})$, which follows from \autoref{lm:induction-k-node}.
        
        \smallskip
        \item If $\res (A)=i\neq k$, then applying the subdivision map gives a unique old node $A^+\in [\bmu]$ such that $\big(\Phi_k(T)\big)(A^+)=d'$, and we have
        $[\bmu]\setminus \{A^+\}=\Shape \Big(\Phi_k \big(T\downarrow (d-1) \big)\Big)$
        and $\Phi_k(T)\downarrow (d'-1)=\Phi_k \big(T\downarrow (d-1) \big)$. Hence:
        \begin{align*}
            \deg \big(\Phi_k(T)\big)&=d_{A^+}(\bmu)+\deg \Big( \Phi_k(T)\downarrow (d'-1) \Big)\\
            &=d_{A^+}(\bmu)
            +\deg \Big(\Phi_k\big( T\downarrow (d-1)\big)\Big)
        \end{align*}
        It suffices to show that $d_A(\blam)=d_{A^+}(\bmu)$, which follows from \autoref{lm:induction-non-k-node}. 
    \end{enumerate}
    
    \medskip
    Hence, $\deg T=\deg \Phi_k(T)$ follows by induction.   
\end{proof}


    
\section{Categorical Subdivision}\label{sec:categorical-subdivisions}
In this section, we use the combinatorial framework of the subdivision maps on partitions and standard tableaux to prove some categorical results concerning the subdivision map on KLR algebras. The main results are \autoref{thm:image-of-idempotent}, \autoref{thm:image-of-permutation-module}, and \autoref{thm:image-of-specht-module}.
\subsection{Image of Idempotents}\label{subsec:image-of-idempotent}
Fix a subdivision datum $\subdatum$. Let $R_{\alpha}$ be the corresponding KLR algebra. For any $\blam\in\Par[\bkappa]_{\alpha}$, the idempotent element $e_{\blam}\in R_{\alpha}$ is defined in \autoref{eq:idempotent-of-partition}. The aim of this section is to describe $\Phi_k(e_{\blam})$. For simplicity, in this section we write $\widebar{I}$, $\overline{\alpha}$, $\overline{\Lambda}$, and $\overline{\blam}$ for the images of $I$, $\alpha$, $\Lambda$, and $\blam$ under the corresponding subdivision map $\Phi_k$.
Let $\be\in R_{\oalpha}$ be the truncation idempotent defined in \autoref{eq:truncation-idempotent}.
\begin{Lemma}\label{lm:vanish-of-idempotent}
    Fix a subdivision datum $\subdatum[\Lambda_x][x]$ and take $\lambda\in\Par[\Lambda_x]_{\alpha}$. Then $e_{\olambda} \be=0$ if and only $\lambda$ is not $k$-horizontal.
\end{Lemma}
\begin{proof}
    By the definition of $e_{\lambda}$, the residue sequence $\bi^\lambda$ corresponds to the standard tableau $T^\lambda$. Suppose $\lambda$ is $k$-horizontal, i.e. there is no non-trivial maximal $(k+1,k)$-strip. Then every $k$-node in $[\lambda]$ lies in a maximal $(k,k+1)$-strip. By \autoref{def:subdivision-young-diagram}, each $k$-node is replaced by a horizontal pair \Tableau[no border,box font=\tiny]{{k}{k{+}1}} in $[\olambda]$. The corresponding residue sequence is locally of the form $\cdots,k,k+1,\cdots$, and thus $\bi^{\olambda}\in\wellorder$. In particular, $e(\bi^{\olambda})\be=e(\bi^{\olambda})\neq 0$.

    Conversely, suppose that $\lambda$ is not $k$-horizontal and that there exists a non-trivial maximal $(k+1,k)$-strip. Let $S$ be the top such strip, and let $A_1$ be its initial node of residue $k+1$ (so $A_1$ lies above the initial nodes of all other non-trivial maximal $(k+1,k)$-strips). Let $A_2$ be the $k$-node immediately below $A_1$. By \autoref{def:subdivision-abacus}, this vertical pair \Tableau[no border,box font=\tiny]{{A_1},{A_2}} is replaced by the triple \Tableau[no border,box font=\tiny]{{k{+}2},{k{+}1},{k}}. In $[\olambda]$, the $k+1$-node inside this triple is precisely the initial node of the top non-trivial maximal $(k+1,k)$-strip. Therefore, every $k$-node above this $(k+1)$-node lies in a maximal $(k,k+1)$-strip and, in particular, occurs in a horizontal adjacent pair \Tableau[no border,box font=\tiny]{{k}{k{+}1}}. Now assume that this $(k+1)$-node is the $m$-th node of $[\olambda]$ in row-reading order, and write $\bi^{\olambda}=(\bi_1',\bi_2',\dots)$. Then the initial segment $(\bi_1',\bi_2',\dots,\bi_m')$ contains one more $k+1$ than $k$, so $\bi^{\olambda}\in\unorder$ by definition and hence $e_{\olambda}\be=0$.
\end{proof}
\begin{Corollary}
   Fix a subdivision datum $\subdatum[\Lambda_x][x]$. If $x\neq k$, then all partitions $\lambda\in\Par[\Lambda_x]_{\alpha}$ such that $\be e_{\olambda}\neq 0$ are $e$-regular.
\end{Corollary}
\begin{proof}
    Let $\absubdatum$ be an abacus subdivision datum for $\lambda$.
    Since $a+d=ce+k$ and $a\equiv x\pmod{e}$, we have $x+d\equiv k\pmod e$. Since $x,k,d\in I$, if $x\neq k$, then $d\not\equiv 0\pmod e$, and hence $e-d<e$. By \autoref{cor:k-lambda-zero-bound-2}, $\lambda$ is $k$-horizontal (i.e.\ $k(\lambda)=0$) if and only if $\ell(\lambda)\le e-d$, and hence $\ell(\lambda)<e$. By \autoref{lm:vanish-of-idempotent}, the conclusion follows.
\end{proof}

Recall from \autoref{thm:subdivision-iso} that $\Phi_k$ is an isomorphism between $R_\alpha$ and the balanced KLR algebra $S_{\oalpha}$. Since $e_{\blam}\in R_\alpha$ is a non-zero idempotent, it must be mapped to a non-zero idempotent in $S_{\oalpha}$. By \autoref{lm:vanish-of-idempotent}, in the level one case, if $\lambda$ is not $k$-horizontal, then
$0=\be e_{\olambda}\be+\be\badideal \be\in S_{\oalpha}$ and thus $\Phi_k(e_{\lambda})\neq e_{\olambda}$. By contrast, in the $k$-horizontal case, we have the following result.
\begin{Theorem}\label{thm:image-of-idempotent}
    Fix a subdivision datum $\subdatum$, and let $\blam\in\Par[\bkappa]_{\alpha}$ be a $k$-horizontal multipartition. Then
    \[
        \Phi_k(e_{\blam})=e_{\Phi_k(\blam)}+\be\badideal \be.
    \]
    
\end{Theorem}
\begin{proof}
    The definition of $e_{\blam}$ and the subdivision maps $\Phi_k$ on partitions and KLR algebras are all defined componentwise. Therefore, it suffices to prove the level-one case.
    By \autoref{def:subdivision-young-diagram}, we only need to consider the local pattern of $k$-nodes as all other nodes are just relabeled in a compatible way. 
    As $\lambda$ is $k$-horizontal, $k$-nodes only appears in maximal $(k,k+1)$-strips in $[\lambda]$. Consider the idempotent string diagram $e_\lambda$, where locally we have the pattern:
\begin{center}
    $\cdots$
    \begin{tikzpicture}[scale=1.2, thick]
        \draw (-3,0) -- (-3,2);
        \node[below] at (-3,0) {$k{-}1$};
        
        \draw (-1,0) -- (-1,2);
        \node[below] at (-1,0) {$k$};
        
        \draw (1,0) -- (1,2);
        \node[below] at (1,0) {$k{+}1$};
    \end{tikzpicture}
    $\cdots$
\end{center}

Under the subdivision, it is mapped to:

\begin{center}
    $\cdots$
    \begin{tikzpicture}[scale=1.2, thick]
        \draw (-3,0) -- (-3,2);
        \node[below] at (-3,0) {$k{-}1$};
        
        \draw (-1,0) -- (-1,2);
        \node[below] at (-1,0) {$k$};
        
        \draw (0,0) -- (0,2);
        \node[below] at (0,0) {$k{+}1$};
        
        \draw (1,0) -- (1,2);
        \node[below] at (1,0) {$k{+}2$};
    \end{tikzpicture}
    $\cdots$
\end{center}
    Translating this back to the partition, by \autoref{def:subdivision-young-diagram}, this corresponds to: the $(k,k+1)$-strip becoming the $(k,k+1,k+2)$-strip. Since the other strings only change their labels accordingly, the idempotent string diagram corresponds to $\Phi_k(\lambda)$. Hence, we obtain $\Phi_k(e_\lambda)=e_{\Phi_k(\lambda)}+\be\badideal \be$.
\end{proof}
\subsection{Subdivision of Permutation Modules}\label{subsec:subdivision-permutation-modules}
In this section, we present our first categorical result on subdivision. We first introduce a new map on the set of partitions (and extend it naturally to multipartitions), inspired by the level-up phenomenon in \cite{qin-specht-filtration-permutation}: the Specht modules appearing in the generalized Specht filtration of a permutation module $M^{\lambda}$, where $\lambda$ is a partition, are typically indexed by multipartitions.

Throughout this subsection, we fix a subdivision datum $\subdatum[\Lambda_x][x]$, unless stated otherwise.

Take a partition $\lambda=(\lambda_1,\dots,\lambda_{\ell(\lambda)})\in\Par[\Lambda_x]_{\alpha}$. As in \autoref{subsubsec:klam-nonzero-case}, let $s=k(\lambda)$ be the number of non-trivial maximal $(k+1,k)$-strips in $[\lambda]$, and list them as $S_1,\dots,S_s$ so that the row containing the initial node of $S_i$, written $k_i:=k_{S_i}$, satisfies
\[
k_1<k_2<\cdots<k_s.
\]
For convenience set $k_0:=0$ and $k_{s+1}:=\ell(\lambda)$. Note that $(k_i,1)\in[\lambda]$ is a $(k+1)$-node for each $1\le i\le s$, and $(k_i+1,1)\in[\lambda]$ is a $k$-node (in the first column) for each $1\le i\le s$.

Set $\bkappa=(x,\underbrace{k,\cdots,k}_{s \text{ times}})\in I^{s+1}$. Define the \emph{splitting map}
\[
\Psi_k:\Par[\Lambda_x]_\alpha\longrightarrow \Par[\bkappa]_\alpha
\]
by
\[
\Psi_k(\lambda)=(\lambda^{0}\mid \cdots\mid\lambda^{s}),
\qquad
\lambda^{i}:=(\lambda_{k_i+1},\lambda_{k_i+2},\dots,\lambda_{k_{i+1}})
\quad(0\le i\le s).
\]
Equivalently, $\Psi_k(\lambda)$ is obtained by cutting the Young diagram $[\lambda]$ into consecutive blocks of rows, with cuts immediately below rows $k_1,\dots,k_s$, and regarding each block as a separate component.

The map $\Psi_k$ is well-defined because, with the fixed choice $\bkappa$, the natural identification of nodes between $[\lambda]$ and $[\Psi_k(\lambda)]$ preserves residues. 
Indeed, $\Psi_k$ only cuts $[\lambda]$ into consecutive blocks of rows and reinterprets each block as a separate component, without changing the column positions of any nodes. 
The $0$th component inherits charge $x$, which matches the residue of the first node of $[\lambda]$. 
For each $1\le i\le s$, the first node of the $i$th component corresponds to the first node of row $k_i+1$ of $[\lambda]$, and by construction this is a $k$-node; assigning charge $k$ to that component therefore ensures that residues in this component agree with the residues of the corresponding nodes in $[\lambda]$. 
Consequently every node keeps the same residue under the identification, so the residue content is preserved and $\Psi_k(\lambda)\in\Par[\bkappa]_\alpha$.

Take a charge $\varkappa=(\varkappa_1,\cdots,\varkappa_{\ell})\in I^{\ell}$ and consider $\blam\in\Par[\varkappa]_{\alpha}$. We define $\Psi_k(\blam)$ componentwise: if $\blam=(\blam^{(1)},\cdots,\blam^{(\ell)})$, then we set
\[
    \Psi_k(\blam)=\Big(\Psi_k\big(\blam^{(1)}\big),\cdots,\Psi_k\big(\blam^{(\ell)}\big)\Big)
\]
If the number of non-trivial maximal $(k+1,k)$-strips in $\blam^{(i)}$ is $s_i$ for $1\le i\le \ell$, then it is straightforward to check that $\Psi_k(\blam)\in\Par[\pmb{\varkappa}]_{\alpha}$, where
\[
    \pmb{\varkappa}=(\varkappa_1,\underbrace{k,\cdots,k}_{s_1 \text{ times}},\varkappa_2,\underbrace{k,\cdots,k}_{s_2 \text{ times}},\cdots,\varkappa_{\ell},\underbrace{k,\cdots,k}_{s_{\ell} \text{ times}})\in I^{s+\ell},\quad s=\sum\limits_{1\leq i\leq \ell}s_i.
\]
\begin{Example}\label{eg:split-image}
    Again consider \autoref{eg:k-strip}, the image $\Psi_k(\lambda)$ has the following Young diagram:
    \begin{center}
        \tikzset{
            C/.style={fill=cyan,text=white},
            O/.style={fill=OrangeRed, text=white}
                }
        \Multitableau[no border,box height=0.5,box width=0.5]{[C]1[C]2340[C]1[C]2340[C]1,0[C]1[C]2340[C]1[C]2340,40[C]1[C]2340[C]1[C]234,340[C]1[C]2340[C]1[C]23,[O]2340[C]1[C]2340[C]1[C]2  |[C]1[C]2340[C]1[C]2340[C]1,0[C]1[C]2340[C]1[C]2340}
    \end{center}
\end{Example}
There is a special case in which the splitting map does nothing.
\begin{Lemma}\label{lm:split-no-change}
    Take a $k$-horizontal partition $\lambda\in\Par[\Lambda_x]_{\alpha}$. Then $\Psi_k(\lambda)=\lambda$.
\end{Lemma}
\begin{proof}
    This follows directly from the definition.
\end{proof}
As we can observe in \autoref{eg:split-image}, the advantage of the splitting map is the following:
\begin{Lemma}\label{lm:split-image-klam-zero}
   Take any partition $\lambda\in\Par[\Lambda_x]_{\alpha}$. Then $\Psi_k(\lambda)$ is $k$-horizontal.
\end{Lemma}
\begin{proof}
    This is obvious by construction.
\end{proof}
\begin{Corollary}\label{cor:initial-sequence-ordered}
    Take a partition $\lambda\in\Par[\Lambda_x]_{\alpha}$ and set $\bmu=\Phi_k\big(\Psi_k(\lambda)\big)$. Then for each $k$-node $(m,r,c)\in[\bmu]$, we have $(m,r,c+1)\in[\bmu]$, and this node has residue $k+1$. In particular, $\bi^{\bmu}\in\wellorder$.
\end{Corollary}
\begin{proof}
    By \autoref{lm:split-image-klam-zero}, $\Psi_k(\lambda)$ is $k$-horizontal. Hence, when apply $\Phi_k$, a $k$-node $A$ is replaced by a horizontal pair $(A^+,A^{\sharp})$ where $A^+$ is a $k$-node and $A^{\sharp}$ is a $(k+1)$-node.
\end{proof}

Another direct consequence of \autoref{lm:split-image-klam-zero} is the following.
\begin{Corollary}\label{cor:image-of-splitting-idempotent}
    Take a partition $\lambda\in\Par[\Lambda_x]_{\alpha}$. Then
    $$
        \Phi_k\big(e_{\Psi_k(\blam)}\big)=e_{\Phi_k\big(\Psi_k(\blam)\big)}+\be\badideal\be. \eqno{\square}
    $$
\end{Corollary}
Our main theorem in this section is the following.

\begin{Theorem}\label{thm:image-of-permutation-module}
    Fix a subdivision datum $\subdatum$ and take $\blam\in\Par[\Lambda]_{\alpha}$. Then there is an isomorphism of graded $R_\alpha$-modules
    \[
        M^{\Psi_k(\blam)}\ \cong\ 
        \be M^{\Phi_k\big(\Psi_k(\blam)\big)}\big/\be\badideal\be\,M^{\Phi_k\big(\Psi_k(\blam)\big)}.
    \]
\end{Theorem}

\begin{proof}
    Let $\bmu=\Psi_k(\blam)$ and $\bnu=\Phi_k\big(\Psi_k(\blam)\big)$, and set
    \[
        N^{\bnu}:=\be M^{\bnu}\big/\be \badideal \be\, M^{\bnu}.
    \]
    
    Let $\height(\alpha)=d$. Let $T^{\bmu}$ and $T^{\bnu}$ be the initial tableaux of shapes $\bmu$ and $\bnu$, respectively, and let $\bi^{\bmu}$ and $\bi^{\bnu}$ be the corresponding residue sequences. As usual, let
    $\bar{I}$ and $\oalpha$ be the images of $I$ and $\alpha$ under the corresponding subdivision maps
    $\Phi_k$. Let $d'=\height(\oalpha)$.

    For $\bmu$ and $\alpha$, set $K^{\bmu}_{\alpha}$ to be the ideal of $R_{\alpha}$ generated by the following set of elements:
    \begin{equation}\label{eq:generators-of-K-bmu}
        \{e(\bi)-\delta_{\bi,\bi^{\bmu}}, y_j, \psi_t \mid \bi\in I^{\alpha},1\leq j\leq d,t\rightarrow_{T^{\bmu}}t+1 \}
    \end{equation}
    then $M^{\bmu}\cong \big( R_{\alpha}/K^{\bmu}_{\alpha}\big)\langle \deg T^{\bmu}\rangle$ induced by the canonical projection $R_{\alpha}\twoheadrightarrow M^{\bmu}, r\mapsto r\cdot m^{\bmu}$. 
    
    Similarly, for $\bnu$ and $\oalpha$, set $K^{\bnu}_{\oalpha}$ to be the ideal of $R_{\oalpha}$ generated by the following set of elements:
    \[
        \{e(\bj)-\delta_{\bj,\bi^{\bnu}}, y_j, \psi_t \mid \bj\in I^{\oalpha},1\leq j\leq d',t\rightarrow_{T^{\bnu}}t+1 \}
    \]
    then $M^{\bnu}\cong \big( R_{\oalpha}/K^{\bnu}_{\oalpha}\big)\langle \deg T^{\bnu}\rangle$ induced by the canonical projection $R_{\oalpha}\twoheadrightarrow M^{\bnu}, r\mapsto r\cdot m^{\bnu}$. 

    Let $\be\in R_{\oalpha}$ be the truncation idempotent, defined in \autoref{eq:truncation-idempotent}. Then it follows that 
    \[
        \big(\be R_{\oalpha}/\be K^{\bnu}_{\oalpha}\big)\langle \deg T^{\bnu}\rangle \cong \be M^{\bnu}.
    \]
    which is induced from the canonical projection $\be R_{\oalpha}\twoheadrightarrow \be M^{\bnu}, \be r\mapsto \be r\cdot m^{\bnu}$.
    
    Since $\bi^{\bnu}\in\wellorder$ by \autoref{cor:initial-sequence-ordered}, we know $\be m^{\bnu}=m^{\bnu}$. Hence the map $\be R_{\oalpha}\be\to \be M^{\bnu}$ given by $\be r\be\mapsto \be r\be\cdot \be m^{\bnu}$ is surjective as $\be r\be\cdot \be m^{\bnu}=\be r \be m^{\bnu}=\be r m^{\bnu}$. Therefore, it follows that:
    \[
        \big(\be R_{\oalpha}\be /\be K^{\bnu}_{\oalpha}\be \big)\langle \deg T^{\bnu}\rangle \cong \be M^{\bnu}.
    \]
    Equivalently, $\be M^{\bnu}$ is a cyclic $\be R_{\oalpha}\be$-module with cyclic generator $\be m^{\bnu}$ annihilated by
    \[
        \{\be \big( e(\bj)-\delta_{\bj,\bi^{\bnu}}\big)\be, \be y_j\be , \be \psi_t\be  \mid \bj\in I^{\oalpha},1\leq j\leq d',t\rightarrow_{T^{\bnu}}t+1 \}
    \]    
    By definition of $\be$, it can be simplified as the following:
    \[
        \{e(\bj)-\delta_{\bj,\bi^{\bnu}}\be, \be y_j\be , \be \psi_t\be  \mid \bj\in \wellorder,1\leq j\leq d',t\rightarrow_{T^{\bnu}}t+1 \}
    \]
    
    Let $S_{\oalpha}$ be the balanced KLR algebra $\be R_{\oalpha}\be/ \be \badideal\be$ which is isomorphic to $R_{\alpha}$ by the subdivision map, see \autoref{thm:subdivision-iso}.
    The module $N^{\bnu}$ is naturally a cyclic $S_{\oalpha}$-module with cyclic generator $w^{\bnu}:=\be m^{\bnu}+\be\badideal\be$, which has the same degree as $m^{\nu}$. Moreover, we have:
    \begin{align*}
        \be M^{\bnu}/\be\badideal\be M^{\bnu} & \cong \Big(\big(\be R_{\oalpha}\be /\be K^{\bnu}_{\oalpha}\be \big)/\big( \be\badideal\be (\be R_{\oalpha}\be /\be K^{\bnu}_{\oalpha}\be )\big)\Big)\langle \deg T^{\bnu}\rangle\\
        & \cong \Big(\big(\be R_{\oalpha}\be /\be K^{\bnu}_{\oalpha}\be \big)/\big( (\be\badideal\be+\be K^{\bnu}_{\oalpha}\be)  /\be K^{\bnu}_{\oalpha}\be\big)\Big)\langle \deg T^{\bnu}\rangle\\
        & \cong \Big(\be R_{\oalpha}\be /(\be\badideal\be+\be K^{\bnu}_{\oalpha}\be)\Big)\langle \deg T^{\bnu}\rangle\\
        & \cong \Big(\big(\be R_{\oalpha}\be /\be \badideal\be \big)/\big( (\be\badideal\be+\be K^{\bnu}_{\oalpha}\be)  /\be \badideal\be\big)\Big)\langle \deg T^{\bnu}\rangle\\
        & = \Big( S_{\oalpha}/ \big( (\be\badideal\be+\be K^{\bnu}_{\oalpha}\be)  /\be \badideal\be\big)\Big)\langle \deg T^{\bnu}\rangle
    \end{align*}
    Set $\overline{K}^{\bnu}_{\oalpha}:=(\be\badideal\be+\be K^{\bnu}_{\oalpha}\be)  /\be \badideal\be$. Then it is generated by the following set of elements:
    \begin{equation}\label{eq:generators-of-K-bnu}
        \{e(\bj)-\delta_{\bj,\bi^{\bnu}}\be+\be\badideal\be, \be y_j\be +\be\badideal\be, \be \psi_t\be +\be\badideal\be \mid \bj\in \wellorder,1\leq j\leq d',t\rightarrow_{T^{\bnu}}t+1 \}
    \end{equation}
    and $N^{\bnu}$ is the $S_{\oalpha}$-module with cyclic generator $w^{\bnu}$ annihilated by this set of elements.

    We construct two maps:
    \[
        f: R_{\alpha}\to N^{\bnu},\qquad x\mapsto \Phi_k(x) w^{\bnu}
    \]
    and 
    \[
        g: S_{\oalpha}\to M^{\bmu},\qquad y\mapsto \Phi^{-1}_k(y) m^{\bmu}
    \]
    We want to show that 
    \begin{equation}\label{eq:mapping-kernel-to-zero}
        f(K^{\bmu}_{\alpha})=0,\qquad g(\overline{K}^{\bnu}_{\oalpha})=0
    \end{equation}
    If this is true, then $f$ and $g$ induce surjective maps:
    \[
        \overline{f}: M^{\bmu}\twoheadrightarrow N^{\bnu},\quad x m^{\bmu}\mapsto \Phi_k(x)w^{\bnu}
    \]
    and
    \[
        \overline{g}: N^{\bnu}\twoheadrightarrow M^{\bmu},\quad y w^{\bnu}\mapsto \Phi^{-1}_k(y)m^{\bmu}.
    \]
    By construction, we have $\overline{f}\circ \overline{g}=\operatorname{id}$ and $\overline{g}\circ \overline{f}=\operatorname{id}$. Moreover, by \autoref{thm:degree-invariance-standard-tableau}, $\deg T^{\bmu}=\deg T^{\bnu}$. Hence $\overline{f}$ is a degree-$0$ homogeneous isomorphism of $R_\alpha$-modules between $M^{\bmu}$ and $N^{\bnu}$.
    
    \medskip
    To prove the \autoref{eq:mapping-kernel-to-zero}, we proceed case by case, according to the type of generators in \autoref{eq:generators-of-K-bmu} and \autoref{eq:generators-of-K-bnu}.
    
    \smallskip
    \noindent\textit{Verification of $e(\bi)$-relations.}
    For $\bi\in I^\alpha$, we have
    \begin{align*}
        f\bigl(e(\bi)-\delta_{\bi,\bi^{\bmu}}\bigr)
        &=\Phi_k\big(e(\bi)-\delta_{\bi,\bi^{\bmu}}\big)\,w^{\bnu}\\
        &=\Big( e\big(\Phi_k(\bi)\big)-\delta_{\bi,\bi^{\bmu}}\Phi_k(1) +\be\badideal\be\Big)w^{\bnu}\\
        &=\Big( e(\overline{\bi})-\delta_{\bi,\bi^{\bmu}}\be+\be\badideal\be\Big)w^{\bnu}\\
        &=\Big( e(\overline{\bi})-\delta_{\overline{\bi},\bi^{\bnu}}\be+\be\badideal\be\Big)w^{\bnu}=0
    \end{align*}
     where the second equality follows from \autoref{cor:image-of-splitting-idempotent}, the third equality follows from the isomorphism $\Phi_k:R_{\alpha}\to S_{\oalpha}$ sending $1$ to $\be$ and the notation $\overline{\bi}:=\Phi_k(\bi)$, and the fourth equality follows from \autoref{eq:subdivision-bijection-words}: the map $\bi\mapsto \overline{\bi}$ gives a bijection $I^\alpha\to \wellorder$, and, by construction in \autoref{subsec:subdivision-standard-tableaux}, $\Phi_k(T^{\bmu})=T^{\bnu}$; hence $\overline{\bi}=\bi^{\bnu}$ if and only if $\bi=\bi^{\bmu}$.
    
    
    Conversely, for $\bj\in\wellorder$, again by \autoref{eq:subdivision-bijection-words}, there exists unique $\bi\in I^{\alpha}$ such that $\Phi_k(\bi)=\bj$. Hence, reversing the argument above, we have:    
    \begin{align*}
        g\Big(e(\bj)-\delta_{\bj,\bi^{\bnu}}\be+\be\badideal\be\Big)
        &=\Big(\Phi^{-1}_k\big(e(\bj)+\be\badideal\be\big)-\delta_{\bj,\bi^{\bnu}}\Phi^{-1}_k\big(\be+\be\badideal\be\big)\Big)m^{\bmu}\\
        &=\Big(e\big(\Phi_k^{-1}(\bj)\big)-\delta_{\bj,\bi^{\bnu}}\big)\Big)m^{\bmu}\\
        &=\Big(e(\bi)-\delta_{\bj,\bi^{\bnu}}\Big)m^{\bmu}\\
        &=\Big(e(\bi)-\delta_{\bi,\bi^{\bmu}}\Big)m^{\bmu}=0
    \end{align*}    
        
    \smallskip
    \noindent\textit{Verification of $\psi$-relations.}
    Let $\phi:=\phi_{\bi^{\bmu}}$ be the position-tracing function defined in \autoref{def:position-tracing-function}.
    For $1\leq t\leq d-1$ such that $t\rightarrow_{T^{\bmu}} t+1$, we have
    \begin{align*}
        f(\psi_t)
        &=\Phi_k(\psi_t)w^{\bnu}=\big(\be\psi_{\phi(t)}\be+\be\badideal\be\big)w^{\bnu}
    \end{align*}
    Let $A,B\in [\bmu]$ be the two nodes such that $T^{\bmu}(A)=t$ and $T^{\bmu}(B)=t+1$. By our construction of subdivision on standard tableaux, applied to $T^{\bmu}$, we have the following two cases. 
    \begin{enumerate}[label=$\bullet$]
        \item If $\res(A)\neq k$, then $A^+$ and $B^{+}$ are horizontally adjacent nodes in $[\bnu]$ such that $T^{\bnu}(A^+)=\phi(t)$ and $T^{\bnu}(B^+)=\phi(t+1)=\phi(t)+1$.

        \item If $\res(A)=k$, then $A^+$ and $A^{\sharp}$ are the corresponding horizontally adjacent nodes in $[\bnu]$ such that $T^{\bnu}(A^+)=\phi(t)$ and $T^{\bnu}(A^{\sharp})=\phi(t)+1$. 
    \end{enumerate}
    Hence, in any case, $\be\psi_{\phi(t)}\be+\be\badideal\be\in \overline{K}^{\bnu}_{\oalpha}$ and $f(\psi_t)=0$.

    \smallskip
    Conversely, for $t\rightarrow_{T^{\bnu}} t+1$. There are three cases: 
    \begin{enumerate}[label=$\bullet$]
        \item If $\res_{T^{\bnu}}(t)=k$, then there exists $A\in [\bmu]$ such that $T^{\bnu}(A^+)=t$ and $T^{\bnu}(A^{\sharp})=t+1$. However, in this case, since $\res(A^{\sharp})=k+1$, we have $\psi_t e(\bi^{\bnu})=e\big(\sigma_t(\bi^{\bnu})\big)\psi_t$ and $\sigma_t(\bi^{\bnu})\in\unorder$. Hence $\be\psi_t\be\in \be\badideal\be$.

        \item If $\res_{T^{\bnu}}(t)= k+1$, then there exist two horizontally adjacent nodes $A,B\in [\bmu]$ such that $T^{\bnu}(A^{\sharp})=t$ and $T^{\bnu}(B^+)=t+1$. Set $t':=\phi^{-1}(t+1)$; then $T^{\bmu}(A)=t'-1$ and $T^{\bmu}(B)=t'$, i.e. $(t'-1)\rightarrow_{T^{\bmu}} t'$. Hence we have:        
        \[
            g(\be\psi_t\be+ \be\badideal\be)=\Phi^{-1}_k\big(\be\psi_t\be+ \be\badideal\be\big)m^{\bmu}=\psi_{t'-1}m^{\bmu}=0.
        \]
        
        \item If $\res_{T^{\bnu}}(t)\neq k,k+1$, then there exist two horizontally adjacent nodes $A, B\in [\bmu]$ such that $T^{\bnu}(A^+)=t$ and $T^{\bnu}(B^+)=t+1$. Let $t'=\phi^{-1}(t)$; then $T^{\bmu}(A)=t'$ and $T^{\bmu}(B)=\phi^{-1}(t+1)=t'+1$, i.e.\ $t'\rightarrow_{T^{\bmu}} (t'+1)$. Hence we have:        
        \[
            g(\be\psi_t\be+ \be\badideal\be)=\Phi^{-1}_k\big(\be\psi_t\be+ \be\badideal\be\big)m^{\bmu}=\psi_{t'}m^{\bmu}=0.
        \]
    \end{enumerate}
    
    So, in any case, we have $g(\be\psi_t\be+\be\badideal\be)=0$. 
    
    \smallskip
    \noindent\textit{Verification of $y$-relations.}
     Let $\phi:=\phi_{\bi^{\bmu}}$ be the position-tracing function defined in \autoref{def:position-tracing-function}. Take $1\leq j\leq d$, then
     \[
        f(y_j)=\Phi_k(y_j)w^{\bnu}=\big(\be y_{\phi(j)}\be +\be\badideal\be\big)w^{\bnu}=0 
     \]

    \smallskip
    Conversely, take $1\le j\le d'$. If there exists $1\leq j'\leq d$ such that $\phi(j')=j$, then 
    \[
        g(\be y_j\be +\be\badideal\be)=\Phi_k^{-1}\big(\be y_j\be +\be\badideal\be \big)m^{\bmu}=y_{j'}m^{\bmu}=0
    \]
    Else, there exists some $k$-node $A\in [\bmu]$ such that $T^{\bnu}(A^{\sharp})=j$ and $T^{\bnu}(A^+)=j-1$. Then 
    \[
        0+\be\badideal\be=\be \psi^2_{j-1}\be+\be\badideal\be=\be \psi^2_{j-1}e(\bi^{\bnu})+\be\badideal\be=\be Q_{k,k+1}(y_{j-1},y_{j})e(\bi^{\bnu}) +\be\badideal\be=\be(y_{j}-y_{j-1})\be+\be\badideal\be
    \]
    where the first equality follows from the $\psi$-relations check in the last step, the third equality follows from \autoref{eq:psi_square} in \autoref{def:klr-algebras}, and we use $e(\bi^{\bnu})+\be\badideal\be=\be+\be\badideal\be$ in $S_{\oalpha}$, which follows from the presentation \autoref{eq:generators-of-K-bnu}. Hence    
    \[
        g\big(\be y_j\be+\be\badideal\be \big)=g\big(\be y_{j-1}\be+\be\badideal\be \big)=0.
    \]
    
    \medskip
    After verifying the three types of generators, we have proved that \autoref{eq:generators-of-K-bmu} and \autoref{eq:generators-of-K-bnu} correspond to each other under $f$ and $g$, which completes the proof.
\end{proof}

\begin{Corollary}\label{cor:image-of-permutation-module}
   Fix a subdivision datum $\subdatum$ and take $\blam\in\Par[\Lambda]_{\alpha}$. Then
    \[
        M^{\blam}\cong \be M^{\Phi_k\big(\Psi_k(\blam)\big)}/\be\badideal\be M^{\Phi_k\big(\Psi_k(\blam)\big)}
    \]
\end{Corollary}
\begin{proof}
    This follows from \autoref{cor:iso-of-permutation-module}, which gives $M^{\blam}\cong M^{\Psi_k(\blam)}$, together with \autoref{thm:image-of-permutation-module}.
\end{proof}
\subsection{Subdivision of Specht Modules}\label{subsec:subdivision-specht-modules}
We extend \autoref{thm:image-of-permutation-module} to the Specht modules by checking that the Garnir relations are preserved, thereby extending \autoref{thm:image-of-idempotent}.

We follow the definitions and notation from \autoref{subsec:subdivision-standard-tableaux}, such as old nodes $A^+$ and new nodes $A^{\sharp}$.

\begin{Lemma}\label{lm:bijection-of-garnir-nodes}
    Fix a subdivision datum $\subdatum$, and let $\blam \in \Par[\bkappa]_{\alpha}$.  
    Let $\phi:= \phi_{\bi^{\blam}}$ be the position-tracing function associated with $(k,\bi^{\blam})$.  
    Set $\bmu := \Phi_k(\blam)$, and let $T^{\blam}$ and $T^{\bmu}$ be the corresponding initial tableaux. If $[\blam]$ is $k$-horizontal, then for each $1\leq t\leq \height(\alpha)$, $A=(T^{\blam})^{-1}(t)$ is a Garnir node of $[\blam]$ if and only if $A^{+}=(T^{\bmu})^{-1}(\phi(t))$ is a Garnir node of $[\bmu]$.
\end{Lemma}

\begin{proof}
    Since $\Phi_k$ is defined componentwise, it suffices to treat the case of an ordinary partition.
    Let $\lambda$ be a partition and set $\mu=\Phi_k(\lambda)$. By the above arguments,
    for the node $A=(a,c)=(T^\lambda)^{-1}(t)\in[\lambda]$, the corresponding node $A^+$ in $[\mu]$ is \begin{equation}\label{eq:Aplus-coordinate}
        A^+=(a,\ \phi(t)-\phi(t_a)+c).
    \end{equation}
    
   Let $B:=(a+1,c)$ and, if $B\in[\lambda]$, set
    $s:=T^\lambda(B)$, so that $t\downarrow_{T^\lambda}s$. We show that
    \begin{equation}\label{eq:garnir-preserved}
        B\in[\lambda]
        \quad\Longleftrightarrow\quad
        \text{there exists a node $B'\in[\mu]$ with $A^+\downarrow_{T^\mu} B'$.}
    \end{equation}
    Since $A$ is a Garnir node of $[\lambda]$ if and only if $B\in[\lambda]$, and $A^+$ is a Garnir node of $[\mu]$
    if and only if there exists $B'$ with $A^+\downarrow_{T^\mu}B'$, this will prove the lemma.
    
    \medskip
    
    \noindent\emph{($\Rightarrow$).}
    Assume $B\in[\lambda]$ and keep the notation above. Let $B^+$ be the \emph{old} node of $[\mu]$ corresponding to $B$,
    i.e.\ the unique node with $T^\mu(B^+)=\phi(s)$. As in \eqref{eq:Aplus-coordinate}, if
    $t_{a+1}:=1+\sum_{1\le i\le a}\lambda_i$ is the first entry of row $a+1$, then
    \begin{equation}\label{eq:Bplus-coordinate}
        B^+=(a+1,\ \phi(s)-\phi(t_{a+1})+c).
    \end{equation}
    
    We distinguish two cases.
    
    \smallskip
    
    \noindent{Case 1: $\res_{T^\lambda}(t)=k+1$.}
    Then $\res_{T^\lambda}(s)=k$. In this situation, $B$ is a $k$-node and, because $k(\lambda)=0$,
    the subdivision rule for maximal $(k,k+1)$-strips inserts exactly one new node in row $a+1$ immediately to the
    \emph{right} of the old node $B^+$ (this is the new node carrying the label $\phi(s)+1$, which by definition is not
    in the image of $\phi$). 
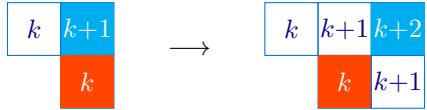
    \begin{center}
        \tikzset{
            C/.style={fill=cyan,text=white},
            O/.style={fill=OrangeRed, text=white}
                }
        \RibbonTableau[no border,box height=0.7,box width=0.7]{[C]12_{k{+}1},[O]22_{k},11_{k}}$\qquad\longrightarrow\qquad$\RibbonTableau[no border,box height=0.7,box width=0.7]{12_{k{+}1},[C]13_{k{+}2},[O]22_{k},23_{k{+}1},11_{k}}
    \end{center}
    Moreover, comparing the horizontal shifts in rows $a$ and $a+1$ up to column $c$, we have
    \[
        \phi(t)-\phi(t_a)=\phi(s)-\phi(t_{a+1})+1,
    \]
    so \eqref{eq:Aplus-coordinate}--\eqref{eq:Bplus-coordinate} give that $B^+$ lies one column to the \emph{left} of $A^+$.
    Therefore the newly inserted node immediately to the right of $B^+$ lies in the \emph{same column} as $A^+$ and is in
    row $a+1$. Denote this node by $B^{++}$; then $A^+\downarrow_{T^\mu}B^{++}$.
    
    \smallskip
    
    \noindent{Case 2: $\res_{T^\lambda}(t)\neq k+1$.}
    In this case, no new node is inserted between the images of $A$ and $B$ in the relevant column, and the horizontal
    shifts in rows $a$ and $a+1$ up to column $c$ coincide:
    \[
        \phi(t)-\phi(t_a)=\phi(s)-\phi(t_{a+1}).
    \]
    Hence \eqref{eq:Aplus-coordinate}--\eqref{eq:Bplus-coordinate} imply that $B^+$ lies directly below $A^+$ in the same
    column, so $A^+\downarrow_{T^\mu}B^+$.
    
    Thus in either case there exists $B'\in[\mu]$ with $A^+\downarrow_{T^\mu}B'$, proving the forward implication of
    \eqref{eq:garnir-preserved}.

    \noindent\emph{($\Leftarrow$).}
    Conversely, assume that there exists $B'\in[\mu]$ with $A^+\downarrow_{T^\mu}B'$, so $B'$ lies in row $a+1$ and in the
    same column as $A^+$.
    
    First suppose that $T^\mu(B')\in\im(\phi)$. Then there exists $s$ such that $T^\mu(B')=\phi(s)$, and hence
    $B'=(T^\mu)^{-1}(\phi(s))=B^+$ for the node $B:=(T^\lambda)^{-1}(s)\in[\lambda]$.
    We claim that necessarily $\res_{T^\lambda}(t)\neq k+1$. Indeed, if $\res_{T^\lambda}(t)=k+1$, then by
    \textup{Case~1} above the unique node of $[\mu]$ lying in row $a+1$ and directly below $A^+$ is the inserted node
    $B^{++}$, whose label is $\phi(s)+1\notin\im(\phi)$; this contradicts $T^\mu(B')\in\im(\phi)$.
    Therefore $\res_{T^\lambda}(t)\neq k+1$, and then \textup{Case~2} applies: the node of $[\mu]$ in row $a+1$
    directly below $A^+$ is exactly $B^+$, where $B=(a+1,c)$. Since $B'=B^+$, we conclude that $B=(a+1,c)\in[\lambda]$.
    
    Now suppose that $T^\mu(B')\notin\im(\phi)$. By \autoref{def:subdivision-young-diagram} and the assumption that $k(\lambda)=0$, labels not in
    $\im(\phi)$ occur precisely on the nodes inserted in \textup{Case~1}. In particular, the only way to have a node in
    row $a+1$ directly below $A^+$ with label outside $\im(\phi)$ is that we are in \textup{Case~1}, and then the defining
    construction of $B^{++}$ forces the existence of $B=(a+1,c)\in[\lambda]$.
    
    Thus in either case the existence of $B'\in[\mu]$ with $A^+\downarrow_{T^\mu}B'$ implies that $B=(a+1,c)\in[\lambda]$,
    which completes the reverse implication of \eqref{eq:garnir-preserved}.

\end{proof}

\begin{Lemma}\label{lm:bijection-of-garnir-relations}
    Under the hypothesis of \autoref{lm:bijection-of-garnir-nodes}, let $A\in[\blam]$ be a Garnir node such that
    $T^{\blam}(A)=t$, and let $A^+\in[\bmu]$ be the corresponding Garnir node such that $T^{\bmu}(A^+)=\phi(t)$.
    Let $g^A\in R_{\alpha}$ and $g^{A^+}\in R_{\oalpha}$ be the Garnir elements defined in \autoref{def:garnir-relation}.
    Then, in the balanced KLR algebra $S_{\oalpha}=\mathbbm{e}R_{\oalpha}\mathbbm{e}/\mathbbm{e}\badideal\mathbbm{e}$,
    we have
    \[
        \Phi_k(g^A)\;=\;\mathbbm{e}\,g^{A^+}\,\mathbbm{e}\;+\;\mathbbm{e}\badideal\mathbbm{e}.
    \]
\end{Lemma}

\begin{proof}
    By \autoref{lm:bijection-of-garnir-nodes}, $A^+$ is a Garnir node whenever $A$ is. Write $\mathcal{B}^A$ and $\mathcal{B}^{A^+}$ for the Garnir belts in $[\blam]$ and $[\bmu]$ respectively, and let
    \[
        B^A_1,\dots,B^A_{k^A}
        \qquad\text{and}\qquad
        B^{A^+}_1,\dots,B^{A^+}_{k^{A^+}}
    \]
    be the corresponding lists of row-bricks, so that the brick transpositions
    $w_r^A\in\Sym_d$ and $w_r^{A^+}\in\Sym_{d'}$ are defined by
    \[
        w_r^A=\prod_{a=0}^{e-1}\bigl(n_r^A+a,\ n_r^A+e+a\bigr),
        \qquad
        w_r^{A^+}=\prod_{a=0}^{e}\bigl(n_r^{A^+}+a,\ n_r^{A^+}+e+1+a\bigr),
    \]
    with $n_r^A=\min\{G^A(x)\mid x\in B_r^A\}$ and $n_r^{A^+}=\min\{G^{A^+}(x)\mid x\in B_r^{A^+}\}$.

    \medskip
    Let $i:=\res_{n_1^A}(G^{A})$. Then any row brick $B_j^A$ has residue sequence of one of the following forms:
    \begin{center}
        \begin{enumerate}
            \item \Tableau[no border,box height=0.7,box width=0.7]{i{i{+}1}{\cdots}k{\cdots}{e{-}1}0{\cdots}{i{-}1}}\quad if $i\le k$;
    
            \item \Tableau[no border,box height=0.7,box width=0.7]{i{i{+}1}{\cdots}{e{-}1}0{\cdots}k{\cdots}{i{-}1}}\quad if $i>k$.
        \end{enumerate}
    \end{center}
    After applying the subdivision map $\Phi_k$, we obtain the corresponding consecutive nodes in a row of $\mathcal{B}^{A^+}$, with residue sequence
    \begin{center}
        \begin{enumerate}
            \item \Tableau[no border,box height=0.7,box width=0.7]{i{i{+}1}{\cdots}k{k{+}1}{\cdots}{e}0{\cdots}{i{-}1}}\quad if $i\le k$;
    
            \item \Tableau[no border,box height=0.7,box width=0.7]{{i{+}1}{i{+}2}{\cdots}{e}0{\cdots}k{k{+}1}{\cdots}{i}}\quad if $i>k$.
        \end{enumerate}
    \end{center}
    In either case, this is the row brick $B_j^{A^+}$. Thus we obtain a bijection of bricks $B_j^A\leftrightarrow B_j^{A^+}$, and in particular $k^A=k^{A^+}$. Let $f^A$ and $f^{A^+}$ be the numbers of row-bricks in the top rows of the corresponding Garnir belts. Then $f^A=f^{A^+}$ as well. Consequently, we obtain an identification of brick permutation groups
    \[
        S^A=\langle w_1^A,\dots,w_{k^A-1}^A\rangle\ \cong\ 
        \langle w_1^{A^+},\dots,w_{k^{A^+}-1}^{A^+}\rangle=S^{A^+},
    \]
    sending $w_r^A\mapsto w_r^{A^+}$. This also sends the minimal coset representatives
    $\mathscr D^A\subseteq S^A$ bijectively onto $\mathscr D^{A^+}\subseteq S^{A^+}$.
    For $u\in\mathscr D^A$, write $u^+\in\mathscr D^{A^+}$ for its image under this bijection.

    \medskip
    Recall the brick operators in \autoref{def:brick-operators-sigma-tau}:
    \[
        \sigma_r^A=\psi_{w_r^A}\,e(\bi^A),
        \qquad
        \tau_r^A=(\sigma_r^A+1)\,e(\bi^A),
    \]
    and similarly for $A^+$.
    The diagrammatic presentation of $\psi_{w^A_r}$ is as follows (with all undrawn strands taken to be vertical straight strands):
    \begin{center}
        \resizebox{0.9\linewidth}{!}
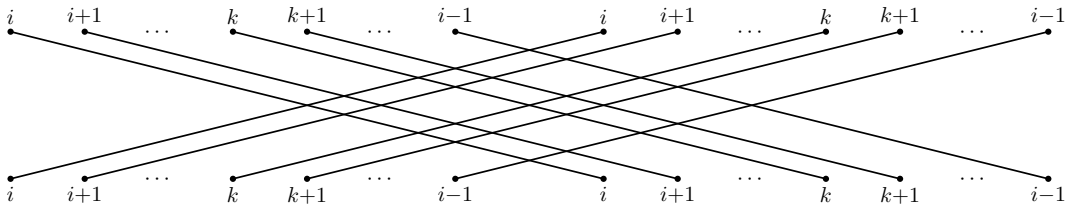
{%
        \begin{tikzpicture}[scale=1.2, thick, line cap=round]
        
            \def\yb{0}
            \def\yt{2}
            
            \coordinate (BLi)    at (-7,\yb); \coordinate (TLi)    at (-7,\yt);
            \coordinate (BLip1)  at (-6,\yb); \coordinate (TLip1)  at (-6,\yt);
            \coordinate (BLk)    at (-4,\yb); \coordinate (TLk)    at (-4,\yt);
            \coordinate (BLkp1)  at (-3,\yb); \coordinate (TLkp1)  at (-3,\yt);
            \coordinate (BLim1)  at (-1,\yb); \coordinate (TLim1)  at (-1,\yt);
            
            \coordinate (BRi)    at ( 1,\yb); \coordinate (TRi)    at ( 1,\yt);
            \coordinate (BRip1)  at ( 2,\yb); \coordinate (TRip1)  at ( 2,\yt);
            \coordinate (BRk)    at ( 4,\yb); \coordinate (TRk)    at ( 4,\yt);
            \coordinate (BRkp1)  at ( 5,\yb); \coordinate (TRkp1)  at ( 5,\yt);
            \coordinate (BRim1)  at ( 7,\yb); \coordinate (TRim1)  at ( 7,\yt);
            
            \foreach \B/\T/\lab in {
              BLi/TLi/{i},
              BLip1/TLip1/{i{+}1},
              BLk/TLk/{k},
              BLkp1/TLkp1/{k{+}1},
              BLim1/TLim1/{i{-}1},
              BRi/TRi/{i},
              BRip1/TRip1/{i{+}1},
              BRk/TRk/{k},
              BRkp1/TRkp1/{k{+}1},
              BRim1/TRim1/{i{-}1}
            }{
              \fill (\B) circle (1.2pt);
              \fill (\T) circle (1.2pt);
              \node[below] at (\B) {$\lab$};
              \node[above] at (\T) {$\lab$};
            }
            
            \node at (-5,\yb) {$\cdots$};
            \node at (-2,\yb) {$\cdots$};
            \node at ( 3,\yb) {$\cdots$};
            \node at ( 6,\yb) {$\cdots$};
            
            \node at (-5,\yt) {$\cdots$};
            \node at (-2,\yt) {$\cdots$};
            \node at ( 3,\yt) {$\cdots$};
            \node at ( 6,\yt) {$\cdots$};
            
            \draw (BLi)   -- (TRi);
            \draw (BLip1) -- (TRip1);
            \draw (BLk)   -- (TRk);
            \draw (BLkp1) -- (TRkp1);
            \draw (BLim1) -- (TRim1);
            
            \draw (BRi)   -- (TLi);
            \draw (BRip1) -- (TLip1);
            \draw (BRk)   -- (TLk);
            \draw (BRkp1) -- (TLkp1);
            \draw (BRim1) -- (TLim1);
        
        \end{tikzpicture}
        }
    \end{center}
    By the construction of the subdivision map $\Phi_k$ (see \autoref{subsec:subdivision-KLR}), the image corresponds to the following string diagram(with all undrawn strands taken to be vertical straight strands):
    \begin{center}
        \resizebox{0.9\linewidth}{!}
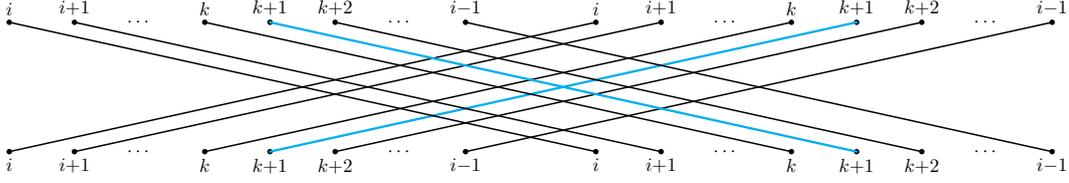
{%
        \begin{tikzpicture}[scale=1.2, thick, line cap=round]
        
            \def\yb{0}
            \def\yt{2}
            
            \coordinate (BLi)    at (-8,\yb); \coordinate (TLi)    at (-8,\yt);
            \coordinate (BLip1)  at (-7,\yb); \coordinate (TLip1)  at (-7,\yt);
            \coordinate (BLk)    at (-5,\yb); \coordinate (TLk)    at (-5,\yt);
            \coordinate (BLkp1)  at (-4,\yb); \coordinate (TLkp1)  at (-4,\yt);
            \coordinate (BLkp2)  at (-3,\yb); \coordinate (TLkp2)  at (-3,\yt);
            \coordinate (BLim1)  at (-1,\yb); \coordinate (TLim1)  at (-1,\yt);
            
            \coordinate (BRi)    at ( 1,\yb); \coordinate (TRi)    at ( 1,\yt);
            \coordinate (BRip1)  at ( 2,\yb); \coordinate (TRip1)  at ( 2,\yt);
            \coordinate (BRk)    at ( 4,\yb); \coordinate (TRk)    at ( 4,\yt);
            \coordinate (BRkp1)  at ( 5,\yb); \coordinate (TRkp1)  at ( 5,\yt);
            \coordinate (BRkp2)  at ( 6,\yb); \coordinate (TRkp2)  at ( 6,\yt);
            \coordinate (BRim1)  at ( 8,\yb); \coordinate (TRim1)  at ( 8,\yt);
            
            \foreach \B/\T/\lab in {
              BLi/TLi/{i},
              BLip1/TLip1/{i{+}1},
              BLk/TLk/{k},
              BLkp1/TLkp1/{k{+}1},
              BLkp2/TLkp2/{k{+}2},
              BLim1/TLim1/{i{-}1},
              BRi/TRi/{i},
              BRip1/TRip1/{i{+}1},
              BRk/TRk/{k},
              BRkp1/TRkp1/{k{+}1},
              BRkp2/TRkp2/{k{+}2},
              BRim1/TRim1/{i{-}1}
            }{
              \fill (\B) circle (1.2pt);
              \fill (\T) circle (1.2pt);
              \node[below] at (\B) {$\lab$};
              \node[above] at (\T) {$\lab$};
            }
            \node at (-6,\yb) {$\cdots$};
            \node at (-2,\yb) {$\cdots$};
            \node at ( 3,\yb) {$\cdots$};
            \node at ( 7,\yb) {$\cdots$};
            
            \node at (-6,\yt) {$\cdots$};
            \node at (-2,\yt) {$\cdots$};
            \node at ( 3,\yt) {$\cdots$};
            \node at ( 7,\yt) {$\cdots$};
            \draw (BLi)   -- (TRi);
            \draw (BLip1) -- (TRip1);
            \draw (BLk)   -- (TRk);
            \draw[cyan,very thick] (BLkp1) -- (TRkp1); 
            \draw (BLkp2) -- (TRkp2);
            \draw (BLim1) -- (TRim1);
            
            \draw (BRi)   -- (TLi);
            \draw (BRip1) -- (TLip1);
            \draw (BRk)   -- (TLk);
            \draw[cyan,very thick] (BRkp1) -- (TLkp1); 
            \draw (BRkp2) -- (TLkp2);
            \draw (BRim1) -- (TLim1);
        \end{tikzpicture}
        }
    \end{center}
    This string diagram corresponds to the element $\psi_{w_r^{A^+}}\in R_{\oalpha}$. Moreover, it is easy to see
    \[
        \Phi_k\big(e(\bi^A)\big)=\be\,e(\bi^{A^+})\,\be+\be\badideal\be.
    \]
    Since $\Phi_k$ is an $R_\alpha$-algebra homomorphism, we have:
    \[
        \Phi_k(\sigma_r^A)
        \;=\;
        \mathbbm{e}\,\sigma_r^{A^+}\,\mathbbm{e}\;+\;\mathbbm{e}\badideal\mathbbm{e}
        \qquad\text{in }S_{\oalpha}.
    \]
    and 
    \[
        \Phi_k(\tau_r^A)
        =
        \Phi_k\bigl((\sigma_r^A+1)e(\mathbf{i}^A)\bigr)
        =
        \bigl(\Phi_k(\sigma_r^A)+1\bigr)\Phi_k\bigl(e(\mathbf{i}^A)\bigr)
        =
        \mathbbm{e}\,\tau_r^{A^+}\,\mathbbm{e}\;+\;\mathbbm{e}\badideal\mathbbm{e}.
    \]
    
    If $u\in\mathscr D^A$ has a reduced expression $u=w_{r_1}^A\cdots w_{r_a}^A$, then
    \[
        \tau_u^A=\tau_{r_1}^A\cdots\tau_{r_a}^A,
        \qquad
        \tau_{u^+}^{A^+}=\tau_{r_1}^{A^+}\cdots\tau_{r_a}^{A^+},
    \]
    and hence multiplicativity gives
    \[
        \Phi_k(\tau_u^A)
        =
        \mathbbm{e}\,\tau_{u^+}^{A^+}\,\mathbbm{e}\;+\;\mathbbm{e}\badideal\mathbbm{e}
        \qquad(u\in\mathscr D^A).
    \]

    \medskip    
    Let $T^A$ and $T^{A^+}$ be the maximal tableaux in the Garnir sets $\Gar^{A}$ and $\Gar^{A^+}$ respectively. Recall $T^A$ is got by rearranging the row bricks of $G^A$ in $\mathcal{B}^A$ by row-reading-order and similarly for $T^{A^+}$, hence by the previous correspondence on bricks $B_{j}^{A}\leftrightarrow B_{j}^{A}$, it is immediately that
    $T^{A^+}$ is obtained from $T^A$ by
    replacing each row-brick $B^A_j$ by $B^{A^+}_j$.

    Recall outside the Garnir belts $\mathcal{B}^A$ and $\mathcal{B}^{A^+}$, $T^A$ and $T^{A^+}$ coincide with the initial tableaux $T^{\blam}$ and $T^{\bmu}$.
    This means, when compute the permutations $w^{T^A}\in\Sym_d$ such that $w^{T^A}\cdot T^{\blam}=T^A$ and $w^{T^{A^+}}\in\Sym_{d'}$ such that $w^{T^{A^+}}\cdot T^{\bmu}=T^{A^+}$, we only need the consider the changes in Garnir belts. Now the Garnir belt $\mathcal{B}^A$ in $T^A$ is the following form:
    \begin{center}
        \resizebox{0.9\linewidth}{!}
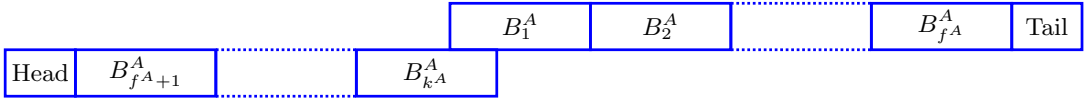
{%
        \begin{tikzpicture}[
              x=6mm, y=6mm, 
              cell/.style={draw=blue, line width=1.0pt},
              dots/.style={draw=blue, line width=1.0pt, densely dotted},
              lab/.style={font=\small}
            ]
            
            \newcommand{\Rec}[3]{%
              \draw[cell] (#1,#2) rectangle ++(3,1);
              \node[lab] at ($(#1,#2)+(1.5,0.5)$) {\(#3\)};
            }
            
            \newcommand{\dottedRec}[3]{%
              \draw[cell] (#1,#2) -- ++(0,1);
              \draw[cell] ($(#1,#2)+(3,0)$) -- ++(0,1);
              \draw[dots] (#1,#2) -- ++(3,0);
              \draw[dots] ($(#1,#2)+(0,1)$) -- ++(3,0);
              \node[lab] at ($(#1,#2)+(1.5,0.5)$) {\(#3\)};
            }
            
            \newcommand{\Bod}[3]{%
              \draw[cell] (#1,#2) rectangle ++(1.5,1);
              \node[lab] at ($(#1,#2)+(0.75,0.5)$) {\(#3\)};
            }
            
            
            \Rec       {0}{1}{B^A_{1}}
            \Rec       {3}{1}{B^A_{2}}
            \dottedRec {6}{1}{}
            \Rec       {9}{1}{B^A_{f^A}}
            \Bod       {12}{1}{\text{Tail}}
            
            \Bod       {-9.5}{0}{\text{Head}}
            \Rec       {-8}{0}{B^A_{f^A+1}}
            \dottedRec {-5}{0}{}
            \Rec       {-2}{0}{B^A_{k^A}}
        
        \end{tikzpicture}%
        }
    \end{center} 
    
    Here the head and the tail refer to the parts of the Garnir belt that do not lie in any row $A$-brick. In view of the initial tableau $T^{\blam}$ and the maximal tableau $T^A$, the string-diagram presentation of $\psi_{w^{T^A}}$ is of the form (with all undrawn strands taken to be vertical straight strands):
    
    \begin{center}
    \resizebox{0.9\linewidth}{!}
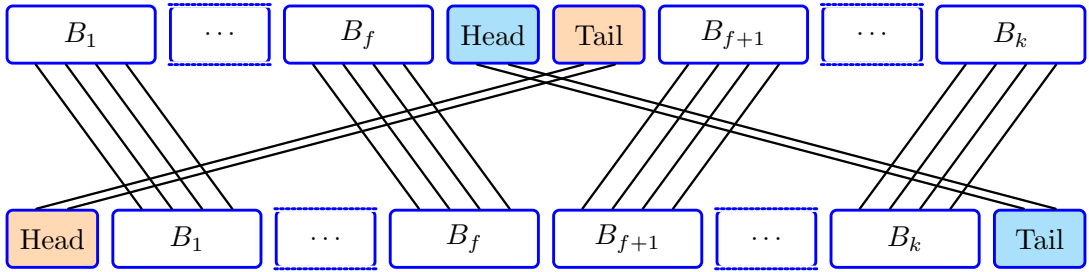
{%
        \begin{tikzpicture}[thick, line cap=round]
        
            \def\yB{0}
            \def\yT{25mm}
            
            \def\Hgt{7mm}      
            \def\sep{1.5mm}    
            \def\wHead{11mm}   
            \def\wTail{11mm}   
            \def\wBrick{18mm}  
            \def\wDots{12mm}   
            
            \tikzset{
              blanket/.style={draw=blue, line width=1.0pt, rounded corners=2pt, minimum height=\Hgt, align=center},
              head/.style={blanket, fill=orange!30, minimum width=\wHead},
              tail/.style={blanket, fill=cyan!30,   minimum width=\wTail},
              brick/.style={blanket, fill=white,    minimum width=\wBrick},
              dots/.style={draw=blue, line width=1.0pt, densely dotted},
              lab/.style={font=\small}
            }
            
            \newcommand{\DottedBlanket}[3]{%
              \node[blanket, anchor=west, minimum width=#3] (#1) at #2 {};%
              \draw[white, line width=2.2pt] (#1.north west) -- (#1.north east);
              \draw[white, line width=2.2pt] (#1.south west) -- (#1.south east);
              \draw[dots] (#1.north west) -- (#1.north east);
              \draw[dots] (#1.south west) -- (#1.south east);
              \node[lab] at (#1.center) {\(\cdots\)};
            }
            
            \newcommand{\ConnectTwo}[2]{%
              \foreach \t in {0.33,0.67}{
                \draw
                  ($(#1.north west)!\t!(#1.north east)$) --
                  ($(#2.south west)!\t!(#2.south east)$);
              }
            }
            \newcommand{\ConnectFour}[2]{%
              \foreach \t in {0.2,0.4,0.6,0.8}{
                \draw
                  ($(#1.north west)!\t!(#1.north east)$) --
                  ($(#2.south west)!\t!(#2.south east)$);
              }
            }
            \node[head,  anchor=west] (Hbot)  at (0,\yB)            {\(\mathrm{Head}\)};
            \node[brick, anchor=west] (B1bot) at ([xshift=\sep]Hbot.east) {\(B_1\)};
            
            \DottedBlanket{gap1bot}{([xshift=\sep]B1bot.east)}{\wDots}
            \node[brick, anchor=west] (Bfbot)  at ([xshift=\sep]gap1bot.east) {\(B_f\)};
            
            \node[brick, anchor=west] (Bfp1bot) at ([xshift=\sep]Bfbot.east) {\(B_{f+1}\)};
            
            \DottedBlanket{gap2bot}{([xshift=\sep]Bfp1bot.east)}{\wDots}
            \node[brick, anchor=west] (Bkbot)  at ([xshift=\sep]gap2bot.east) {\(B_k\)};
            
            \node[tail,  anchor=west] (Tbot)  at ([xshift=\sep]Bkbot.east) {\(\mathrm{Tail}\)};
            
            \node[brick, anchor=west] (B1top) at (0,\yT) {\(B_1\)};
            
            \DottedBlanket{gap1top}{([xshift=\sep]B1top.east)}{\wDots}
            \node[brick, anchor=west] (Bftop) at ([xshift=\sep]gap1top.east) {\(B_f\)};
            
            \node[tail, anchor=west] (Ttop) at ([xshift=\sep]Bftop.east) {\(\mathrm{Head}\)};
            \node[head, anchor=west] (Htop) at ([xshift=\sep]Ttop.east)  {\(\mathrm{Tail}\)};
            
            \node[brick, anchor=west] (Bfp1top) at ([xshift=\sep]Htop.east) {\(B_{f+1}\)};
            
            \DottedBlanket{gap2top}{([xshift=\sep]Bfp1top.east)}{\wDots}
            \node[brick, anchor=west] (Bktop) at ([xshift=\sep]gap2top.east) {\(B_k\)};
            
            \ConnectTwo {Hbot}{Htop}
            \ConnectTwo {Tbot}{Ttop}
            
            \ConnectFour{B1bot}{B1top}
            \ConnectFour{Bfbot}{Bftop}
            \ConnectFour{Bfp1bot}{Bfp1top}
            \ConnectFour{Bkbot}{Bktop}
        \end{tikzpicture}%
    }
    \end{center}
    
    By definition, $T^{A^+}$ and $\psi_{w^{T^{A^+}}}$ admit analogous descriptions. By the correspondence of row bricks and the definition of subdivision via string diagrams, we have:
    \[
        \Phi_k(\psi^{T^A})
        \;=\;
        \mathbbm{e}\,\psi^{T^{A^+}}\,\mathbbm{e}\;+\;\mathbbm{e}\badideal\mathbbm{e}
        \qquad\text{in }S_{\oalpha}.
    \]

    \medskip
    Recall the Garnir elements are of the form:
    \[
        g^A=\sum_{u\in\mathscr D^A}\tau_u^A\,\psi^{T^A},
        \qquad
        g^{A^+}=\sum_{v\in\mathscr D^{A^+}}\tau_v^{A^+}\,\psi^{T^{A^+}}.
    \]
    
    Computing in $S_{\oalpha}$:
    \[
    \begin{aligned}
        \Phi_k(g^A)
        &=
        \sum_{u\in\mathscr D^A}\Phi_k(\tau_u^A)\,\Phi_k(\psi^{T^A})\\
        &=
        \sum_{u\in\mathscr D^A}
        \Bigl(\mathbbm{e}\,\tau_{u^+}^{A^+}\,\mathbbm{e}+\mathbbm{e}\badideal\mathbbm{e}\Bigr)
        \Bigl(\mathbbm{e}\,\psi^{T^{A^+}}\,\mathbbm{e}+\mathbbm{e}\badideal\mathbbm{e}\Bigr)\\
        &=
        \mathbbm{e}\Bigl(\sum_{u^+\in\mathscr D^{A^+}}\tau_{u^+}^{A^+}\,\psi^{T^{A^+}}\Bigr)\mathbbm{e}
        \;+\;\mathbbm{e}\badideal\mathbbm{e}\\
        &=
        \mathbbm{e}\,g^{A^+}\,\mathbbm{e}\;+\;\mathbbm{e}\badideal\mathbbm{e}.
    \end{aligned}
    \]
\end{proof}

We can now state the main theorem of this section:
\begin{Theorem}\label{thm:image-of-specht-module}
    Fix a subdivision datum $\subdatum$ and take $\blam\in\Par[\Lambda]_{\alpha}$. Then there is an isomorphism of graded $R_\alpha$-modules
    \[
        S^{\Psi_k(\blam)}\ \cong\ 
        \be S^{\Phi_k\big(\Psi_k(\blam)\big)}\big/\be\badideal\be\,S^{\Phi_k\big(\Psi_k(\blam)\big)}.
    \]
\end{Theorem}

\begin{proof}
    Let $\bmu=\Psi_k(\blam)$ and $\bnu=\Phi_k(\bmu)$, and set
    \[
        \widetilde S^{\bnu}:=\be S^{\bnu}\big/\be\badideal\be\,S^{\bnu}.
    \]
    Write $\widetilde z^{\bnu}:=\be z^{\bnu}+\be\badideal\be\,S^{\bnu}\in \widetilde S^{\bnu}$ for the image of the
    standard cyclic generator $z^{\bnu}$. Then $\widetilde S^{\bnu}$ is a cyclic $S_{\oalpha}$-module generated by
    $\widetilde z^{\bnu}$.
    
    \smallskip
    By \autoref{thm:image-of-permutation-module} we already have an isomorphism of graded $R_\alpha$-modules
    \[
        M^{\bmu}\ \cong\ \be M^{\bnu}\big/\be\badideal\be\,M^{\bnu}.
    \]
    induced by sending the cyclic generator $m^{\bmu}$ of $M^{\bmu}$ to the cyclic generator $w^{\bnu}:=\be m^{\bnu}+\be\badideal\be$ of $\be M^{\bnu}\big/\be\badideal\be\,M^{\bnu}$.
    Since $S^{\bmu}$ (resp.\ $S^{\bnu}$) is obtained from $M^{\bmu}$ (resp.\ $M^{\bnu}$) by imposing the Garnir relations, to check the map $S^{\bmu}\to \widetilde S^{\bnu}$ given by $z^{\bmu}\mapsto \widetilde z^{\bnu}$ is an isomorphism, it remains to check that the Garnir relations correspond under $\Phi_k$.

    By \autoref{lm:bijection-of-garnir-nodes} and \autoref{lm:bijection-of-garnir-relations}, let $\phi:=\phi_{\bi^{\bmu}}$ be the position-tracing function associated to $(k,\bi^{\bmu})$. Then the subdivision map $\Phi_k$ induces a bijection between Garnir nodes $A\in[\bmu]$ and the corresponding nodes $A^+\in[\bnu]$, and it matches the Garnir elements $g^A$ with $\be g^{A^+}\be+\be\badideal\be$. Recall that the nodes $B\in[\bmu]$ of the form $B=A^+$ are precisely those satisfying $T(B)=\phi(t)$ for some $1\le t\le d$. It therefore remains to consider those nodes $X$ that are not of the form $A^+$ for any $A\in[\bmu]$. These new nodes all have residue $k+1$ and occur immediately to the right of a $k$-node. In particular, the Garnir belt $\mathcal{B}^X$, filled with residues, has the following form:

    \begin{center}
    \resizebox{0.9\linewidth}{!}
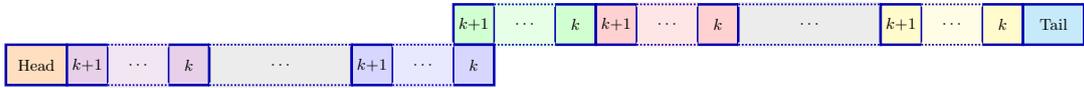
{%
        \begin{tikzpicture}[
          x=8mm, y=8mm,
          outer/.style={draw=blue!70!black, line width=1.4pt},
          inner/.style={draw=blue!70!black, line width=0.9pt},
          dots/.style ={draw=blue!70!black, line width=1.0pt, densely dotted},
          lab/.style  ={font=\small}
        ]
        
            \def\wSide{1}        
            \def\wMid{1.5}       
            \def\wMidHalf{0.75}  
            \def\wRec{3.5}       
            \def\wBod{1.5}       
            
            \newcommand{\Rec}[6]{%
              \path[fill=#4] (#1,#2) rectangle ++(\wSide,1);
              \path[fill=#5] ($(#1,#2)+(\wSide,0)$) rectangle ++(\wMid,1);
              \path[fill=#6] ($(#1,#2)+(\wSide+\wMid,0)$) rectangle ++(\wSide,1);
            
              \draw[outer] (#1,#2) rectangle ++(\wRec,1);
            
              \draw[inner] ($(#1,#2)+(\wSide,0)$) -- ++(0,1);
              \draw[inner] ($(#1,#2)+(\wSide+\wMid,0)$) -- ++(0,1);
            
              \draw[draw=#5, line width=2.8pt] ($(#1,#2)+(\wSide,0)$) -- ++(\wMid,0);
              \draw[draw=#5, line width=2.8pt] ($(#1,#2)+(\wSide,1)$) -- ++(\wMid,0);
              \draw[dots] ($(#1,#2)+(\wSide,0)$) -- ++(\wMid,0);
              \draw[dots] ($(#1,#2)+(\wSide,1)$) -- ++(\wMid,0);
            
              \node[lab] at ($(#1,#2)+(0.5,0.5)$) {\(#3{+}1\)};
              \node[lab] at ($(#1,#2)+(\wSide+\wMidHalf,0.5)$) {\(\cdots\)};
              \node[lab] at ($(#1,#2)+(\wSide+\wMid+0.5,0.5)$) {\(#3\)};
            }
            
            \newcommand{\dottedRec}[3]{%
              \path[fill=#3] (#1,#2) rectangle ++(\wRec,1);
              \draw[outer] (#1,#2) rectangle ++(\wRec,1);
            
              \draw[draw=#3, line width=2.8pt] (#1,#2) -- ++(\wRec,0);
              \draw[draw=#3, line width=2.8pt] ($(#1,#2)+(0,1)$) -- ++(\wRec,0);
              \draw[dots] (#1,#2) -- ++(\wRec,0);
              \draw[dots] ($(#1,#2)+(0,1)$) -- ++(\wRec,0);
            
              \node[lab] at ($(#1,#2)+(0.5*\wRec,0.5)$) {\(\cdots\)};
            }
            
            \newcommand{\Bod}[4]{%
              \path[fill=#3] (#1,#2) rectangle ++(\wBod,1);
              \draw[outer] (#1,#2) rectangle ++(\wBod,1);
              \node[lab] at ($(#1,#2)+(0.5*\wBod,0.5)$) {\(#4\)};
            }
            
            
            \Rec       {0}{1}{k}{green!18}{green!10}{green!18}
            \Rec       {\wRec}{1}{k}{red!18}{red!10}{red!18}
            \dottedRec {2*\wRec}{1}{gray!15}
            \Rec       {3*\wRec}{1}{k}{yellow!25}{yellow!12}{yellow!25}
            \Bod       {4*\wRec}{1}{cyan!20}{\text{Tail}}
            
            \Bod       {-11}{0}{orange!25}{\text{Head}}
            \Rec       {-9.5}{0}{k}{violet!18}{violet!10}{violet!18}
            \dottedRec {-6}{0}{gray!15}
            \Rec       {-2.5}{0}{k}{blue!16}{blue!9}{blue!16}
        
        \end{tikzpicture}%
    }
    \end{center}

    In particular, outside the Garnir belt $\mathcal{B}^X$, the last entry in $T^{\bnu}$ (and hence in $T^X$ and $G^X$) before any entry of $\mathcal{B}^X$ has residue $k$. Then the residue sequence $\bi^{X}$ of the Garnir tableau $G^{X}$ (and hence the residue sequence of every tableau in the Garnir set $\Gar{X}$) is:
    \begin{center}
        \[
            \cdots,k, \res(\text{Head}),\underbrace{k{+}1,\cdots,e,0,\cdots,k}_{k^X\text{ times }},\res(\text{Tail}),\cdots
        \]
    \end{center}
    Here $\res(\text{Head})$ and $\res(\text{Tail})$ are the residue sequences corresponding to the head and tail parts of the Garnir belt. Note that the first node in $\text{Head}$ cannot have residue $k+1$: by assumption there are only non-trivial $(k,k+1)$-strips, so every $(k+1)$-node must occur immediately to the right of a $k$-node. In particular, $\bi^X\notin\wellorder$, since there is an occurrence of $k$ that is not immediately followed by $k+1$. Therefore $e(\bi^X)\be=\be e(\bi^X)=0$, and hence:
    \begin{align*}
        \be g^X\be +\be\badideal\be
        &=
        \be\sum_{u\in\mathscr D^X}\tau_u^X\,e(\bi^X)\,\psi^{T^X}\be+\be\badideal\be\\
        &=
        \be\sum_{u\in\mathscr D^X}\tau^X_{r^u_1}\cdots\tau^X_{r^u_{\ell(u)}}\,e(\bi^X)\,\psi^{T^X}\be+\be\badideal\be\\
        &=\be\sum_{u\in\mathscr D^X}e(\bi^X)\,\tau^X_{r^u_1}\cdots\tau^X_{r^u_{\ell(u)}}\,e(\bi^X)\,\psi^{T^X}\be+\be\badideal\be\\
        &=0\in S_{\oalpha}.
    \end{align*}
    where the second-to-last equality holds because $\tau^X_{i}=(\sigma^X_i+1)e(\bi^X)$ and $\sigma^X_i$ only permutes the row bricks inside the Garnir belt $\mathcal{B}^X$, and hence does not change the residue sequence $\bi^X$. Therefore, we have shown that the Garnir element $g^X$ vanishes in $S_{\oalpha}$ whenever $X$ is not of the form $A^+$ for any $A\in[\bmu]$. Now the theorem follows from \autoref{thm:image-of-permutation-module}.
\end{proof}

\begin{Corollary}\label{cor:image-of-specht-module}
    Fix a subdivision datum $\subdatum$, and take $\blam\in \Par[\bkappa]_{\alpha}$ to be $k$-horizontal. Then there is an isomorphism of graded $R_\alpha$-modules
    \[
        S^{\blam}\ \cong\ 
        \be S^{\Phi_k(\blam)}\big/\be\badideal\be\,S^{\Phi_k(\blam)}.
    \]
\end{Corollary}
\begin{proof}
    This follows from \autoref{lm:split-no-change} and \autoref{thm:image-of-specht-module}.
\end{proof}


\section{Connection with Runner Removal Theorems}\label{sec:categorification}
In this section, we explain that the combinatorics of the subdivision map developed in the previous section coincides with the runner removal theorems, thereby providing a natural categorification.
\subsection{Level 1 Case}\label{subsec:level-one-runner-removal}
Fix a subdivision datum $\subdatum[\Lambda=\Lambda_x][x]$ such that $\height(\alpha)=n$. By the Brundan--Kleshchev isomorphism \cite[Theorem 1.1]{bk-blocks-iso}, the level $1$ cyclotomic KLR algebras are isomorphic to the corresponding Iwahori--Hecke algebras $\mathcal{H}_q(\Sym_n)$, where $q$ is a primitive $e$-th root of unity. Hence, in this section, the Specht modules and simple modules can be viewed as modules over the classical Iwahori--Hecke algebras, equipped with a grading coming from the KLR grading. See also \cite{bkw-graded-specht,humathas-graded-cellular,kmr-universal-specht-type-A} for three different approaches to construct graded Specht modules.

\medskip
Following \cite{fayers-full-runner-removal}, we define an operation $+$ on partitions. Fix an integer $d\in\Z$. For a partition $\lambda\in\Par[\Lambda]_\alpha$, choose a positive integer $a\geq \ell(\lambda)$ such that $a+d\geq 0$.
Let $c,k$ be the unique non-negative integers satisfying
\[
    a+d=ce+k,\quad k\in I=\{0,1,\cdots,e-1\}.
\]
Form the $e$-abacus of $\lambda$ with $a$ beads, and add a flush runner with $c$ beads on the left of $k$-runner. This new $(e+1)$-abacus with $a+c$ beads corresponds uniquely to a partition, which is denoted by $\lambda^{+}:=\lambda^{+d}$. It is not hard to verify that this definition of $\lambda^+$ is independent of $a$.

The importance of this construction is illustrated by the following theorem.
\begin{Theorem}[{(Runner Removal Theorem)}]\label{thm:runner-removal-thms-level-1}
    Assume that the base field $\bk$ of the Hecke algebras has characteristic $0$. Suppose that $\lambda,\mu\in\Par[\Lambda]_\alpha$ and that $\lambda$ is $e$-regular. Fix $d\in\Z$. If one of the following holds:    \begin{itemize}
        \item $d\ge\lambda_1$ (cf. \textup{\cite[Theorem~3.1]{fayers-full-runner-removal}});
        \item $d< -\ell(\lambda)$ (cf. \textup{\cite[Theorem~4.5]{jamesmathas-empty-runner-removal}} and \textup{\cite[Theorem~4.1]{alice-empty-runner-removal}}).
    \end{itemize}
    Then $d^{e}_{\mu,\lambda}(q)=d^{e+1}_{\mu^{+d},\lambda^{+d}}(q)$. 
\end{Theorem}

The $e$-regular assumption is not necessary, since they work with $q$-Schur algebras. In fact, they prove a stronger statement in terms of the canonical basis of the Fock space.
Let $\mathcal{F}_e(\Lambda)$ be the (level-$1$) $q$-Fock space: the free $\mathbb{Q}(q)$-vector space
with standard basis $\{\,|\lambda\rangle \mid \lambda\in \Par[\Lambda]\,\}$.
It carries an integrable $U_q(\widehat{\mathfrak{sl}}_e)$-module structure, and the submodule
generated by the vacuum vector $|\varnothing\rangle$ is isomorphic to the irreducible
highest-weight module $V(\Lambda)$.

The bar involution on $U_q(\widehat{\mathfrak{sl}}_e)$ induces a bar involution on $V(\Lambda)$,
which extends to a bar involution on the whole of $\mathcal{F}_e(\Lambda)$ by work of \cite{lt-canonical-bases-q-deformed-fock}.
Consequently, for every partition $\mu\in \Par[\Lambda]$ there is a unique bar-invariant vector
\begin{equation}\label{eq:canonical-basis-expansion}
  G_e(\mu)=\sum_{\lambda\in \Par[\Lambda]} d^{e}_{\lambda\mu}(q)\,|\lambda\rangle,
\end{equation}
such that $d^{e}_{\mu\mu}(q)=1$ and $d^{e}_{\lambda\mu}(q)\in q\mathbb{Z}[q]$ for $\lambda\neq\mu$;
moreover $d^{e}_{\lambda\mu}(q)=0$ unless $\lambda\unlhd \mu$ in dominance order of partitions. The set $\{G_e(\mu)\mid \mu\in \Par[\Lambda]\}$ is called the \emph{canonical basis} of $\mathcal{F}_e(\Lambda)$.

When $\mu$ is $e$-regular, the element $G_e(\mu)$ lies in $V(\Lambda)$ and coincides with the canonical basis
element in $V(\Lambda)$ indexed by $\mu$. If the corresponding Iwahori--Hecke algebra is defined over a field of
characteristic $0$ and is specialised at a primitive $e$th root of unity, then Ariki's categorification theorem \cite[Theorem 4.4]{ariki-categorification} identifies
$d^{e}_{\lambda\mu}(1)$ with the decomposition number $[S^\lambda:D^\mu]$ (with $\mu$  $e$-regular). Thanks to the Brundan--Kleshchev isomorphism \cite{bk-blocks-iso}, Brundan and Kleshchev lift Ariki's categorification theorem to the graded case in \cite[Section 5.5]{bk-graded-decomposition-number}: $d_{\lambda\mu}^{e}(q)$ is the graded decomposition number $[S^\lambda:D^\mu]_q$.

Their proof of \autoref{thm:runner-removal-thms-level-1} proceeds by extending
$\lambda\mapsto \lambda^{+d}$ to a linear operator on the whole Fock space and then establishing
\begin{equation}\label{eq:canonical-basis-subdivision-compatibility}
  G_{e+1}(\lambda^{+d}) \;=\; G_{e}(\lambda)^{+d}.
\end{equation}
For standard references on level $1$ Fock spaces, decomposition numbers, quantum groups, and related topics,
see \cite[Chapter 6]{mathas-iwahori-hecke}.

\bigskip
The construction of $\lambda^{+d}$ is very similar to the subdivision on partitions in \autoref{def:subdivision-abacus}, but it has some different features.
In \autoref{def:subdivision-abacus} one starts with a fixed integer $k\in I$, which specifies the position in the abacus where the new runner is inserted, whereas the construction of $\lambda^{+d}$ starts with the parameter $d\in\Z$. The two parameters are linked by the equation $a+d=ce+k$. 

In \autoref{def:subdivision-abacus}, we require $a\equiv x\pmod e$, while this is not required in the definition of $\lambda^{+d}$. This is because, in level $1$, we have $V(\Lambda_0)\cong V(\Lambda_x)$ for any $x\in I$. For convenience, one often also chooses $a\equiv x\pmod e$ in the construction of $\lambda^{+d}$.
Moreover, in \autoref{def:subdivision-abacus}, since $k\in I$, we necessarily have $d\in I$, whereas in the construction of $\lambda^{+d}$ the parameter $d$ can be any integer. Thus \autoref{def:subdivision-abacus} can be viewed as a special case of the construction of $\lambda^{+d}$, with an equivalent description in Young diagrams given by \autoref{def:subdivision-young-diagram}.\footnote{It would be interesting to find a general description of $\lambda^{+d}$ in terms of Young diagrams.}

The most important feature of this special case (namely $d\in I$) is that the hypotheses of \autoref{thm:runner-removal-thms-level-1} are rarely satisfied. In other words, for a general partition $\lambda$, the interval $[0,e-1]$ is an exceptional set on which \autoref{thm:runner-removal-thms-level-1} does not apply.

Nevertheless, there is a general runner removal theorem that applies in this situation. To state it, we first need an object called the extended beta multiset.

Recall that the $a$-beta set $B(\lambda;a)$ of a partition is the set of $a$-beta numbers of $\lambda$,
$B(\lambda;a)=\{\beta_i^a(\lambda)\mid 1\le i\le a\}$. The extended beta multiset can be viewed as a multiplicity-counting version of the beta set. For each $z\in\Z_{\ge 0}$, define the multiplicity $\operatorname{mul}_z$ to be $\#\bigl(B(\lambda;a)\cap\{z,z+e,z+2e,\dots\}\bigr)$. The extended beta multiset $\mathfrak{X}_a^e(\lambda)$ is the union of $\{\underbrace{z,\cdots,z}_{\operatorname{mul}_z}\}$ over all $z\in\Z_{\ge 0}$, taken as a multiset.

For each fixed integer $d\in\Z$, define $\epsilon_d(\lambda)$ to be the number of elements $z\in\mathfrak{X}_a^e(\lambda)$, counted with multiplicity, such that $z\ge a+d$. It is an exercise to show that $\epsilon_d(\lambda)$ does not depend on the choice of $a$.
\begin{Example}\label{eg:extended-beta-number-and-fayers-arithmetic}
    Fix type $\Aone[3]$ and $d=4$. Let $\lambda=(5,4,2,1,1)$ and $\mu=(7,2,2,1,1)$. Take $a=9$. Then $a+d=13=4\cdot 3+1$, so $c=3$ and $k=1$. The $e$-abaci with $9$ beads for $\lambda$ and $\mu$ are as follows:
    \begin{center}
        \Abacus[runner labels={0,1,2,3},entries=betas]{4}{5,4,2,1,1,0,0,0,0}\qquad
        \Abacus[runner labels={0,1,2,3},entries=betas]{4}{7,2,2,1,1,0,0,0,0}
    \end{center}
    The beads are labelled by the corresponding beta numbers. The extended beta multisets of $\lambda$ and $\mu$ are
    \begin{align}
        \mathfrak{X}_a^e(\lambda)=\{0,0,1,1,1,2,2,3,3,4,5,5,6,7,8,9,11,13\},\\
        \mathfrak{X}_a^e(\mu)=\{0,0,1,1,1,2,2,3,3,4,5,5,6,7,8,9,11,15\}.
    \end{align}
    Since $a+d=13$, we have $\epsilon_d(\lambda)=1$ and $\epsilon_d(\mu)=1$. The partitions $\lambda^{+d}$ and $\mu^{+d}$ have the following abacus displays:
    \begin{equation}
        \Abacus[runner labels={0,1,2,3,4},entries=betas]{5}{6,4,*2,2,1,1,*1,0,0,0,*0,0}\qquad
        \Abacus[runner labels={0,1,2,3,4},entries=betas]{5}{8,2,*2,2,1,1,*1,0,0,0,*0,0}
    \end{equation}
    Hence $\lambda^{+d}=(6,4,2,2,1,1,1)$ and $\mu^{+d}=(8,2,2,2,1,1,1)$.
\end{Example}

Now we are able to state the \emph{general runner removal theorem}.
\begin{Theorem}[{\cite[Theorem 3.4]{fayers-general-runner-removal}}]\label{thm:general-runner-removal-level-1}
    Take $d\in\Z$ and let $\lambda,\mu\in\Par[\Lambda]_{\alpha}$ with $\lambda$ $e$-regular. Set $\lambda^+:=\lambda^{+d}$ and $\mu^+:=\mu^{+d}$. If $\epsilon_d(\lambda)=\epsilon_d(\mu)$, then $d^{e}_{\mu,\lambda}(q)=d^{e+1}_{\mu^{+},\lambda^{+}}(q)$.
\end{Theorem}

We remark, following \cite[Section 4.3]{fayers-general-runner-removal}, that \autoref{thm:general-runner-removal-level-1} actually \emph{contains} \autoref{thm:runner-removal-thms-level-1} as a special case; this is why it is called the general runner removal theorem. Another key feature is that the condition in \autoref{thm:general-runner-removal-level-1} is a relation between the two partitions under consideration, rather than an absolute requirement as in \autoref{thm:runner-removal-thms-level-1}.

Unlike \autoref{thm:runner-removal-thms-level-1}, \autoref{thm:general-runner-removal-level-1} imposes no restriction on the parameter $d$. By specialising to $0 \le d \le e-1$, it applies to our subdivision map on partitions. In view of \autoref{thm:image-of-specht-module}, the subdivision isomorphism $\Phi_k$ \emph{categorifies} this runner removal theorem over a field of characteristic $0$.

\begin{Remark}\label{rmk:comparison-with-fayers-general-runner-removal}
    There are several differences between the setting of \cite{fayers-general-runner-removal} and ours, and we briefly discuss some of them here.
    Fayers studies runner removal in \cite{fayers-general-runner-removal}, whereas this paper treats runner addition. In other words, we pass from $\lambda$ to $\lambda^{+d}$, that is, from an $e$-abacus to an $(e+1)$-abacus, while Fayers passes from $\lambda$ to $\lambda^{-d}$, that is, from an $(e+1)$-abacus to an $e$-abacus.
    
    For the reader's convenience, we clarify the notation. Let $\overline{d}$ be the unique integer in $[0,e]$ such that $\overline{d}\equiv d\pmod{e+1}$. Then $\lambda^{+d}$ is $\overline{d}$-empty and $\lambda=(\lambda^{+d})^{-\overline{d}}$ in Fayers' notation. It is also straightforward to check that $\epsilon_d(\lambda)=\mathfrak{L}_{\overline{d}}(\lambda^{+d})$, where $\mathfrak{L}_{\overline{d}}(-)$ is defined in \cite[Definition 3.3]{fayers-general-runner-removal}.
    
    The reader should be careful that we use the equation $a+d=ce+k$, whereas \cite{fayers-general-runner-removal} uses $a+k=ce+d$. Thus, when comparing results, the parameters $k$ and $d$ should be swapped.
\end{Remark}

\subsection{Higher Level Case}\label{subsec:higher-level-runner-removal}
The discussion of Fock spaces and canonical bases in \autoref{subsec:level-one-runner-removal} has an analogous, but more complex, higher-level theory; see \cite{bk-graded-decomposition-number} or \cite[Chapter 6]{geckjacon-hecke-algebras-root-of-unity} for a general discussion of graded decomposition numbers of cyclotomic KLR algebras of type $\Aone[e-1]$ and related topics, including higher-level Fock spaces.

Recently, \autoref{thm:runner-removal-thms-level-1} has been generalized to higher levels in \cite{alice-empty-runner-removal,alice-full-runner-removal} via computations in higher-level Fock spaces, whereas \autoref{thm:general-runner-removal-level-1} has not yet been extended. Nevertheless, we formulate a conjecture below concerning our subdivision.

Fix a subdivision datum $\subdatum$ and take an $\ell$-partition $\blam=(\blam^{(1)},\cdots,\blam^{(\ell)})\in\Par[\bkappa]_{\alpha}$. For each $1\le i\le \ell$, take an abacus subdivision datum $\absubdatum=(a_i,c_i,d_i,a_i')$ for $\blam^{(i)}$.

We then extend the level-$1$ construction of $\lambda^{+d}$ to higher levels by defining it componentwise:
$$
  \blam^{+\bbd}
  := \Bigl( \bigl(\blam^{(1)}\bigr)^{+d_{1}},\,\dots,\,\bigl(\blam^{(\ell)}\bigr)^{+d_{\ell}} \Bigr).
$$
Set
$$
  \epsilon_{\bbd}(\blam)
  := \bigl(\epsilon_{d_{1}}(\blam^{(1)}),\,\dots,\,\epsilon_{d_{\ell}}(\blam^{(\ell)})\bigr),
  \quad
  \bigl|\epsilon_{\bbd}(\blam)\bigr|
  := \sum_{i=1}^\ell \epsilon_{d_{i}}(\blam^{(i)}).
$$
Recall from \autoref{sec:combinatorial-subdivision} that the subdivision map on partitions is defined componentwise. As a result, by the discussion in \autoref{subsec:level-one-runner-removal}, if $d_i\in I$ for each $1\le i\le \ell$, then ${(\cdot)}^{+\bbd}$ coincides with the subdivision map, that is, $\Phi_k(\blam)=\blam^{+\bbd}$. In view of \autoref{thm:image-of-specht-module}, the subdivision provides a categorification of this case.

We can now state the following two conjectures:
\begin{Conjecture}\label{conj:general-runner-removal-high-level}
    Let $\bbd=(d_{1},\cdots,d_{\ell})\in I^\ell$ such that $0\leq a_i+d_{i}=c_ie+k$ for each $1\leq i\leq \ell$, take $\blam,\bmu\in\Par[\kappa]_{\alpha}$, if $\epsilon_{\bbd}(\blam)=\epsilon_{\bbd}(\bmu)$, then $d^{e}_{\bmu,\blam}(q)=d^{e+1}_{\bmu^{+\bbd},\blam^{+\bbd}}(q)$.
\end{Conjecture}
\begin{Conjecture}\label{conj:general-runner-removal-level-two}
    Let $\bbd=(d_{1},d_{2})\in I^2$ such that $0\leq a_i+d_{i}=c_ie+k$ for each $1\leq i\leq 2$, take $\blam,\bmu\in\Par[\kappa]_{\alpha}$, if $|\epsilon_{\bbd}(\blam)|=|\epsilon_{\bbd}(\bmu)|$, then $d^{e}_{\bmu,\blam}(q)=d^{e+1}_{\bmu^{+\bbd},\blam^{+\bbd}}(q)$.
\end{Conjecture}
We remark that the hypothesis of \autoref{conj:general-runner-removal-level-two} is much weaker than that of \autoref{conj:general-runner-removal-high-level}. However, it can only reasonably be expected to hold in level two, since analogous statements admit counterexamples in higher levels. We end this section with an example supporting \autoref{conj:general-runner-removal-level-two}.
\begin{Example}\label{eg:general-runner-removal-level-2}
    Fix quiver $\Aone[2]$. Fix $k=1$ and the charge $\bkappa=(0,1)$. Take a $2$-partition $\blam = \bigl((6),\,(5,1,1)\bigr)$.
    By choosing suitable $\mathbf{a}$ and the corresponding $\bbd$ (for example, take $\mathbf{a}=(3,4)$ and $\bbd=(1,0)$), one easily verifies that
    $\blam^{+\bbd} = \bigl((8),\,(7,1,1)\bigr)$ and $\epsilon_{\bbd}(\blam) = (2,2)$.

    The canonical basis element $G_e(\blam)$ is given in \autoref{fig:Ge_expansion}. The terms coloured \textcolor{red}{red} correspond to $2$-partitions $\bmu$ such that $\bmu^{+\bbd}$ \emph{does not} occur in the expansion of $G_{e+1}(\blam^{+\bbd})$.
    The terms coloured \textcolor{cyan}{cyan} are those that \emph{do} occur and satisfy $|\epsilon_{\bbd}(\blam)|=|\epsilon_{\bbd}(\bmu)|=4$.
    The remaining terms likewise occur in the expansion of $G_{e+1}(\blam^{+\bbd})$ but satisfy $|\epsilon_{\bbd}(\blam)|\neq|\epsilon_{\bbd}(\bmu)|$.

    The canonical basis element $G_{e+1}(\blam^{+\bbd})$ is given in \autoref{fig:Ge_plus_1_expansion}:
    the \textcolor{cyan}{cyan} terms in $G_{e+1}(\blam^{+\bbd})$ correspond bijectively to the \textcolor{cyan}{cyan} terms in $G_e(\blam)$. The reader can verify that for all \textcolor{cyan}{cyan} $\bmu$, we have $d^{e}_{\bmu,\blam}(q)=d^{e+1}_{\bmu^{+\bbd},\blam^{+\bbd}}(q)$. For simplicity, we have omitted terms $\bmu$ that are not of the form $\bnu^{+\bbd}$ for some $\bnu$ appearing in $G_e(\blam)$.
\end{Example}

\bigskip
\noindent\textbf{Acknowledgments.}
We thank Andrew Mathas for suggesting this problem and for many helpful discussions, as well as for reading the first manuscript of this paper and providing many helpful comments. We are also grateful to Travis Scrimshaw for communications regarding the SageMath package \texttt{FockSpaces}, which enabled the computation of canonical bases. We thank Matthew Fayers and Ruslan Maksimau for helpful correspondence regarding their papers related to this work. This work was partially supported by the Australian Research Council Discovery Grant DP240101809.
\bibliographystyle{alpha}
\bibliography{reference}  
\newpage
\appendix
\section*{Appendix. Explicit Formulas for Canonical Bases}
\label{app:formulas}
\markboth{APPENDIX. EXPLICIT FORMULAS FOR CANONICAL BASES}{APPENDIX. EXPLICIT FORMULAS FOR CANONICAL BASES}

\noindent
\begin{tcolorbox}[
    colback=white, 
    colframe=black, 
    width=\textwidth, 
    boxrule=0.8pt, 
    sharp corners=all, 
    title=\textbf{Canonical Basis Expansion: $G_e(\boldsymbol{\lambda})$}
]
    \normalsize
    \begin{align*}
    G_e(\blam) &= \textcolor{cyan}{\lvert[6], [5,1^2]\rangle}
    &&+ q\,\textcolor{cyan}{\lvert[6], [3^2,1]\rangle} \\
    &\quad + q\,\textcolor{cyan}{\lvert[5,1], [5,1^2]\rangle}
    &&+ q^{2}\,\textcolor{cyan}{\lvert[5,1], [3^2,1]\rangle} \\
    &\quad + q\,\textcolor{cyan}{\lvert[4,1], [5,2,1]\rangle}
    &&+ q^{2}\,\textcolor{red}{\lvert[4,1], [5,1^3]\rangle} \\
    &\quad + q^{2}\,\textcolor{cyan}{\lvert[4,1], [4,3,1]\rangle}
    &&+ q^{3}\,\textcolor{red}{\lvert[4,1], [3^2,1^2]\rangle} \\
    &\quad + (q^{3}+q)\,\textcolor{cyan}{\lvert[4], [5,3,1]\rangle}
    &&+ q\,\lvert[3,2,1], [5,1^2]\rangle \\
    &\quad + q^{2}\,\lvert[3,2,1], [3^2,1]\rangle
    &&+ q^{2}\,\textcolor{cyan}{\lvert[3,2], [5,2,1]\rangle} \\
    &\quad + q^{3}\,\lvert[3,2], [5,1^3]\rangle
    &&+ q^{3}\,\textcolor{cyan}{\lvert[3,2], [4,3,1]\rangle} \\
    &\quad + q^{4}\,\lvert[3,2], [3^2,1^2]\rangle
    &&+ q\,\textcolor{red}{\lvert[3,1^2], [6,1^2]\rangle} \\
    &\quad + q^{2}\,\textcolor{red}{\lvert[3,1^2], [3^2,2]\rangle}
    &&+ q^{2}\,\textcolor{cyan}{\lvert[3,1], [7,1^2]\rangle} \\
    &\quad + (q^{3}+q)\,\textcolor{cyan}{\lvert[3,1], [6,2,1]\rangle}
    &&+ (q^{4}+q^{2})\,\lvert[3,1], [6,1^3]\rangle \\
    &\quad + q^{2}\,\textcolor{cyan}{\lvert[3,1], [5,2^2]\rangle}
    &&+ q^{3}\,\textcolor{red}{\lvert[3,1], [5,1^4]\rangle} \\
    &\quad + q^{2}\,\textcolor{cyan}{\lvert[3,1], [4^2,1]\rangle}
    &&+ 2q^{3}\,\textcolor{cyan}{\lvert[3,1], [4,3,2]\rangle} \\
    &\quad + q^{4}\,\textcolor{cyan}{\lvert[3,1], [3^3]\rangle}
    &&+ (q^{5}+q^{3})\,\lvert[3,1], [3^2,2,1]\rangle \\
    &\quad + q^{4}\,\textcolor{red}{\lvert[3,1], [3^2,1^3]\rangle}
    &&+ q\,\textcolor{cyan}{\lvert[3], [8,1^2]\rangle} \\
    &\quad + 2q^{2}\,\textcolor{cyan}{\lvert[3], [6,3,1]\rangle}
    &&+ q^{3}\,\textcolor{cyan}{\lvert[3], [5,4,1]\rangle} \\
    &\quad + (q^{4}+q^{2})\,\textcolor{cyan}{\lvert[3], [5,3,2]\rangle}
    &&+ q^{3}\,\textcolor{red}{\lvert[3], [3^2,2^2]\rangle} \\
    &\quad + q^{2}\,\textcolor{red}{\lvert[2^3], [5,1^2]\rangle}
    &&+ q^{3}\,\textcolor{red}{\lvert[2^3], [3^2,1]\rangle} \\
    &\quad + (q^{4}+q^{2})\,\textcolor{cyan}{\lvert[2^2], [5,3,1]\rangle}
    &&+ q^{2}\,\textcolor{cyan}{\lvert[2,1], [8,1^2]\rangle} \\
    &\quad + (2q^{3}+q)\,\textcolor{cyan}{\lvert[2,1], [6,3,1]\rangle}
    &&+ (q^{4}+q^{2})\,\textcolor{cyan}{\lvert[2,1], [5,4,1]\rangle} \\
    &\quad + (q^{5}+2q^{3}+q)\,\textcolor{cyan}{\lvert[2,1], [5,3,2]\rangle}
    &&+ q^{2}\,\textcolor{red}{\lvert[2,1], [5,1^5]\rangle} \\
    &\quad + (q^{4}+q^{2})\,\lvert[2,1], [3^2,2^2]\rangle
    &&+ q^{3}\,\textcolor{red}{\lvert[2,1], [3^2,1^4]\rangle} \\
    &\quad + q^{2}\,\textcolor{red}{\lvert[1^3], [6,3,1]\rangle}
    &&+ q^{3}\,\textcolor{red}{\lvert[1^3], [5,4,1]\rangle} \\
    &\quad + (q^{4}+q^{2})\,\textcolor{red}{\lvert[1^3], [5,3,2]\rangle}
    &&+ q^{3}\,\textcolor{red}{\lvert[1^3], [5,1^5]\rangle} \\
    &\quad + q^{3}\,\textcolor{red}{\lvert[1^3], [3^2,2^2]\rangle}
    &&+ q^{4}\,\textcolor{red}{\lvert[1^3], [3^2,1^4]\rangle} \\
    &\quad + q^{2}\,\textcolor{cyan}{\lvert[1^2], [8,2,1]\rangle}
    &&+ q^{3}\,\textcolor{red}{\lvert[1^2], [8,1^3]\rangle} \\
    &\quad + 2q^{3}\,\textcolor{cyan}{\lvert[1^2], [7,3,1]\rangle}
    &&+ (2q^{4}+q^{2})\,\lvert[1^2], [6,3,1^2]\rangle \\
    &\quad + q^{4}\,\textcolor{cyan}{\lvert[1^2], [5^2,1]\rangle}
    &&+ (q^{5}+q^{3})\,\lvert[1^2], [5,4,1^2]\rangle \\
    &\quad + (q^{5}+q^{3})\,\textcolor{cyan}{\lvert[1^2], [5,3^2]\rangle}
    &&+ (q^{6}+2q^{4}+q^{2})\,\lvert[1^2], [5,3,2,1]\rangle \\
    &\quad + (q^{5}+q^{3})\,\textcolor{red}{\lvert[1^2], [5,3,1^3]\rangle}
    &&+ q^{3}\,\lvert[1^2], [5,2^3]\rangle \\
    &\quad + q^{4}\,\textcolor{red}{\lvert[1^2], [5,2,1^4]\rangle}
    &&+ 2q^{4}\,\lvert[1^2], [4,3,2^2]\rangle \\
    &\quad + q^{5}\,\textcolor{red}{\lvert[1^2], [4,3,1^4]\rangle}
    &&+ q^{5}\,\lvert[1^2], [3^3,2]\rangle \\
    &\quad + (q^{4}+q^{2})\,\textcolor{cyan}{\lvert[1], [8,3,1]\rangle}
    &&+ (q^{5}+q^{3})\,\lvert[1], [5,3,2^2]\rangle \\
    &\quad + q^{2}\,\textcolor{cyan}{\lvert\emptyset, [11,1^2]\rangle}
    &&+ q^{3}\,\textcolor{cyan}{\lvert\emptyset, [9,3,1]\rangle} \\
    &\quad + q^{3}\,\textcolor{cyan}{\lvert\emptyset, [8,3,2]\rangle}
    &&+ q^{4}\,\textcolor{red}{\lvert\emptyset, [6,3,2^2]\rangle} \\
    &\quad + q^{4}\,\textcolor{red}{\lvert\emptyset, [5,3,2^2,1]\rangle}
    &&+ q^{3}\,\textcolor{red}{\lvert\emptyset, [4,3,2^2,1^2]\rangle} \\
    &\quad + q^{4}\,\textcolor{red}{\lvert\emptyset, [3^3,2,1^2]\rangle}
    &&+ q^{5}\,\textcolor{red}{\lvert\emptyset, [3^2,2^3,1]\rangle}
    \end{align*}
\end{tcolorbox}
\captionof{figure}{Expansion of $G_3\big((6),\,(5,1,1)\big)$.}
\label{fig:Ge_expansion}

\clearpage 

\noindent
\begin{tcolorbox}[
    colback=white, 
    colframe=black, 
    width=\textwidth, 
    boxrule=0.8pt, 
    sharp corners=all, 
    title=\textbf{Canonical Basis Expansion: $G_{e+1}(\blam^+)$}
]
    \normalsize
    \begin{align*}
    G_{e+1}(\blam^+) &= \textcolor{cyan}{\lvert[8], [7, 1^2]\rangle}
    &&+ q\,\textcolor{cyan}{\lvert[8], [4^2, 1]\rangle} \\
    &\quad + q\,\textcolor{cyan}{\lvert[7, 1], [7, 1^2]\rangle}
    &&+ q^{2}\, \textcolor{cyan}{\lvert[7, 1], [4^2, 1]\rangle} \\
    &\quad + q\,\textcolor{cyan}{\lvert[5, 1], [7, 3, 1]\rangle}
    &&+ q^{2}\,\textcolor{cyan}{\lvert[5, 1], [6, 4, 1]\rangle} \\
    &\quad + (q^{3}+q)\,\textcolor{cyan}{\lvert[5], [7, 4, 1]\rangle}
    &&+ q\,\lvert[4, 3, 1], [7, 1^2]\rangle \\
    &\quad + q^{2}\,\lvert[4, 3, 1], [4^2, 1]\rangle
    &&+ q^{2}\,\textcolor{cyan}{\lvert[4, 2], [7, 3, 1]\rangle} \\
    &\quad + q^{2}\,\lvert[4, 2], [7, 1^4]\rangle
    &&+ q^{3}\,\textcolor{cyan}{\lvert[4, 2], [6, 4, 1]\rangle} \\
    &\quad + q^{3}\,\lvert[4, 2], [4^2, 1^3]\rangle
    &&+ q^{2}\,\textcolor{cyan}{\lvert[4, 1], [10, 1^2]\rangle} \\
    &\quad + (q^{3}+q)\,\textcolor{cyan}{\lvert[4, 1], [8, 3, 1]\rangle}
    &&+ q^{3}\,\lvert[4, 1], [8, 1^4]\rangle \\
    &\quad + q^{2}\,\textcolor{cyan}{\lvert[4, 1], [7, 3, 2]\rangle}
    &&+ q^{2}\,\textcolor{cyan}{\lvert[4, 1], [6, 5, 1]\rangle} \\
    &\quad + 2q^{3}\,\textcolor{cyan}{\lvert[4, 1], [6, 4, 2]\rangle}
    &&+ q^{4}\,\textcolor{cyan}{\lvert[4, 1], [4^3]\rangle} \\
    &\quad + q^{4}\,\lvert[4, 1], [4^2, 2, 1^2]\rangle
    &&+ q\,\textcolor{cyan}{\lvert[4], [11, 1^2]\rangle} \\
    &\quad + 2q^{2}\,\textcolor{cyan}{\lvert[4], [8, 4, 1]\rangle}
    &&+ q^{3}\,\textcolor{cyan}{\lvert[4], [7, 5, 1]\rangle} \\
    &\quad + (q^{4}+q^{2})\,\textcolor{cyan}{\lvert[4], [7, 4, 2]\rangle}
    &&+ (q^{4}+q^{2})\,\textcolor{cyan}{\lvert[3, 2], [7, 4, 1]\rangle} \\
    &\quad + q^{2}\,\textcolor{cyan}{\lvert[3, 1], [11, 1^2]\rangle}
    &&+ (2q^{3}+q)\,\textcolor{cyan}{\lvert[3, 1], [8, 4, 1]\rangle} \\
    &\quad + (q^{4}+q^{2})\,\textcolor{cyan}{\lvert[3, 1], [7, 5, 1]\rangle}
    &&+ (q^{5}+2q^{3}+q)\,\textcolor{cyan}{\lvert[3, 1], [7, 4, 2]\rangle} \\
    &\quad + q^{3}\,\lvert[3, 1], [4^2, 2^2, 1]\rangle
    &&+ q^{2}\,\textcolor{cyan}{\lvert[1^2], [11, 3, 1]\rangle} \\
    &\quad + 2q^{3}\,\textcolor{cyan}{\lvert[1^2], [10, 4, 1]\rangle}
    &&+ q^{3}\,\lvert[1^2], [8, 4, 1^3]\rangle \\
    &\quad + q^{4}\,\textcolor{cyan}{\lvert[1^2], [7^2, 1]\rangle}
    &&+ q^{4}\,\lvert[1^2], [7, 5, 1^3]\rangle \\
    &\quad + (q^{5}+q^{3})\,\textcolor{cyan}{\lvert[1^2], [7, 4^2]\rangle}
    &&+ (q^{5}+q^{3})\,\lvert[1^2], [7, 4, 2, 1^2]\rangle \\
    &\quad + q^{4}\,\lvert[1^2], [7, 3, 2^2, 1]\rangle
    &&+ (q^{5}+q^{3})\,\lvert[1^2], [6, 4, 2^2, 1]\rangle \\
    &\quad + q^{4}\,\lvert[1^2], [4^3, 2, 1]\rangle
    &&+ (q^{4}+q^{2})\,\textcolor{cyan}{\lvert[1], [11, 4, 1]\rangle} \\
    &\quad + q^{4}\,\lvert[1], [7, 4, 2^2, 1]\rangle
    &&+ q^{2}\,\textcolor{cyan}{\lvert\emptyset, [15, 1^2]\rangle} \\
    &\quad + q^{3}\,\textcolor{cyan}{\lvert\emptyset, [12, 4, 1]\rangle}
    &&+ q^{3}\,\textcolor{cyan}{\lvert\emptyset, [11, 4, 2]\rangle} + \dots
    \end{align*}
\end{tcolorbox}
\captionof{figure}{Expansion of $G_4\big((8),\,(7,1,1)\big)$.}
\label{fig:Ge_plus_1_expansion}

\clearpage
\end{document}